\documentclass[12pt,english,reqno]{amsart}
\usepackage{amscd,amssymb,amsmath,amstext,amsthm,amsfonts,bbold,wasysym}
\usepackage[dvips]{graphicx}
\usepackage{lipsum}
\usepackage{mathrsfs}
\usepackage{appendix}
\usepackage[normalem]{ulem}
\usepackage{float}
\usepackage{hyperref}
\usepackage[dvipsnames]{xcolor}
\usepackage[ansinew]{inputenc}
\usepackage{enumerate}

\usepackage{a4wide}
\newcommand\supsetsim{\mathrel{
  \ooalign{\raise0.2ex\hbox{$\supset$}\cr\hidewidth\raise-0.8ex\hbox{\scalebox{0.9}{$\sim$}}\hidewidth\cr}}}

\newtheorem{Theorem}{Theorem}[subsection]
\newtheorem*{theorem}{Theorem}

\newtheorem{maintheorem}{Theorem}

\newtheorem{T}{Theorem}[subsection]
\newtheorem{Corollary}[T]{Corollary}
\newtheorem{Proposition}[T]{Proposition}

\newtheorem{Lemma}[T]{Lemma}

\newtheorem{Definition}[T]{Definition}
\newtheorem{Example}[T]{Example}
\newtheorem*{claim}{Claim}
\newtheorem{Claim}[T]{Claim}

\def \AA {{\mathbb A}}

\def \dd {{\mathbb d}}
\def \RR {{\mathbb R}}
\def \ZZ {{\mathbb Z}}
\def \NN {{\mathbb N}}

\def \TT {{\mathbb T}}

\def \KK {{\mathbb K}}
\def \EE {{\mathbb E}}

\def \XX {{\mathbb X}}
\def \QQ {{\mathbb Q}}

\def \YY {{\mathbb Y}}

\def \cf {\mathcal{F}}
\def \cm {\mathcal{M}}

\def \cd {\mathcal{D}}
\def \cq {\mathcal{Q}}
\def \cg {\mathcal{G}}

\def \cc {\mathcal{C}}

\def \cu {\mathcal{U}}
\def \co {\mathcal{O}}
\def \cn {\mathcal{N}}

\def \ca {\mathcal{A}}

\def \cx {\mathcal{X}}

\def \cR {\mathscr{R}}

\newcommand{\leb}{\operatorname{Leb}}
\newcommand{\dist}{\operatorname{dist}}

\newcommand{\Agemo}{\operatorname{\mho}}
\newcommand{\supp}{\operatorname{supp}}
\newcommand{\im}{\operatorname{Im}}
\newcommand{\per}{\operatorname{Per}}

\newcommand{\interior}{\operatorname{interior}}

\newcommand{\diam}{\operatorname{diameter}}

\begin{document}

\author{V. Pinheiro}
\address{Instituto de Matematica - UFBA, Av. Milton Santos, s/n,
40170-110 Salvador Bahia}
\email{viltonj@ufba.br}

\thanks{Work carried out at National Institute for Pure and Applied Mathematics (IMPA) and Federal University of Bahia (UFBA). Partially supported by CNPq-Brazil (PQ 313272/2020-4)}
\date{\today}

\title{Ergodic Formalism for topological Attractors and historic behavior}

\maketitle
\begin{abstract}
We introduce the concepts of Baire Ergodicity and Ergodic Formalism, employing them to study topological and statistical attractors. Specifically, we establish the existence and finiteness of such attractors and provide applications for maps of the interval, Viana maps, non-uniformly expanding maps, partially hyperbolic systems, strongly transitive dynamics, and skew-products.

In a dynamical system with an abundance of historic behavior (encompassing all systems with some hyperbolicity, particularly Axiom A systems), one can show the existence of a residual set with zero measure for every invariant probability measure.
Hence, in principle, utilizing the classical ergodic theory to control the asymptotic topological/statistical behavior of generic orbits is not feasible.

Nevertheless, the results presented here can also be applied to such a system, contributing to the study of generic orbits in systems with an abundance of historic behavior.
\end{abstract}
\tableofcontents

\section{Introduction}

As it is wildly known, the asymptotic behavior of the forward orbit of a point $x\in X$ with respect to map $f:X\circlearrowleft$ is, in general, quite complex and strongly dependent on $x$.
To understand the behavior of the orbit of $x$ we should focus on the $\omega$-limit of $x$, denoted by $\omega_f(x)$ and defined as the set of accumulating points of the sequence $\{f^n(x)\}_{n\ge0}$.
The existence of an attractor $A\subset X$ can simplify dramatically the study of the $\omega$-limit sets, as a large proportion of points $x\in X$ are attracted to this attractor and in many cases $\omega_f(x)=A$ for most of the attracted points.
We say that a ``large proportion of points'' belongs to the {\em basin of attraction of $A$}, denoted by  $\beta_f(A)$, if $\beta_f(A)$ is not a zero measure set or a meager set.  A set is called {\em meager} if it is a countable union of nowhere dense subsets of $X$.
In this paper we focus on topological attractors, that is, when $\beta_f(A)$ is not a meager set.
In particular, this implies that we can consider a more general context than the context of metrical attractors (when $\beta_f(A)$ has Lebesgue positive measure), as it is not necessary to be restricted to finite-dimensional spaces.

A metric attractor, specially one that supports a physical  measure, has all the tools of  Ergodic Theory to study the statistical  behavior, such as Birkhoff averages, of almost all points in its basin of attraction.
In general, this is not the case for a topological attractor.
Indeed, generic points in the basin of attraction of a topological attractor exhibit  {\em historic behavior}, meaning that the convergence of the Birkhoff average is not expected, even for continuous functions (for more details, see Section~\ref{SectionMWAHB}). 
Here, we introduce the concept of Baire ergodicity (and variations) with two main objectives: (1) to study statistical properties of generic points (including Birkhoff averages) even for points with historic behavior, and (2) to study the existence and finiteness of topological attractors. 

Emphasizing again the importance of attractors, it was conjectured by Palis in 1995 (see \cite{Pa00,Pa2005}) that, in a compact smooth manifold, there is a dense set $D$ of differentiable dynamics such that, among other properties, any element of $D$ display finitely many (metrical) attractors whose union of basins of attraction has total probability measure in the ambient manifold.
This conjecture, known as the ``Palis Global Conjecture'' was built in such a way that, if proved to be true, one can then concentrate the attention on the description of the properties of these finitely many attractors and their basins of attraction to have an understanding on the whole system.
In a finite-dimensional space, we observe the existence of a strong connection between topological and metrical attractors (for instance, see Theorem~\ref{mainThojhgf} and Section~\ref{Connections}).
Hence, the problem of existence and finiteness of topological attractors is also strongly related with Palis conjecture.

\section{Statement of mains results}\label{SecStatOfMainsR}

Let $\XX$ be a compact metric space with distance $d$.
In this context, a set $U\subset\XX$ is {\bf\em meager} if it is contained in a countable union of compact sets with empty interior.
We say that two sets $U,V\subset\XX$ are (meager) {\bf\em  equivalent}, denoted by $U\sim V$, when $U\triangle V:=(U\setminus V)\cup(V\setminus U)$ is a meager set $($\footnote{ Knowing that $A\triangle A=\emptyset$, $A\triangle A=B\triangle A$, and  $A\triangle C\subset (A\triangle B)\cup(B\triangle C)$, we get that $\sim$ is an equivalence relation.
}$)$.
We say that a given property $P$ is {\bf\em generic on $U\subset\XX$}, or {\bf\em $P$ holds for a residual set of points $x\in U$}, when $\{x\in U\,;\,P$ is not valid for $x\}$ is a meager set.
A map $f:\XX\circlearrowleft$ is called {\bf\em non-singular} if the pre-image of a meager set is also a meager set.

A continuous map $f:\XX\circlearrowleft$ is called {\bf\em transitive} if $\bigcup_{n\ge0}f^n(V)$ is dense in $\XX$ for every open set $V\subset\XX$.
If $\bigcup_{n\ge0}f^n(V)=\XX$ for every open set $V\subset\XX$, then $f$ is called {\bf\em strongly transitive}.
Given $U\subset\XX$ and $x\in\XX$, the {\bf\em frequency of visits} of $x$ to $U$ is defined as 
$$\tau_x(U)=\limsup_{n\to+\infty}\frac{1}{n}\#\left\{0\le j<n\,;\,f^j(x)\in U\right\}.$$

Our first result, Theorem~\ref{mainTheoTrans} below, shows that transitivity homogenizes the asymptotic behavior of Birkhoff averages for generic points, not only for continuous functions, but for all bounded measurable ones.

\begin{maintheorem}\label{mainTheoTrans}
Let $f:\XX\circlearrowleft$ be a non-singular continuous map.
If $f$ is transitive then, given a  Borel measurable bounded function $\varphi:\XX\to\RR$, there exists $\gamma\in\RR$  such that $$\limsup_{n\to+\infty}\frac{1}{n}\sum_{j=0}^{n-1}\varphi\circ f^{j}(x)=\gamma$$ for a residual set of points $x\in\XX$.
As a consequence, for each Borel set $U\subset\XX$, there exists $\theta\in[0,1]$  such that $\tau_x(U)=\theta$ for a residual set of points $x\in\XX$.
\end{maintheorem}

The {\bf\em basin of attraction} of a set $A\subset\XX$ is defined as $$\beta_f(A)=\{x\in\XX\,;\,\omega_f(x)\subset A\},$$
where $\omega_f(x)=\bigcap_{n\ge0}\overline{\co_f^+(f^n(x))}$ is the {\bf\em omega limit} of $x$ and $\co_f^+(x)=\{f^n(x)\,;\,n\ge0\}$ is the {\bf\em forward orbit} of a point $x\in\XX$.
Following Milnor's definition of topological attractor \cite{Mi}, a compact set $A$
is called  {\bf\em a topological attractor} if
$\beta_f({A})$ and $\beta_f(A)\setminus\beta_f(A')$ are not  meager sets for every compact set $A'\subsetneqq A$.

The {\bf\em support} of a Borel probability measure $\mu$, $\supp\mu$, is the set of all points $x\in\XX$ such that $\mu(B_{\varepsilon}(x))>0$ for every $\varepsilon>0$, where $B_{\varepsilon}(x)=\{y\in\XX\,;\,d(y,x)<\varepsilon\}$ is the {\bf\em open ball} of center $x$ and radius $\varepsilon$.

In Theorem~\ref{Theoremjhvuvpb} below, we use a measure-theoretic criterion to decompose the space into the union of the basins of attraction of a finite number of topological attractors.

\begin{maintheorem}\label{Theoremjhvuvpb}
Let $f:\XX\circlearrowleft$ be a non-singular continuous map and $\mu$ a Borel probability measure on $\XX$ with $\supp\mu=\XX$.
If $$\inf\left\{\mu\left(\bigcup_{n\ge0}\interior\big(f^n(B_{\varepsilon}(x))\big)\right)\,;\,x\in\XX\text{ and }\varepsilon>\right\}>0$$
 then there exists a finite collection of topological attractors $A_1,\cdots, A_\ell$ such that $$\beta_f(A_j)\cup\cdots\cup\beta_f(A_{\ell})\sim\XX.$$
Furthermore, the following statements are true for every $1\le j\le \ell$.
\begin{enumerate}
\item \label{itemugvigv11} $\omega_f(x)=A_j$ for a residual set of points $x\in\beta_f(A_j)$.
\item \label{itemugvigv22} If $\varphi:\XX\to\RR$ is a Borel measurable bounded function then there exists $\gamma_j\in\RR$ such that $$\limsup_{n\to+\infty}\frac{1}{n}\sum_{j=0}^{n-1}\varphi\circ f^j(x)=\gamma_j$$ for a residual set of points $x\in\beta_f(A_j)$.
\item \label{itemugvigv33} If $U$ is a Borel subset of $\XX$ then there exists $\theta_j\in[0,1]$ such that $$\tau_x(U)=\theta_j$$
for a residual set of points $x\in\beta_f(A_j)$.
\end{enumerate}
\end{maintheorem}

The {\bf\em nonwandering set}, $\Omega(f)$, of a map $f:\XX\circlearrowleft$ is the  set of all $x\in \XX$ such that $V\cap\bigcup_{n\ge1}f^n(V)\ne\emptyset$ for every open set $V$ containing $x$.
Denote the set of all periodic points of $f$ by $\per(f)$, i.e., $\per(f)=\{p\in\XX\,;\,p\in\co_f^+(p)\}\subset\Omega(f)$.
The map $f$ has {\bf\em sensitive dependence on initial condition} \cite{Gu} if there exists $r>0$ such that $\sup_{n}\diam(f^n(B_{\varepsilon}(x)))\ge r$ for every $x\in\XX$ and $\varepsilon>0$.
According to Ruelle \cite{Ru} and Takens \cite{Ta08}, a point $x\in\XX$ has {\bf\em historic behavior} when $$\limsup_{n\to+\infty}\frac{1}{n}\sum_{j=0}^{n-1}\varphi\circ f^j(x)>\liminf_{n\to+\infty}\frac{1}{n}\sum_{j=0}^{n-1}\varphi\circ f^j(x)$$ for some $\varphi\in C^0(\XX,\RR)$.

Now, using topological criteria, Theorem~\ref{TheoremFatErgodicAttractorsMAIN} ensures the decomposition of the space into the union of the basins of attraction of a finite number of topological attractors, allowing us to control many of the asymptotic aspects of generic points.

\begin{maintheorem}
\label{TheoremFatErgodicAttractorsMAIN}Let $f:\XX\circlearrowleft$ be a non-singular continuous map.
If there exists $\delta>0$ such that $\overline{\bigcup_{n\ge0}f^n(U)}$ contains some open ball of radius $\delta$, for every nonempty open set $U\subset\XX$, then there exists a finite collection of topological attractors $A_1,\cdots, A_\ell$ such that 
$$\beta_f(A_j)\cup\cdots\cup\beta_f(A_{\ell})\text{ contains an open and dense subset of }\XX,$$
and the following statements are true for every $1\le j\le \ell$.
\begin{enumerate}[\hspace{0.6cm}(i)]
\item The statements (\ref{itemugvigv11}),  (\ref{itemugvigv22}) and (\ref{itemugvigv33}) of Theorem~\ref{Theoremjhvuvpb}  remain valid.
\item  $f|_{A_j}$ is transitive.
\item $A_j=\overline{\interior(A_j)}$ and it contains an open ball of radius $\delta$. 
\item $\overline{\Omega(f)\setminus\bigcup_{j=0}^{\ell}A_j}$ is a compact set with empty interior.
\end{enumerate}
Furthermore, if $\bigcup_{n\ge0}f^n(U)$ contains some open ball of radius $\delta$, for every nonempty open set $U\subset\XX$, then the following statements are true.
\begin{enumerate}[\hspace{0.6cm}(i)]
\item[(v)]  For each $A_j$ there is a forward invariant set $\ca_j\subset A_j$, with $A_j=\overline{\interior(\ca_j)}$, such that $f|_{\ca_j}$ is strongly transitive.
\item[(vi)]  Either $\omega_f(x)=A_j$ for every $x\in\ca_j$ or $f\big|_{A_j}$ has sensitive dependence on initial conditions.
\item[(vii)]\label{itenkgujjjfycv} If $\ca_j$ contains more that one periodic orbit, then generically, the points of $\beta_f(A_j)$ have historic behavior. 
\end{enumerate}
\end{maintheorem}

Denote the set of all $f$ invariant Borel probability measures of $f$ by $\cm^1(f)$. Recall that $f$ is called {\bf\em uniquely ergodic} when $f$ has one and only one $f$ invariant Borel probability measure.
The {\bf\em statistical omega-limit} of a point $x\in\XX$, denoted by $\omega_f^{\star}(x)$, is the set of all points $y\in\XX$ such that $\tau_x(V)>0$ for every open set $V$ containing $y$.
If we are considering a metrical attractor $A$ supporting a SRB measure $\mu$ or, more in general, a physical measure, then one can expect  that $\omega_f(x)=A$ and $\omega_f^{\star}(x)=\supp\mu$ for almost every point in the basin of attraction of $A$. 
For instance, there are well known examples {\color{black} of $C^{\infty}$} circle maps such that $\omega_f(x)=S^1=\RR/\ZZ$  and $\omega_f^{\star}(x)=\{[0]\}$
for Lebesgue almost every $x\in S^1$, where $[0]$ is the fixed point of $f$ and $f'([0])=1$ (see Figure~\ref{Renom.png}).
In this case, $\mu:=\delta_{[0]}$ is the physical measure for $f$.
According to Ilyashenko \cite{Ily}, while $S^1$ is the attractor for $f$, $A^*:=\{[0]\}$ is its {\em statistical attractor} (see Section~\ref{SectionSTATAT} for precise definitions).

\begin{figure}
\begin{center}\includegraphics[scale=.3]{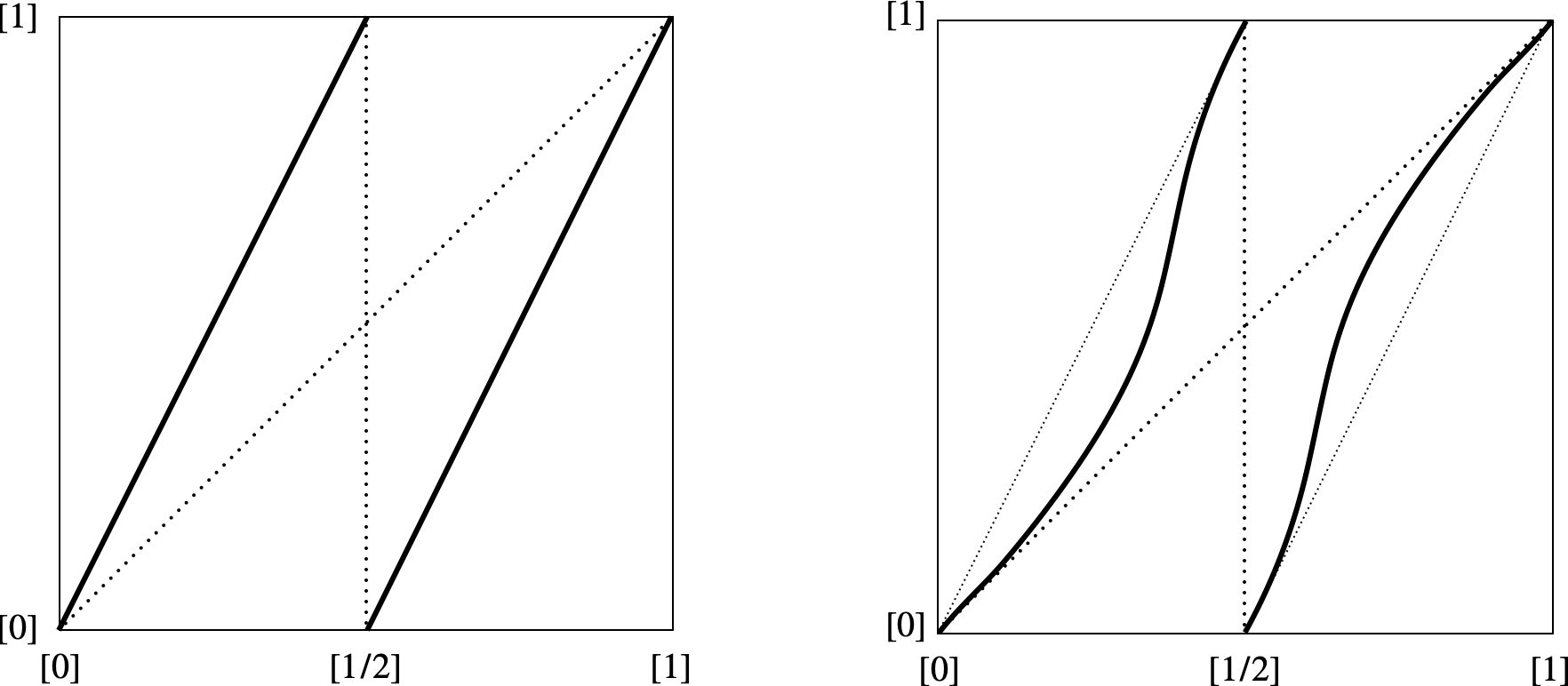}
\caption{\small The picture on the left shows the graph of the double map $S^1\ni[x]\mapsto[2x]\in S^1$.
On the right side, the picture shows the graph of a  $C^{\infty}$ map $f:S^1\circlearrowleft$, $S^1:=\RR/\ZZ$, that is conjugated to the double map and such that $\omega_f(x)=S^1$ and $\omega_f^{\star}([x])=\{[0]\}$ for Lebesgue almost every $[x]\in S^1$.}\label{Renom.png}
\end{center}
\end{figure}

Theorem~\ref{TheoremUEorHB} below shows that the hypothesis of strong transitivity allows us to estimate the (upper) Birkhoff averages for continuous functions at generic points.

\begin{maintheorem}
\label{TheoremUEorHB}
If a continuous map $f:\XX\circlearrowleft$ is strongly transitive then the following statements are true. 
\begin{enumerate}
\item \label{itemhgn} Either $f$ is uniquely ergodic or, generically, the points of $\XX$ have historic behavior.
\item \label{itemdrt} For any continuous function  $\varphi:\XX\to\RR$, 
$$
  \limsup_{n\to+\infty}\frac{1}{n}\sum_{j=0}^{n-1}\varphi\circ f^j(x)=\max\bigg\{\int\varphi d\mu\,;\,\mu\in\cm^1(f)\bigg\}
$$ for a residual set of points $x\in\XX$.
\item \label{itemvgjmu} If $U\subset\XX$ is an open set then $$\sup\left\{\mu\left(\overline{U}\right)\,;\,\mu\in\cm^1(f)\right\}\ge\tau_x(U)\ge
\sup\left\{\mu\left(U\right)\,;\,\mu\in\cm^1(f)\right\}$$
for a residual set of points $x\in\XX$.
\item \label{itemvgg668mu} If $f$ is non-singular and $V\subset\XX$ is a Borel set then 
$$\sup\left\{\mu\left(\overline{U}\right)\,;\,\mu\in\cm^1(f)\right\}\ge\tau_x(V)\ge
\sup\left\{\mu\left(U\right)\,;\,\mu\in\cm^1(f)\right\}$$
for a residual set of points $x\in\XX$, where $U$ is any open set such that $U\sim V$.
\item \label{itemngyj46uu} $\omega_f^{\star}(x)=\overline{\bigcup_{\mu\in\cm^1(f)}\supp\mu}$ for a residual set of points $x\in\XX$.
\end{enumerate}
\end{maintheorem}

Although Theorem~\ref{TheoremUEorHB} above can be applied to some injective maps (for instance, transitive translations of a compact metrizable topological group), most of its applications are for endomorphisms. Because of that, we present below (Theorem~\ref{TheoremUEorHBinject}) a version of Theorem~\ref{TheoremUEorHB} better adapted to injective maps.
For that, given a map $f:\XX\circlearrowleft$, define the {\bf\em $f$-stable set} of $x\in\XX$ as $W_f^s(x)=\left\{y\in\XX\,;\,\lim_{j\to+\infty} d(f^j(x),f^j(y))=0\right\},$
and the {\bf\em pre-orbit} of a set $U\subset\XX$ as $\co_f^-(U)=\bigcup_{n\ge0}f^{-n}(U)$.
Since $f$ being strongly transitive means that $\overline{\co_f^-(\{x\})}=M$ for every $x\in\XX$, using stable sets we can weaken strong transitivity in the following way.
A continuous map $f:\XX\circlearrowleft$ is called {\bf\em strongly $u$-transitive} when $\overline{\co_f^-(W_f^s(x))}=\XX$ for every $x\in\XX$.
Of course that all strongly transitive maps are strongly $u$-transitive, since $x\in W_f^s(x)$.
On the other hand, a ``linear'' Anosov diffeomorphism \cite{Fr} and a non-transitive circle homeomorphism with irrational rotation number \cite{De} are examples of strongly $u$-transitive maps that are not strongly transitive.

\begin{maintheorem}
\label{TheoremUEorHBinject}
	Let $f:\XX\circlearrowleft$ be a continuous map.
If $f$  strongly $u$-transitive (in particular, if $\overline{W_f^s(x)}=\XX$ for all $x$) then all the enumerated statements of  Theorem~\ref{TheoremUEorHB} hold.
\end{maintheorem}

A {\em growing map} is a topological generalization of the non-uniformly expanding maps, but in a very weak way.
A non-singular continuous map $f:\XX\circlearrowleft$ is called a {\bf\em growing map} if there exists $\delta>0$ such that for each nonempty open set $V\subset\XX$ one can find $n\ge0$, $q\in\XX$ and a connected component $U\subset V$ of $f^{-n}(B_{\delta}(q))$ such that $f^n(U)=B_{\delta}(q)$.
We note that being a growing map does not depend on the metric, only on the topology.
That is, if $d_1$ and $d_2$ are two metrics inducing the same topology on $\XX$, then $f$ is a growing map with respect to $d_1$ if and only if it is a growing map with respect to $d_2$.
In particular, the property of being a growing map is preserved by topological conjugations.  

\begin{maintheorem}
\label{TheoremFatErgodicAttractors2MAIN}
If $f:\XX\circlearrowleft$ is a growing map then there exists a finite collection of topological attractors $A_1,\cdots, A_\ell$ such that 
$$\beta_f(A_j)\cup\cdots\cup\beta_f(A_{\ell})\text{ contains an open and dense subset of }\XX,$$
and following statements are true for every $1\le j\le \ell$.
\begin{enumerate}
\item The statements (i),(ii),(iii) and (iv) of Theorem~\ref{TheoremFatErgodicAttractorsMAIN} remain valid.
\item $\omega_f^{\star}(x)=A_j$ for a residual set of points $x\in\beta_f(A_j)$.
\item $h_{top}(f|_{A_j})>0$, i.e., the topological entropy of $f$ restrict to $A_j$ is positive.
\item There is a strongly transitive and forward invariant set $\ca_j\subset A_j=\overline{\interior(\ca_j)}$.
\item $f|_{\ca_j}$ has an uncountable set of ergodic invariant probability measures.
\item If $\varphi\in C(\XX,\RR)$ then  there exist constants $\gamma_+$ and $\gamma_-\in\RR$ such that 
$$
    \limsup_{n\to+\infty}\frac{1}{n}\sum_{j=0}^{n-1}\varphi\circ f^j(x)=\gamma_+\ge\sup\bigg\{\int\varphi d\mu\,;\,\mu\in\cm^1(f|_{\ca_j})\bigg\}\ge
$$
$$\ge\inf\bigg\{\int\varphi d\mu\,;\,\mu\in\cm^1(f|_{\ca_j})\bigg\}\ge\gamma_-=\liminf_{n\to+\infty}\frac{1}{n}\sum_{j=0}^{n-1}\varphi\circ f^j(x)$$ for a residual set of points $x\in\beta_f(A_j)$.
\end{enumerate}
Furthermore,
\begin{enumerate}
\setcounter{enumi}{6}
\item $f$ has sensitive dependence on initial conditions
\item Generically, the points of $\XX$ have historic behavior.
\item If $\XX$ is a compact manifold (possibly with boundary) then $\overline{\per(f)}\supset\bigcup_{j}A_j$.
\end{enumerate}

\end{maintheorem}

Theorem~\ref{mainThojhgf} below relates the support of physical measures with the topological attractors.
If $f:M\circlearrowleft$ is a homeomorphism on a Riemannian manifold $M$, an ergodic $f$-invariant probability measure $\mu$ is called a {\bf\em  physical measure} when its basin of attraction has positive Lebesgue measure.
The {\bf\em basin of attraction} of a measure $\mu\in\cm^1(M)$, denoted by $\beta_f(\mu)$, is the set of all $x\in M$ such that $\frac{1}{n}\sum_{j=0}^{n-1}\delta_{f^j(x)}$ converges to $\mu$ in the weak$^{\star}$ topology, see Section~\ref{Connections} for more details and related  results. 
 Given  $U\subset M$, define $W_f^s(U)=\bigcup_{x\in U}W_f^s(x)$. 

\begin{maintheorem}\label{mainThojhgf}
Let $M$ be a compact Riemannian manifold and $f:M\circlearrowleft$ a homeomorphism such that $W_f^s=\{W_f^s(x)\}_{x\in M}$ is a continuous foliation of $M$ $($\footnote{\label{footfolia}
 See Section~\ref{SectionCOntF} in Appendix.}$)$.
If there exists $\varepsilon>0$ such that $\leb(W_f^s(\bigcup_{n\ge0}f^n(U)))$ $\ge$ $\varepsilon$ for every nonempty open set $U\subset M$, then there exists a finite collection of topological attractors $A_1,\cdots,A_k$, with $1\le k\le\leb(M)/\varepsilon$,  such that
$$\beta_f(A_j)\cup\cdots\cup\beta_f(A_{k})\sim M$$
and $\omega_f(x)=A_j$ for a residual set of points $x\in\beta_f(A_j)$ and every $1\le j\le k$.

Furthermore, if $\mu$ is a physical measure for $f$ then $\supp\mu\subset A_j$ for some $1\le j\le k$ or $\beta_f(\mu)$ is a nowhere dense set.
\end{maintheorem}

We have chosen to present the main results in a less technical and more unified form.
Nevertheless, we observe that Theorem~\ref{mainTheoTrans} is also true for a large class of non compact spaces and unbounded function, see Theorem~\ref{Theoremkviduro}.
Moreover, Theorem~\ref{TheoremFatErgodicAttractorsMAIN} and Theorem~\ref{TheoremFatErgodicAttractors2MAIN} above are simplified (and less technical) versions of Theorem~\ref{TheoremFatErgodicAttractors} and \ref{TheoremFatErgodicAttractors2}. Indeed, Theorem~\ref{TheoremFatErgodicAttractors} and \ref{TheoremFatErgodicAttractors2} can be applied to maps with discontinuity when the closure of the set of all discontinuities has empty interior.
Furthermore, the results of Section~\ref{Connections}, used to prove Theorem~\ref{mainThojhgf}, show others connections between metrical and topological attractors.

We would like to thank Minkov,  Okunev, and Shilin, who noted in \cite{MOS} that in the earlier versions of this paper, the hypothesis of Theorem A in those older versions (which is included as part of Theorem C in the current version) was incomplete. They suggested adding the hypothesis that the map in question should be an open one and provided several interesting counterexamples for non-open maps in \cite{MOS}.
Indeed, the hypothesis was incomplete, and the natural hypothesis was the non-singularity of the map, which was already widely used in many results in the earlier versions of the paper. This hypothesis is much less restrictive than the hypothesis of the map being open.

\subsection{Organization of the text}
Section~\ref{SectionBaireErgodicity} is dedicated to the {\em Ergodic Formalism}, which comprises results analogous to those valid in the context of ergodic invariant probability measures (see, for instance, Proposition~\ref{Propositionkjuytf9tfg76}).
In this section we introduce the notion of {\em Baire ergodicity} and study its relation with transitivity and asymptotic transitivity.

In Section~\ref{SectionAttractors}, we relate Baire and $u$-Baire ergodicity with topological and statistical attractors.
In this section we provide some criteria for the existence of a finite Baire (or $u$-Baire) ergodic decomposition. 

In the last section, Section~\ref{SectionAplications}, we apply the results of the previous two sections to several examples of dynamical systems. Additionally, we prove all the theorems stated above.

\section{Topological $\times$ Baire ergodicity}\label{SectionBaireErgodicity}

A {\bf\em Baire space} $X$ is a topological space with the property that the intersection of any given countable collection of open dense sets is a dense set.
It is known that all complete metric spaces and all locally compact Hausdorff spaces are Baire spaces.
As commented before, a countable union of nowhere dense subsets of $X$ is said to be {\bf\em meager}; the complement of such a set is called a {\bf \em residual set} and it contains a countable intersection of open and dense sets.
If a set is not meager then it is called {\bf \em fat} $($\footnote{ The meager and fat sets also are called, respectively, first and second category sets.}$)$.
A set $V$ {\bf\em is residual in a set $U$} when $U\setminus V$ is a meager set.

A subset $A\subset X$ of a topological space $X$ is said to have the {\bf\em  Baire property} if there is an open set $U$ such that $A\triangle U$ is a meager set, i.e., $A\sim U$.
A set with the Baire property is also called {\bf  \em an almost open set}.
Hence, a set $A\subset X$ with the Baire property is fat if and only if $A\sim U$ for some nonempty  open set $U\subset X$.

\begin{Proposition}[Prop. 8.22, pp. 47 of \cite{KeBook}]\label{PropBorelBaire}
Let $X$ be a topological space.
	The class of subsets of $X$  having the Baire property is a $\sigma$-algebra on $X$. Indeed, it is the smallest $\sigma$-algebra containing all open sets and all meager sets. In particular, this $\sigma$-algebra contains the Borel $\sigma$-algebra.
\end{Proposition}

When every element of a $\sigma$-algebra $\mathfrak{A}$ on a topological space $X$ has the Baire property, we say that $\mathfrak{A}$ {\bf\em  has the Baire property}.

A set $V\subset{X}$ is called {\bf \em invariant} if $f^{-1}(V)=V$ and it is called {\bf\em almost invariant} when $f^{-1}(V)\triangle V$ is meager.
If $f(V)\subset V$, then $V$ is called {\bf \em  forward invariant}.
Analogously, $V$ is called {\bf\em almost forward invariant} if $f(V)\setminus V$ is meager.

The natural way to define a {\bf\em topologically ergodic} map $f:X\circlearrowleft$ is that every invariant set is meager or residual.
A basic example of a topologically ergodic map on a compact space is a periodic orbit.
That is, a map $f:X\circlearrowleft$, where $X=\{p_1,\cdots,p_n\}$, $f(p_1)=p_2,\cdots,f(p_{n-1})=p_n$ and $f(p_n)=p_1$.
Although the definition above is  perfectly consistent, it follows from Proposition~\ref{Propositionjghjk4} below that essentially only singular maps can be topologically ergodic, with the exception, as in the example above, of spaces that have isolated points.

As observed in Section~\ref{SecStatOfMainsR}, the map $f$ is called {\bf\em topologically non-singular} or, for short, {\bf\em non-singular}, if the pre-image of a meager set is also a meager set.
The concept of non-singular maps is inspired by non-singular measure, that is, a measure on a space $X$ is  {\em $f$-non-singular} when $\mu(A)=0\implies\mu(f^{-1}(A))=0$ for every measurable set $A\subset X$. Note the all $f$-invariant measure are non-singular measures.
A non-singular measure, even if it is not invariant, has many ergodic properties (see Section~3 of \cite{Pi11}).
Similarly, non-singular continuous maps have many interesting topological (see Appendix) and ergodic properties (see, for instance, Theorem~\ref{TheoremProposi0ytd6881}).

If we are considering a metric space $(\YY,d)$, the {\bf\em open ball of radius $r\ge0$ and center $p\in\YY$} is given by $$B_r(p)=\{x\in\YY\,;\,d(x,p)<r\}.$$ Note that $B_0(p)=\emptyset$, $\forall\,p\in\YY$.

\begin{Proposition}\label{Propositionjghjk4}
A complete separable metric space $X$ without isolated points does not admit a topologically ergodic non-singular map $f:X\circlearrowleft$.
\end{Proposition}
\begin{proof} Let $f:X\circlearrowleft$ be a non-singular map defined on a complete separable metric space $X$ without isolated points.
Given $x\in X$ let the {\bf\em total orbit} of $x$ be defined as $$\co_f(x)=\{y\in X\,;\, f^n(y)=f^m(x)\text{ for some }n,m\ge0\}.$$

Let $\cu$ be the collection of all the total orbit of points of $X$.
Using the Axiom of Choice select for each $O\in\cu$ a single point $x_{_O}\in O$.
Let $A=\{x_{_O}\,;\,O\in\cu\}$.
Let $A_0=\bigcup_{j\ge0}f^j(A)$.
As $X=\bigcup_{n\ge0}f^{-n}(A_0)$ and $f$ is non-singular, $A_0$ must be a fat set.
Hence, there exists $m\ge0$ such that $f^m(A)$ is a fat set.
It can be seen that $f^m|_{A}$ is a bijection of $A$ with $f^m(A)$.

Let $X_0=\{x\in X\,;\,B_r(x)\cap f^m(A)\text{ is not a meager set for every }r>0\}$. 
As $X$ is separable, $X_0\ne\emptyset$. Otherwise, for each $x\in X$ there is $r_x>0$ such that $\,B_r(x)\cap f^m(A)$ is a meager set.
Choosing any countable subcover $\{B_{r_{x_n}}(x_n)\}_{n\in\NN}$ of $\{B_{r_{x}}(x)\}_{x\in X}$, we conclude that $f^m(A)=f^m(A)\cap\bigcup_{n\in\NN}B_{r_{x_n}}(x_n)$ is meager, a contradiction.

Let $p\in X_0$. As $X$ does not have isolated points we have that $f^m(A)\setminus\{p\}=\bigcup_{n\in\NN}\big(f^m(A)\setminus B_{\frac{1}{n}}(p)\big)$ is not a meager set.
Hence, there exists $r>0$ ($r=1/n$ for some $n\in\NN$) such that $f^m(A)\setminus B_{r}(p)$ is not a meager set.

Let $P=\bigcup_{n\ge0}f^{-n}(\bigcup_{j\ge0}f^j(f^m(A)\cap B_r(p)))$ and $Q=\bigcup_{n\ge0}f^{-n}(\bigcup_{j\ge0}f^j(f^m(A)\setminus B_r(p)))$.
 As $P$ and $Q$ are $f$-invariant fat sets and $P\cap Q=\emptyset$, we conclude that $f$ is not topologically ergodic.
\end{proof}

As a consequence of Proposition~\ref{Propositionjghjk4}, even an irrational rotation on the circle cannot be topologically ergodic.
Therefore, we weaken the definition of ergodicity by considering only measurable invariant sets having the Baire property. 

\begin{Definition}[Baire ergodic maps]
	Let $X$ be a Baire space and $\mathfrak{A}$ a $\sigma$-algebra on $X$ with the Baire property.
A $\mathfrak{A}$-measurable map $f:X\circlearrowleft$ is called {\bf \em Baire ergodic} if every invariant set $U\in\mathfrak{A}$ is either meager or residual.
\end{Definition}

Ergodicity and transitivity are notions related with the idea of a dynamical system being  indecomposable.
Therefore, it is not surprising that these two concepts are connected.
The parallel between transitivity and ergodicity was pointed out as early as the 1930s by J. C. Oxtoby \cite{Ox}, a few years after the Ergodic Theorem appeared.
Indeed, applying the ``Zero-one topological law'' to the group $(\{f^n\}_{n\in\ZZ},\circ)$, we can conclude that every transitive homeomorphism $f:X\circlearrowleft$ on a Baire space $X$ is Baire ergodic. 

\begin{Lemma}[Zero-one topological law, see Prop. 8.46, pp. 55 of \cite{KeBook}, see also \cite{GK}]
Let $X$ be a Baire space and $G$ a group of homeomorphism of $X$. Suppose that $X$ is $G$-transitive, that is, given a pair of open sets $A,B\subset X$, there is a $g\in G$ such that $g(A)\cap B\ne\emptyset$.
Let $U\subset X$ be a $G$-invariant set, i.e.,  $g(U)=U$ for every $g\in G$.
If $U$ has the Baire property then either $U$ or $X\setminus U$ is meager. 
\end{Lemma}

\begin{Corollary}
\label{HomeoErgEquivTras}
 If $f:X\circlearrowleft$ is a homeomorphism defined on a Baire space $X$ then $f$ is Baire ergodic if and only if $f$ is transitive.
\end{Corollary}

In contrast with the topological ergodicity, it follows from  Corollary~\ref{HomeoErgEquivTras}  that all irrational rotation on the circle are Baire ergodic maps. As we can see below, continuity and transitivity is enough to ensure ergodicity for non-singular maps. 

\begin{Lemma}
	\label{Lemma98uyhji}
Let $X$ is a Baire space and consider the Borel $\sigma$-algebra on $X$.
If a non-singular map $f:X\circlearrowleft$ is continuous and transitive then $f$ is Baire ergodic. 
\end{Lemma}
\begin{proof}
Let $V\subset X$ be a fat invariant Borel set.
Let $A\subset X$ be an open set such that $V\sim A$, that is, $V\triangle A$ is a meager set.
Since $f^{-n}(V)\triangle f^{-n}(A)=f^{-n}(V\triangle A)$ and $f$ is non-singular, we also get that $f^{-n}(V)\sim f^{-n}(A)$ $\forall n\ge0$.
As $f$ is continuous and transitive $\bigcup_{j\ge0}f^{-j}(A)$ is open and dense in $X$, i.e., $\bigcup_{j\ge0}f^{-j}(A)\sim X$.
Thus, $V=\bigcup_{j\ge0}f^{-j}(V)\sim \bigcup_{j\ge0}f^{-j}(A)\sim X$, proving that $V$ is a residual set of $X$.
\end{proof}

Despite the connection given by Corollary~\ref{HomeoErgEquivTras} and Lemma~\ref{Lemma98uyhji}, in general, an ergodic map can be far from being transitive and a trivial example of such a map is a constant one, i.e., $f:X\circlearrowleft$, with $\# X>1$, such that $f(x)=p$ for some $p\in X$.
However, since constant maps are singular maps, one might ask whether there are non-singular Baire ergodic maps that are not transitive.
The answer again is yes, as we can see in Example~\ref{Unimodal} below. 
Indeed, we need to relax the definition of transitivity to obtain a concept that is closer to ergodicity. 

\begin{Definition}[Asymptotically transitive maps]Let $X$ be a topological space.
A map $f:X\circlearrowleft$ is called {\bf\em asymptotically transitive} if 
$$
\bigg(\bigcup_{j\ge0}f^j(A)\bigg) \cap \bigg(\bigcup_{j\ge0}f^j(B)\bigg)\,\text{ is a fat set}$$ for every nonempty open sets $A$ and $B\subset X$.
\end{Definition}

\begin{Theorem}
\label{TheoremProposi0ytd6881}
Let $\XX$ be a Baire metric space $\XX$ and consider the Borel $\sigma$-algebra on $\XX$.
If $f:\XX\circlearrowleft$ is a non-singular continuous map, then $f$ is Baire ergodic if and only if $f$ is asymptotically transitive.
\end{Theorem}
As some preliminary results are required, we leave the proof of Theorem~\ref{TheoremProposi0ytd6881} above for the end of Section~\ref{SectioNonSing}.
Theorem~\ref{TheoremProposi0ytd6881} can be used to provide  examples of non trivial maps that are Baire ergodic but not transitive. 

\begin{Example}[A non-singular Baire ergodic map that is not transitive]\label{Unimodal}
The maps of the Logistic family $\{f_t\}_{0<t\le1}$, where $f_t(x)=4 t x(1-x)$ (Figure~\ref{Logistic.png}), are classical examples of non-flat $S$-unimodal maps.
\begin{figure}
\begin{center}\includegraphics[scale=.08]{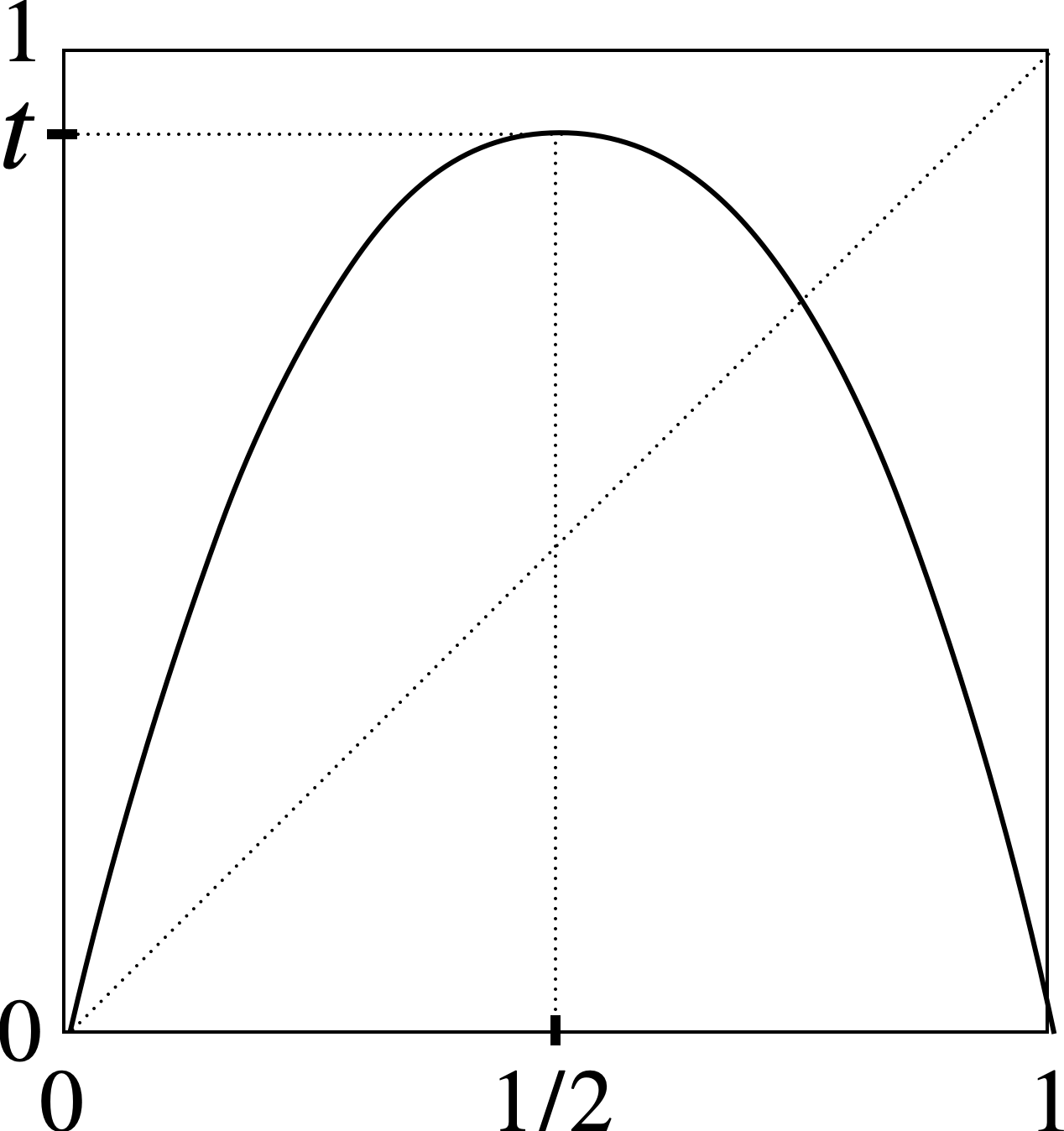}
\caption{The Logistic map $f_t$ with parameter $0<t\le1$.}\label{Logistic.png}
\end{center}
\end{figure}
By \cite{Gu} (see also \cite{MS89}), a non-flat $S$-unimodal map does not have  wandering intervals, that is, a strongly wandering domain for interval maps $($\footnote{An open set $A\subset X$ is called a {\bf\em strongly wandering domain} for a map $f$ if $f^n|_A$ is a homeomorphism of $A$ with $f^n(A)$, for every $n\ge1$, and $f^n(A)\cap f^m(A)=\emptyset$ for every $m>n\ge0$.
If $X$ is a Baire space that is perfect and Hausdorff, then the existence of a strongly wandering domain is an obstruction for a map to be asymptotically transitive and so, Baire ergodic. 
Indeed, as $X$ is perfect and a Hausdorff space, every open set $A$ contains open disjoint subsets $A_0$ and $A_1$.
Thus, if $A$ is a strongly wandering domain, then $\big(\bigcup_{n\ge0}f^n(A_0)\big)\cap \big(\bigcup_{n\ge0}f^n(A_1)\big)=\emptyset$, proving that $f$ is not asymptotically transitive. 
}$)$.
When $f_t$ is a $\infty$-renormalizable map (see for instance \cite{MvS} for the definition), $\interior(\omega_{f_t}(x))=\emptyset$ $\forall\,x\in[0,1]$ and so, $f_t$ cannot be transitive, since transitivity implies the existence of dense orbits $($\footnote{ The equivalence between transitivity and the existence of dense forward orbits is well known (see for instance Proposition~11.4 of \cite{ManeLivro}).}$)$.
Nevertheless being $\infty$-renormalizable implies that $\bigcup_{n\ge0}{f_t}^n(A)$ always contains an open neighborhood of the point $1/2$ for every nonempty open set $A\subset[0,1]$.
Thus, one can use this fact and Theorem~\ref{TheoremProposi0ytd6881} to conclude that $f_t$ is Baire ergodic.
\end{Example}

Proposition~\ref{Propositionkjuytf9tfg76} below provides an example of a {\em Ergodic Formalism result}, that is, a result that has a metric analog for ergodic maps with respect to invariant probability measures.
Indeed, one can find in most introductory books of Ergodic Theory a version, for invariant probability measures, of Proposition~\ref{Propositionkjuytf9tfg76} below (see for instance Proposition~2.1 of \cite{ManeLivro}, Proposition~4.1.3 of \cite{OV} or Theorem~1.6 of \cite{Wa}).
To state and prove Proposition~\ref{Propositionkjuytf9tfg76}, we need to introduce some definitions and notations.

 Let $X$ be a Baire space and $\mathfrak{A}$ a $\sigma$-algebra on $X$ with the Baire property.
A {\bf\em Baire potential} on $X$ is a measurable map defined on $X$ and assuming values on a complete separable metric space, i.e., $\varphi$ is a Baire potential on $X$ if $\varphi:(X,\mathfrak{A})\to(\YY,\mathfrak{B})$ is a measurable map for some complete separable metric space $\YY$ and $\mathfrak{B}$ is the Borel $\sigma$-algebra on $\YY$.
Define the {\bf\em image-support of $\varphi$} as $$\im\supp\varphi=\{y\in\YY\,;\,\varphi^{-1}(B_{\varepsilon}(y))\text{ is not meager for every }\varepsilon>0\}.$$

\begin{Lemma}\label{Lemmakytf7634erbv}
$\im\supp\varphi\ne\emptyset$ for every Baire potential $\varphi$ defined on a Baire space $X$.
\end{Lemma}
\begin{proof}
If $\im\supp\varphi=\emptyset$ then, for every $y\in\YY$ there exists $r_y>0$ such that $\varphi^{-1}(B_{r_y}(y))$ is meager.
As $\YY$ is a separable metric space and $\bigcup_{y\in\YY}B_{r_y}(y)=\YY$, there exists a countable set $C=\{y_1,y_2,y_3,\cdots\}$ such that $\bigcup_{n\in\NN}B_{r_{y_n}}(y_n)=\YY$.
Thus, $X=\varphi^{-1}(\YY)=\bigcup_{n\in\NN}\varphi^{-1}(B_{r_{y_n}}(y_n))$ which is a contradiction as $\bigcup_{n\in\NN}\varphi^{-1}(B_{r_{y_n}}(y_n))$ is a meager set (\footnote{ Here we are using the fact that $X$ is a Baire space. Otherwise $X$ itself can be a meager set.
For instance, the rational numbers $\QQ=\bigcup_{q\in\QQ}\{q\}$ is a meager metric space with the usual distance.}).
\end{proof}

We say that $\varphi:X\to Y$ is  {\bf\em almost invariant} (with respect to $f:X\circlearrowleft$) if there exists an invariant residual set $U\in\mathfrak{A}$ such that $\varphi(x)=\varphi\circ f(x)$ for every $x\in U$.
Similarly, $\varphi$ is {\bf\em almost constant} if there exist  $y_0\in Y$ an invariant residual set $U\in\mathfrak{A}$ such that $\varphi(x)=y_0$ for every $x\in U$.

Given a set $A\subset X$, let $\mathbb{1}_A$ be the characteristic function of $A$, that is, $$\mathbb{1}_A(x)=
\begin{cases}
1 & \text{ if }x\in A\\
0 & \text{ if }x\notin A
\end{cases}.$$

\begin{Proposition}\label{Propositionkjuytf9tfg76}
Let $X$ be a Baire space and $\mathfrak{A}$ a $\sigma$-algebra on $X$ with the Baire property.  If $f:X\circlearrowleft$ is a measurable map then the following statements are equivalent: 
\begin{enumerate}
\item $f$ is Baire ergodic.
\item Every almost invariant Baire potential on $X$ is almost constant.
\item ${X}\ni x\mapsto\limsup_{n}\frac{1}{n}\sum_{j=0}^{n-1}\varphi\circ f^{j}(x)$  is almost constant for every measurable function $\varphi:X\to\RR$ such that $\limsup_{n}\frac{1}{n}\sum_{j=0}^{n-1}\varphi\circ f^{j}(x)\in\RR$ for every $x\in X$.
\item $X\ni x\mapsto\limsup_{n\to+\infty}\frac{1}{n}\#\{0\le j<n\,;\,f^{j}(x)\in A\}$ is almost constant \,$\forall\,A\in\mathfrak{A}$.
\end{enumerate}
\end{Proposition}
\begin{proof}
(1)$\implies$(2). Suppose that $f$ is Baire ergodic, $\YY$ a complete separable metric space and $\varphi:X\to\YY$ an almost invariant measurable map.
It follows from Lemma~\ref{Lemmakytf7634erbv} that $\im\supp\varphi\ne\emptyset$.

As $\varphi$ is an almost invariant potential, let $U\in\mathfrak{A}$ be a $f$ invariant residual set  such that $\varphi\circ f(x)=\varphi(x)$ for every $x\in U$.
Choose any $p\in\im\supp\varphi$.
Set $B_n:=\varphi^{-1}(B_{1/n}(p))$, $\triangle_n=\varphi^{-1}(\YY\setminus B_{1/n}(p))$, $B_n'=B_n\cap U$ and $\triangle_n'=\triangle_n\cap U$, where $n\in\NN$.
Note that $B_n'$ and $\triangle_n'$ are $f$-invariant sets and $B_n'\cap\triangle_n'=\emptyset$ for every $n\in\NN$.
As $p\in\im\supp\varphi$, $B_n'$ is a fat set and so, it follows from the ergodicity of $f$ that $\triangle_n'$ is meager for every $n\in\NN$.
Hence, $U\setminus\varphi^{-1}(p)=U\cap\varphi^{-1}(\YY\setminus\{p\})=\bigcup_{n\in\NN}\triangle_n'$ is a meager set, proving that $U\cap\varphi^{-1}(p)$ is a residual set.
That is, $\varphi(x)=p$ for a residual set of points $x\in X$.

(2)$\implies$(3) Let $\psi:X\to\RR$ be given by $\psi(x)=\limsup_{n\to\infty}\frac{1}{n}\sum_{j=0}^{n-1}\varphi\circ f^j(x)$.
The measurability of $\psi$ follows from the measurability of $\varphi$.
Thus, (3) follows from (2) and the fact that $\psi\circ f(x)=\psi(x)$ for every $x\in X$.
Indeed, $$\psi(f(x))=\limsup_{n\to+\infty}\frac{1}{n}\sum_{j=0}^{n-1}\varphi\circ f^{j}(f(x))=$$
$$=\limsup_{n\to+\infty}\bigg(\underbrace{\frac{n+1}{n}}_{\hspace{0.5cm}\to 1}\bigg(\frac{1}{n+1}\sum_{j=0}^{n}\varphi\circ f^{j}(x)\bigg)-\underbrace{\frac{1}{n}\varphi(x)}_{\hspace{0.5cm}\to0}\bigg)=\psi(x).$$

(3)$\implies$(4) Noting that
$\mathbb{1}_A$
 is a measurable map and $\frac{1}{n}\#\{0\le j<n\,;\,f^j(x)\in A\}=\frac{1}{n}\sum_{j=0}^{n-1}\mathbb{1}_A\circ f^j(x)$, we have (4) as a direct consequence of (3).

(4)$\implies$(1) Let $A\in\mathfrak{A}$ be such that $f^{-1}(A)=A$.
Set $\psi(x)=\limsup_n\frac{1}{n}\#\{0\le j<n\,;\,f^j(x)\in A\}$. As $A$ is $f$-invariant,
$$\psi(x)=
\begin{cases}
1 & \text{ if }x\in A\\
0 & \text{ if }x\notin A
\end{cases}.$$
It follows from (4) that there exists a residual set $U\in\mathfrak{A}$ such that either $\psi(x)=1$ for every $x\in U$ or $\psi(x)=0$ for every $x\in U$.
The first case implies that $U\subset A$ and so $A$ is residual. The second case implies that $U\subset X\setminus A$ and so, $A$ is meager.
Thus, every measurable invariant set $A$ is either residual or meager, proving (1). 
\end{proof}

\subsection{Ergodicity for non-singular maps}\label{SectioNonSing}
In this section (Section~\ref{SectioNonSing}), let  $X$ be a Baire space, $\mathfrak{A}$ a $\sigma$-algebra on $X$ with the Baire property.

\begin{Lemma}\label{Lemmagdigi00fi75r8}Let $Y\in\mathfrak{A}$ be a residual subset of $X$ and $f:Y\to X$ a non-singular measurable map.
If $U\subset{Y}$ is a fat almost invariant measurable set then $$U':=\bigcup_{n\ge0}f^{-n}\left(\bigcap_{j\ge0}f^{-j}(U)\right)$$ is a fat invariant measurable set with $U'\sim U$ .
\end{Lemma}
\begin{proof}
As $f$ is non-singular, $f^{-n}(U\triangle f^{-1}(U))$ is a meager set $\forall\,n\ge0$.
This implies that  $U\triangle f^{-j}(U)\subset (U\triangle f^{-1}(U))\cup\cdots\cup(f^{-(j-1)}(U)\triangle f^{-j}(U))=\bigcup_{n=0}^{j-1}f^{-n}(U\triangle f^{-1}(U))$ is also a meager set.
That is, $U\sim f^{-j}(U)$ for every $j\ge0$ and, as a consequence,  $U\sim U_0:=\bigcap_{j\ge0}f^{-j}(U)\subset U$.
So, $U_0$ is a fat measurable set and $f^{-1}(U_0)$ $=$ $f^{-1}\big(\bigcap_{j\ge0}f^{-j}(U)\big)$ $=$ $\bigcap_{j\ge1}f^{-j}(U)$ $\supset$ $U_0$.
Hence, $U'=\bigcup_{j\ge0}f^{-j}(U_0)$ is a fat measurable set and, as  $U_0\cup f^{-1}(U_0)=f^{-1}(U_0)$, we get that $$f^{-1}(U')=f^{-1}\bigg(\bigcup_{j\ge0}f^{-j}(U_0)\bigg)=\bigcup_{j\ge1}f^{-j}(U_0)=\bigcup_{j\ge0}f^{-j}(U_0)=U'.$$
Since $U_0\triangle U=\big(\bigcap_{j\ge0}f^{-j}(U)\big)\triangle U\subset\bigcap_{j\ge0} f^{-j}(U)\triangle U\subset f^{-1}(U)\triangle U$, we have that $U_0\triangle U$ is a meager set.
Moreover, as $f^{-j}(U)\triangle f^{-j}(U_0)= f^{-j}(U\triangle U_0)$ is meager, we have that $f^{-j}(U_0)\sim f^{-j}(U)\sim U$ for every $j\ge0$.
As a consequence, $U'\triangle U\subset\bigcup_{j\ge0}(f^{-j}(U_0)\triangle U)$ is a meager set, that is, $U'\sim U$.
\end{proof}

\begin{Corollary}\label{Corollaryiouregh}
A non-singular measurable map $f:X\circlearrowleft$ is Baire ergodic if and only if every almost invariant measurable set is either meager or residual.
\end{Corollary}
\begin{proof}
As an invariant set is an almost invariant one, we need only to show that if $f$ is Baire ergodic then every almost invariant measurable set is either meager or residual.
Suppose that $U\sim f^{-1}(U)$ is a fat measurable set.
It follows Lemma~\ref{Lemmagdigi00fi75r8} above that  $U'=\bigcup_{n\ge0}f^{-n}\big(\bigcap_{j\ge0}f^{-j}(U)\big)$ is a fat invariant measurable set with $U'\sim U$.
Thus, by the Baire ergodicity, $U\sim U'\sim{X}$, proving that $U$ is a residual set. 
\end{proof}

For non-singular maps, we can use Corollary~\ref{CorCompactErg} below to characterize Baire ergodicity in terms of open or closed invariant sets. 

\begin{Corollary}
\label{CorCompactErg}
If $f:{X}\circlearrowleft$  is a non-singular measurable map then the following statements are equivalent.
\begin{enumerate}
\item $f$ is Baire ergodic.
\item Every almost invariant nonempty open set is dense in ${X}$.
\item ${X}$ is the unique closed almost invariant set without empty interior.
\end{enumerate}
\end{Corollary}
\begin{proof} Since {\em (2)}$\iff${\em (3)} and, by Corollary~\ref{Corollaryiouregh}, {\em (1)}$\implies${\em (2)}, we need only to show that {\em (2)}$\implies${\em (1)}.
For that, suppose that $U$ is a measurable invariant fat set.
Let $A$ be an open set meager equivalent to $U$, i.e., $A\sim U$.
Since $f$ is non-singular, $A\sim U$ $\implies$ $f^{-1}(A)\sim f^{-1}(U)\sim U\sim A$, proving that $A$ is an almost invariant nonempty open set. 
Thus, it follows from {\em (2)} that $A$ is and dense in ${X}$.
As a consequence $U\sim A\sim{X}$, proving that $f$ is Baire ergodic.
\end{proof}

\begin{proof}[\bf Proof of Theorem~\ref{TheoremProposi0ytd6881}]
Suppose that $f$ is Baire ergodic and consider two nonempty open sets  $A,B\subset\XX$.
As $\AA:=\bigcup_{j\in\ZZ}f^j(A)$ and is an invariant subset of $\XX$ and, by Lemma~\ref{LemmaDoInterior} in Appendix, $$\AA\sim\AA_0:=\bigcup_{j\in\ZZ}\interior(f^j(A))\supset A,$$ we have that $\AA_0$ is an almost invariant open set.
Indeed, as $f$ is non-singular, $\AA\sim\AA_0$ $\implies$ $f^{-1}(\AA)\sim f^{-1}(\AA_0)$ and so, $f^{-1}(\AA_0)\sim f^{-1}(\AA)=\AA\sim\AA_0$.
Thus, by Corollary~\ref{CorCompactErg}, $\AA_0$ is an open and dense set, proving that $V:=\interior(f^n(A))\cap B$ is a nonempty open set for some $n\in\ZZ$.

Let $m\ge0$ be so that $m+n\ge 0$.
By Lemma~\ref{LemmaDoInterior} in Appendix, $\interior f^m(V)\ne\emptyset$ and so,
$$\left(\bigcup_{j\ge0}f^j(A)\right)\cap\left(\bigcup_{j\ge0}f^j(B)\right)\supset \interior(f^m(V))\ne\emptyset,$$
proving that $f$ is asymptotically transitive.

Now, assume that $f$ is asymptotically transitive.
If $f$ is not Baire ergodic, there is a fat invariant Borel set $V$ such $X\setminus V$ is also fat.
Let $A,B\subset\XX$ be open sets such that $V\sim$ is residual in $A$ and $\XX\setminus V$ is residual in $B$.

It follows from Corollary~\ref{CorLemResnaoSingsing} in Appendix (and the invariance of $V$ and $\XX\setminus V$) that $V$ is residual in $f^{j}(A)$  and  $\XX\setminus V$ is residual in $f^j(B)$ for every $j\ge0$.
Since $f$ is asymptotically transitive, $W=\big(\bigcup_{j\ge0}f^j(A)\big)\cap\big(\bigcup_{j\ge0}f^j(B)\big)$ is a fat Borel set and both $V$ and $\XX\setminus V$ are residual in $W$.
This is a contradiction, as this would imply that  $V\cap(\XX\setminus V)\ne\emptyset$. 
\end{proof}

\subsection{$u$-Baire ergodicity}
As in Section~\ref{SecStatOfMainsR}, in a metric space $(\XX,d)$, one can define the {\bf\em stable set of a point} $x\in\XX$ with respect to a map $f:\XX\circlearrowleft$ as $$W_f^s(x)=\left\{y\in\XX\,;\,\lim_{n\to+\infty}d(f^n(x),f^n(y))=0\right\}$$
and the {\bf\em stable set of a set} $U\subset\XX$ as $W_f^s(U)=\bigcup_{x\in U}W_f^s(x)$.

From the classical theory of Uniformly Hyperbolic Dynamical Systems, given a $C^1$ Anosov diffeomorphism $f:M\circlearrowleft$ defined on a compact manifold, the tangent space at each $x\in M$ splits into two complementary directions $T_xM=\EE^s\oplus\EE^u$ such that the derivative contracts on  the ``stable'' direction $\EE^s$ and expands on the ``unstable'' direction $\EE^u$ , at uniform rates. 
Moreover, for every $x\in M$, $W_f^s(x)$ is, locally (\footnote{ For every $p\in W_f^s(x)$ and a small enough $\varepsilon>0$, the connected component $N$ of $W_f^s(x)\cap B_{\varepsilon}(p)$ containing $p$ is a submanifold with $T_pN=\EE^s(p)$.}), a sub-manifold of $M$ tangent to $\EE^s$. 
As the asymptotical behavior of the points in $W_f^s(x)$ are the same, we may restrict, in the definition of ergodicity, to invariant set that are equal to its stable set.
This reduces the collection of allowed invariant sets producing a weaker definition of ergodicity called $u$-ergodicity (\footnote{ This name comes from the fact that belonging to a stable set is an equivalence relation ($x\sim y$ if  $x\in W_f^s(y)$), and from the idea that by grouping the points on stable manifolds, we are making a kind of quotient by $\sim$ and so (in a hyperbolic context) seeing only the unstable behavior of the dynamics.}).
The concept of $u$-ergodicity was introduced by Alves, Dias, Pinheiro and Luzzatto \cite{ADLP} for non (necessarily) invariant measures, in this section we adapted it to Baire ergodicity.

In the remaining of this section, $\XX$ is a Baire metric space and $f:\XX\circlearrowleft$ is a measurable map with respect to a $\sigma$-algebra $\mathfrak{A}$, where $\mathfrak{A}$ has the Baire property.

\begin{Definition}[$u$-Baire ergodicity]
We say that $f$ is {\bf\em $u$-Baire ergodic} if every set $U\in\mathfrak{A}$ satisfying $f^{-1}(U)= U= W_f^s(U)$ is meager or residual. \end{Definition}

Of course, every Baire ergodic map is $u$-Baire ergodic.
A simple example of a $u$-Baire but not Baire ergodic map is the contraction $f:\RR\circlearrowleft$ given by $f(x)=x/2$.
In this case $W_f^s(x)=\RR$ for every $x$, proving that $f$ is $u$-Baire ergodic.
Nevertheless,  $A=\bigcup_{n\in\ZZ}f^n((1/2,2/3))$ and $B=\bigcup_{n\in\ZZ}f^n((2/3,1))$ are two $f$ invariant nonempty open sets such that $A\cap B=\emptyset$, proving that $f$ is not Baire ergodic.

A {\bf\em $u$-Baire potential} $\varphi$ for $f:\XX\circlearrowleft$ is a  measurable map $\varphi$ of $(\XX,\mathfrak{A})$ to a measurable space $(\YY,\mathfrak{B})$, where $\YY$ is a  complete separable metric space $\YY$ and $\mathfrak{B}$ is the Borel $\sigma$-algebra on $\YY$, and such that
\begin{equation}\label{Equagftuft4}
  \varphi(x)=\varphi(y)\;\;\forall x\in\XX\text{ and }y\in W_f^s(x).
\end{equation}
That is, $\varphi$ is a $u$-Baire potential if $\varphi$ is a Baire potential satisfying \eqref{Equagftuft4}.

\begin{Proposition}\label{Propositionkjuiiiifg7655}
If $f$ is $u$-Baire ergodic then the following statements are true.
\begin{enumerate}
\item Every almost invariant  $u$-Baire potential on $\XX$ is almost constant.
\item Every continuous and almost  invariant function $\psi:\XX\to\RR$ is constant.
\item $\XX\ni x\mapsto\limsup_{n}\frac{1}{n}\sum_{j=0}^{n-1}\varphi\circ f^{j}(x)$  is almost constant for every  continuous function $\varphi:\XX\to\RR$ such that $\limsup_{n}\frac{1}{n}\sum_{j=0}^{n-1}\varphi\circ f^{j}(x)\in\RR$ for every $x\in\XX$.

\end{enumerate}
\end{Proposition}
\begin{proof}
Proof of item {\em (1)}.
Let $\YY$ be a complete separable metric space and $\varphi:\XX\to\YY$ an almost invariant $u$-Baire potential.
As $\varphi$ is an almost invariant, there exists a $f$ invariant residual set $U$ such that $\varphi\circ f(x)=\varphi(x)$ for every $x\in U$.
Choose any $p\in\im\supp\varphi$ (see Lemma~\ref{Lemmakytf7634erbv}).
Set $B_n:=\varphi^{-1}(B_{1/n}(p))$, $\triangle_n=\varphi^{-1}(\YY\setminus B_{1/n}(p))$, $B_n'=B_n\cap U$ and $\triangle_n'=\triangle_n\cap U$, where $n\in\NN$.
As $\varphi$ is a $u$-Baire potential, we get that $$B_n'\subset W_f^s(B_n')=W_f^s(\varphi^{-1}(B_{1/n}(p))\cap U)\subset W_f^s(\varphi^{-1}(B_{1/n}(p)))=\varphi^{-1}(B_{1/n}(p))=B_n$$ as well as 
$$\triangle_n'\subset W_f^s(\triangle_n')=W_f^s(\varphi^{-1}(\YY\setminus B_{1/n}(p))\cap U)\subset W_f^s(\varphi^{-1}(\YY\setminus B_{1/n}(p)))=\varphi^{-1}(\YY\setminus B_{1/n}(p))=\triangle_n.$$ 

Note that $B_n'$ and $\triangle_n'$ are $f$-invariant sets and $W_f^s(B_n')\cap W_f^s(\triangle_n')\subset B_n\cap\triangle_n=\emptyset$ for every $n\in\NN$.
As $p\in\im\supp\varphi$, $B_n'$ and also $W_f^s(B_n')$ are fat sets and so, it follows from the $u$-ergodicity of $f$ that $W_f^s(\triangle_n')$ is meager. This implies that $\triangle_n'$ is also meager for every $n\in\NN$.
Hence, $U\setminus\varphi^{-1}(p)=U\cap\varphi^{-1}(\YY\setminus\{p\})=\bigcup_{n\in\NN}\triangle_n'$ is a meager set, proving that $U\cap\varphi^{-1}(p)$ is a residual in $\XX$.
That is, $\varphi(x)=p$ for a residual set of points $x\in\XX$.

Proof of item {\em (2)}. Let $\psi:\XX\to\RR$  be a continuous and invariant function.
Let $V$ be a $f$-invariant residual set such that $\psi\circ f(x)=\psi(x)$ for every $x\in V$.

By Lemma~\ref{Lemmakytf7634erbv}, $\im\supp\psi|_V\ne\emptyset$ and so,  choose a point $p\in\im\supp\psi|_V\subset\RR$.
Given $\varepsilon>0$ let  $B_{\varepsilon}:=\psi|_V^{-1}((p-\varepsilon/4,p+\varepsilon/4))$ and $\triangle_n=\psi|_V^{-1}(\YY\setminus (p-\varepsilon,p+\varepsilon))$.
As $B_\varepsilon$ and $\triangle_\varepsilon$ are $f$ invariant sets, we get that $W_f^s(B_\varepsilon)$ and $W_f^s(\triangle_\varepsilon)$ are also $f$-invariant.

We claim that $W_f^s(B_\varepsilon)\cap W_f^s(\triangle_\varepsilon)\sim\emptyset$.
 Indeed, if $x\in V\cap W_f^s(B_\varepsilon)\cap W_f^s(\triangle_\varepsilon)$ then, let $\delta>0$ be such that $|\psi(a)-\psi(b)|<\varepsilon/4$ for every $a,b\in\XX$ with $d(a,b)<\delta$. 
As $d(f^j(x),f^j(B_{\varepsilon}))$ and $d(f^j(x),f^j(\triangle_{\varepsilon}))\to0$ there exists $\ell\ge0$ such that $d(f^j(x),f^j(B_{\varepsilon}))$ and $d(f^j(x),f^j(\triangle_{\varepsilon}))<\delta$ for every $j\ge\ell$.
 From $d(f^j(x),f^j(B_{\varepsilon}))<\delta$, for $j\ge\ell$, we get that $\psi(f^j(x))\in (p-\varepsilon/2,p+\varepsilon/2)$ $\forall\,j\ge\ell$.
On the other hand, from $d(f^j(x),f^j(\triangle_{\varepsilon}))<\delta$ for every $j\ge\ell$, we get that $\psi(f^j(x))\notin(p-\varepsilon,p+\varepsilon)$, a contradiction.

As $B_{\varepsilon}\subset W_f^s(B_{\varepsilon})$ and $B_{\varepsilon}$ is a fat set, we get from the $u$-ergodicity of $f$ that $W_f^s(B_{\varepsilon})$ is a residual set and, as $W_f^s(B_\varepsilon)\cap W_f^s(\triangle_\varepsilon)\sim\emptyset$, we get that $W_f^s(\triangle_\varepsilon)$ is a meager set for every $\varepsilon>0$. 
Hence, $V\setminus W_f^s(\psi^{-1}(p))\subset \bigcup_{n\ge1}W_f^s(\triangle_{1/n})$ is a meager set, proving that $V\cap W_f^s(\psi^{-1}(p))$ is a residual in $\XX$.
This implies that $d(f^j(x),f^j(\psi^{-1}(p)))\to0$ for a residual set of points $x\in\XX$.
That is, $\psi(x)=p$ for a residual set of points $x\in\XX$.
As $f|_V(\psi^{-1}(p))=\psi^{-1}(p)$, we have that $d(f^j(x),\psi^{-1}(p))\to0$ for a residual sets of points $x\in V$ (and so, for a residual sets of points $x\in\XX$).
As a consequence of the continuity of $\psi$, $|\psi(f^j(x))-p|\to0$ for a residual set of points $x$.
But, as $\psi$ is almost invariant, we conclude that $\psi(x)=p$ residually in $\XX$ and so, by continuity, $\psi(x)=p$ for every $x\in\XX$.

Proof of item {\em (3)}. Letting $\psi(x)=\limsup_{n\to+\infty}\frac{1}{n}\sum_{j=0}^{n-1}\varphi\circ f^{j}(x)$, we get that $\psi$ is measurable and $\psi(f(x))=\psi(x)$ for every $x\in\XX$.
As $\varphi$ is equicontinuous,  $\lim_{n\to+\infty}|\varphi\circ f^n(x)-\varphi\circ f^n(y)|\to0$ for every $y\in W_f^s(x)$ and so, $\psi(x)=\psi(W_f^s(x))$ for every $x\in\XX$, proving that $\psi$ is a almost invariant $u$-Baire potential.
Thus,  item {\em (3)} follows from item {\em (1)}.
\end{proof}

\section{Topological and statistical attractors}\label{SectionAttractors}
In many situations (for instance,  expanding/contracting Lorenz maps \cite{Br,GuWi,Me,Ro}), we have a dynamical system generated by a map $f$ that is continuous on the whole space $X$ except in a compact meager set $\cc$ (Figure~\ref{NonSingCont.png}). In this case, we can consider the continuous map $g:X_0\to X$ where $g:=f|_{X_0}$ and $X_0=X\setminus\cc$ is an open and dense subset of $X$.
Thus, due to the applications we want to obtain, we will assume for the  entire  Section~\ref{SectionAttractors} that $\XX$ is a compact metric space, $\mathfrak{A}$ is Borel $\sigma$-algebra of $\XX$ ($\mathfrak{A}$ has the Baire property by Proposition~\ref{PropBorelBaire}), $\XX_0$ is an open and dense subset of $\XX$ and $f:\XX_0\to\XX$ be a non-singular continuous map.
\begin{figure}
\begin{center}\includegraphics[scale=.25]{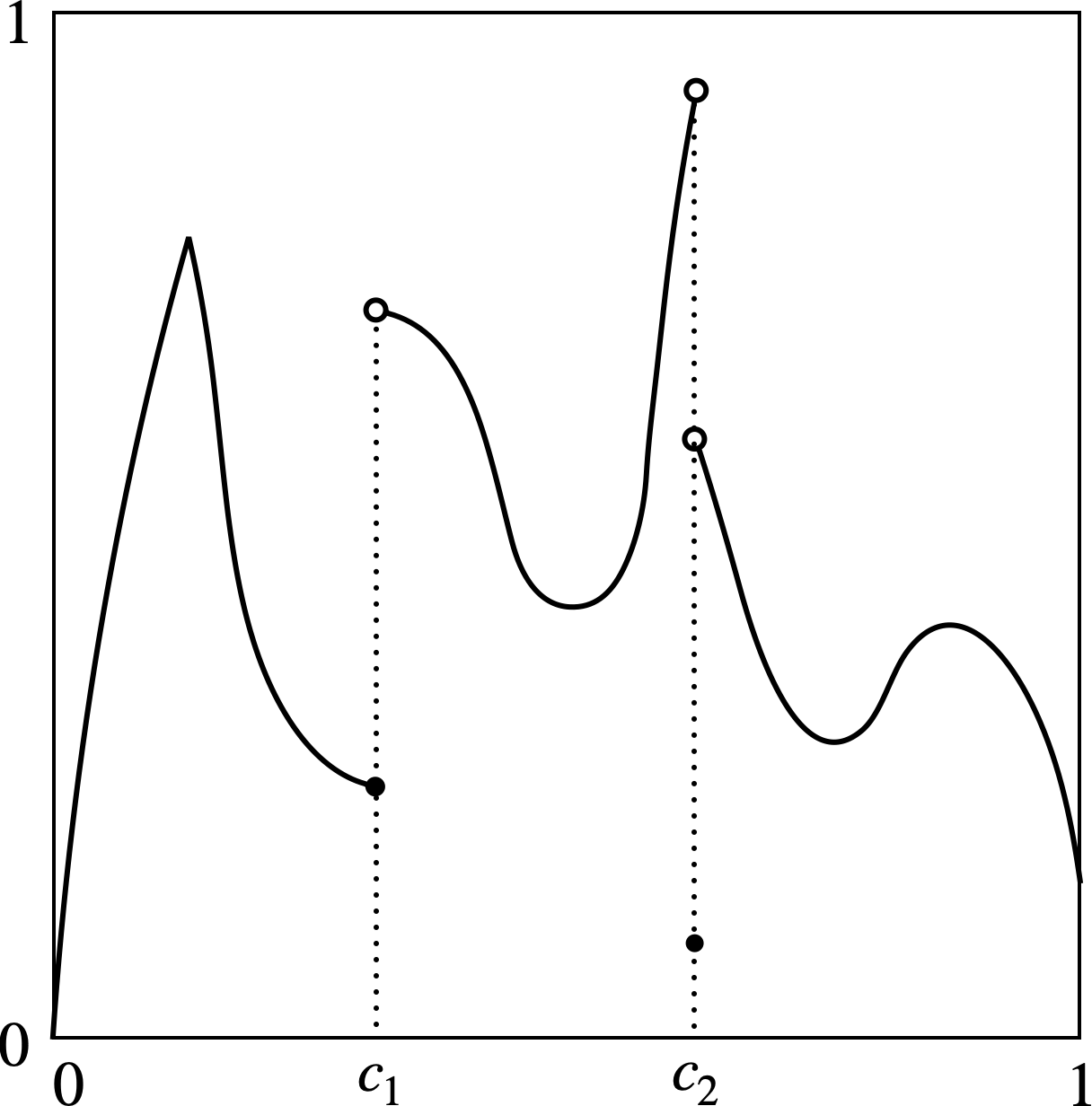}
\caption{The picture above represents the graph of a map $f$ that is not continuous.
Nevertheless the results of Section~\ref{SectionAttractors} can be applied to $g:[0,1]\setminus\cc\to[0,1]$, where $\cc=\{c_1,c_2\}$ and $g=f\big|_{[0,1]\setminus\cc}$,  since $X_0:=[0,1]\setminus\cc$ is an open and dense subset of $[0,1]$ and $g$ is a non-singular continuous map.
As $\widetilde{X}:=\bigcap_{n\ge0}g^{-n}(X_0)$ is a residual set of $[0,1]$ and $f|_{\widetilde{X}}=g|_{\widetilde{X}}$, the generic behavior of a point $x\in[0,1]$ by $f$ can be analyzed by $g$. 
}\label{NonSingCont.png}
\end{center}
\end{figure}

Let $2^{\XX}$ be the power set of $\XX$, that is, the set for all subsets of $\XX$, including the empty set.
Define ${f^*}:2^{\XX}\circlearrowleft$ given by
$${f^*}(U)=\begin{cases}
	\emptyset & \text{ if }U\cap\XX_0=\emptyset\\
	f(U\cap\XX_0) & \text{ if }U\cap\XX_0\ne\emptyset
\end{cases}$$
We say that $U\subset\XX$ is {\bf\em forward invariant} if ${f^*}(U)\subset U$ and, as before, $U$ is called {\bf\em invariant} if $f^{-1}(U)=U$ and {\bf\em almost invariant} when $f^{-1}(U)\sim U$.
Let $\co_f^+(U)=\bigcup_{n\ge0}{f^*}^n(U)$ be the forward orbit of $U\subset\XX$ and $\co_f^-(U)=\bigcup_{n\ge0} f^{-n}(U)$ the backward orbit of $U$.
For short, write ${f^*}^n(x)$, $f^{-n}(x)$, $\co_f^+(x)$ and $\co_f^-(x)$ instead of ${f^*}^n(\{x\})$, $f^{-n}(\{x\})$, $\co_f^+(\{x\})$ and $\co_f^-(\{x\})$ respectively.
The omega-limit of a point $x$, $\omega_f(x)$, is the set of accumulating points of the forward orbit of $x\in\XX$.
Precisely,
\begin{equation}\label{Equationovifcf}
\omega_f(x)=\bigcap_{n\ge0}\overline{\co_f^+({f^*}^n(x))}
\end{equation}
and the {\bf\em alpha limit set} of $x$, the of all accumulation points of the pre-orbit of $x$, is
$$\alpha_f(x)=\bigcap_{n\ge0}\overline{\co_f^-(f^{-n}(x))}.$$

Adapting the definitions given in Section~\ref{SecStatOfMainsR}, 
the {\bf\em basin of attraction} of a compact set $A$ is $$\beta_f(A)=\{x\in\XX\,;\,\emptyset\ne\omega_f(x)\subset A\}.$$
 Thus, as in Section~\ref{SecStatOfMainsR}, a compact set $A$
is called  a {\bf\em topological attractor} if
$\beta_f({A})$ and $\beta_f(A)\setminus\beta_f(A')$ are fat sets for every nonempty compact set $A'\subsetneqq A$.

\subsection{Baire ergodic components}

In many situations, $f$ may not be Baire Ergodic, but  the space can be decompose into subsets in which the restriction of $f$ to each of these subsets is Baire ergodic.
These subsets are the Baire ergodic components of $f$.

\begin{Definition}[Baire ergodic components]\label{DefiBaireErComp}
A measurable almost invariant fat set $U\subset\XX$ is called a {\bf \em Baire ergodic component of $f$} if $V\sim U$ or $V\sim\emptyset$ for every almost invariant measurable set $V\subset U$.
\end{Definition}

In Section~\ref{SectionBaireErgodicity} Baire ergodicity was defined using invariant sets. The connection between the Baire ergodicity and the  almost invariant sets was established there by Corollay~\ref{Corollaryiouregh}.
Here, since the definition above of Baire ergodic components use almost invariant sets, Lemma~\ref{Lemmakuytrdjku4} below connects Baire ergodic components with the invariant sets.

\begin{Lemma}\label{Lemmakuytrdjku4}
If $U$ is a measurable almost invariant fat set then 
$U$ is an Baire ergodic component of $f$ if and only if $V\sim U$ or $V\sim\emptyset$ for every  invariant measurable set $V\subset U$.
\end{Lemma}
\begin{proof}Suppose that $V\sim U$ or $V\sim\emptyset$ for every  invariant measurable set $V\subset U$.
Let $V$ be an almost invariant measurable set.
We may assume that $V$ is a fat set, otherwise there are nothing to prove.
Thus, it follows form Lemma~\ref{Lemmagdigi00fi75r8} that $V':=\bigcup_{n\ge0}f^{-n}\left(\bigcap_{j\ge0}f^{-j}(V)\right)$ is a measurable invariant fat set and $V'\sim V$.
Since $V'$ is invariant and fat, by assumption, $V'\sim U$ and so, $V\sim U$, proving that $U$ is a Baire ergodic component.

On the other hand, if we assume that $U$ is a Baire ergodic component, since every invariant set is almost invariant, we get that $V\sim U$ or $V\sim\emptyset$ for every measurable invariant set $V\subset U$.
\end{proof}

By Definition~\ref{DefiBaireErComp} above, if $U$ and $V$ are Baire ergodic components then either $U\sim V$ or $U\cap V\sim\emptyset$.
Thus, since we are assuming that $\XX$ is compact, $\XX$ has at most a countable number of non (meager) equivalent Baire ergodic components.
We say that $\XX$ can be decomposed into Baire ergodic components when there exists a countable collection $\{U_n\,;\,n\in L\}$, $L\subset\NN$, of Baire ergodic components such that $\XX\sim\bigcup_{n\in L}U_n$. Proposition~\ref{PropositionBaireProjectionCriterium} below gives a criterion for a finite Baire ergodic decomposition of $\XX$.

Let $\mathfrak{I}(f)\subset\mathfrak{A}$ be the sub $\sigma$-algebra of all $f$ invariant measurable sets.
A {\bf\em Baire $f$-function} is a map $\mathfrak{m}:\mathfrak{I}(f)\to[0,+\infty)$ such that $\mathfrak{m}(\XX)>0$ and 
$$A\cap B\sim\emptyset\implies \mathfrak{m}(A)+\mathfrak{m}(B)\le \mathfrak{m}(A\cup B).$$

\begin{Proposition}[Criterium for a finite Baire ergodic decomposition]\label{PropositionBaireProjectionCriterium}
If there exist a Baire $f$-function  $\mathfrak{m}$ and $\ell\in\NN$ such that either $\mathfrak{m}(U)=0$ or $\mathfrak{m}(U)\ge \mathfrak{m}({\XX})/\ell$ $\forall U\in\mathfrak{I}(f)$, then ${\XX}$ can be decomposed (up to a meager set) into at most $\ell$ Baire ergodic components.
\end{Proposition}
\begin{proof}
Let $M\subset {\XX}$ be any fat invariant measurable set (for example,
$M={\XX}$) and let $\mathcal{F}(M)$ be the collection of all fat invariant measurable sets contained in $M$.
Note that  $\mathcal{F}(M)$ is nonempty, because $M\in\mathcal{F}(M)$. Let us consider the inclusion (up to a meager subset) as a partial order on $\mathcal{F}(M)$, i.e., $A\le A'$ if $A'\setminus A$ is meager.

\begin{Claim}\label{Claimkhgivi}
Every  totally ordered subset $\Gamma\subset\mathcal{F}(M)$ is finite.
In particular, it has an upper bound.
\end{Claim}
\begin{proof}
Otherwise there is an infinite sequence
$\gamma_0\supset \gamma_1\supset\gamma_3\supset\cdots$ with $\gamma
_k\in\mathcal{F}(M)$ and
$\Delta_k:=\gamma_{k}\setminus\gamma_{k+1}$ being a fat $\forall k$.
As $\gamma_j$ is invariant $\forall\,j$, $\Delta_k$ is also a fat invariant set, that is, $\Delta_k\in\mathcal{F}(M)$.
Thus, by hypothesis, $\mathfrak{m}(\Delta_k)\ge \mathfrak{m}({\XX})/\ell$. 
As $\Delta_i\cap\Delta_j=\emptyset$ whenever $i\ne j$, we get $\frac{k}{\ell}\mathfrak{m}({\XX})\le \mathfrak{m}(\Delta_1)+\cdots+\mathfrak{m}(\Delta_k)\le \mathfrak{m}({\XX})$ $\forall\,k\in\NN$, which is a contradiction.
\end{proof}

From Zorn's Lemma, there exists a maximal
element $U\in\mathcal{F}(M)$ and, by Lemma~\ref{Lemmakuytrdjku4}, this is necessarily a Baire ergodic component.
Thus, take $M_1={\XX}$ and let $U_1$ be a maximal element of $\mathcal{F}(M_1)$ given by Zorn's Lemma.
As $M_2:={\XX}\setminus U_1$ is an invariant set, either it is meager or we can apply the argument above to $M_2$ and obtain a new Baire ergodic component $U_2$ inside ${\XX}\setminus U_1$.
Inductively, we can construct a collection of Baire ergodic components $U_1,...,U_i$ while ${\XX}\setminus(U_1\cup
...\cup U_i)$ is fat.
Nevertheless, as $\mathfrak{m}(U_j)\ge \mathfrak{m}({\XX})/\ell$ $\forall j$ and $U_j\cap U_k=\emptyset$ whenever $j\ne k$, this process have to stop at some $k\le\ell$ and so, ${\XX}\sim U_1\cup\cdots\cup U_k$.
\end{proof}

Let $\mathfrak{O}$ be the set of all open sets of $X$ and consider the following definition. 

\begin{Definition}[Baire projection]\label{DefBaireProj}
The {\bf\em Baire projection} $\pi:\mathfrak{A}\to\mathfrak{O}$  associates a measurable set to the maximal open set meager equivalent to it.
That is,
$$\pi(U)=\bigcup_{\mathfrak{O}\ni L\sim U}L.$$
\end{Definition}
Observes that $\pi(U)=\interior(\overline{V})$, for every open set $V\sim U$, and
$$A\cap B\sim\emptyset\iff\pi(A)\cap\pi(B)=\emptyset.$$

Lemma~\ref{LemmAAAAfi75r8} below is an improvement of Lemma~\ref{Lemmagdigi00fi75r8} as the $f$ invariant set obtained is a Baire subspace. 

\begin{Lemma}\label{LemmAAAAfi75r8}
If $U\subset\XX_0$ is a fat almost invariant measurable set then $$\widetilde{U}:=\bigcup_{n\ge0}f^{-n}\left(\bigcap_{j\ge0}f^{-j}(\pi(U))\right)\subset\XX_0$$ is a fat  invariant measurable set, $\widetilde{U}\sim U$ and $\widetilde{U}$ is a Baire subspace of $\XX$.
\end{Lemma}
\begin{proof}
As $f$ is non-singular, and $\pi(U)\sim U$, we get that $f^{-1}(\pi(U))\sim f^{-1}(U)$ and so, $\pi(U)\sim U\sim f^{-1}(U)\sim f^{-1}(\pi(U))$, proving that $\pi(U)$ is an almost invariant nonempty open set.
Hence, it follows from Lemma~\ref{Lemmagdigi00fi75r8} that $\widetilde{U}$ is a measurable invariant set and $\widetilde{U}\sim\pi(U)\sim U$. 
Therefore, to conclude the proof, we need to show that $\widetilde{U}$ is a Baire subspace of $\XX$.

\begin{Claim}\label{Claimhviyfyi}
$f^{-j}(\pi(U))$ is an open and dense subset of $\pi(U)$.
\end{Claim}
\begin{proof}[Proof of the claim]
As $f$ is continuous, $f^{-j}(\pi(U))$ is an open set $\forall j\ge0$.
It follows from $f^{-j}(\pi(U))\sim\pi(U)$ that $\pi(U)\cup f^{-j}(\pi(U))\sim\pi(U)\sim U$.
That is, $\pi(U)\cup f^{-j}(\pi(U))$ is an open set meager equivalent to $U$ and so, by the definition of $\pi(U)$,  $\pi(U)\cup f^{-j}(\pi(U))=\pi(U)$, proving that $f^{-j}(\pi(U))\subset\pi(U)$. Since, $f^{-j}(\pi(U))\sim\pi(U)$, we have also that $f^{-j}(\pi(U))$ is dense in $\pi(U)$.
\end{proof}
As $\pi(U)$ is a Baire subspace of $\XX$ and every countable intersection of open and dense subset of a Baire space is a Baire subspace, it follows form Claim~\ref{Claimhviyfyi} that $A:=\bigcap_{j\ge0}f^{-j}(\pi(U))$ is a Baire subspace of $\XX$.

\begin{Claim}\label{Calimufdfghj}
$f^{-j}(A)$ is a Baire subspace of $\XX$ for every $j\ge0$
\end{Claim}
\begin{proof}[Proof of the claim]
Suppose that $V_1,V_2,V_3,\cdots$ is a countable collection of open and dense subsets (in the induced topology) of $f^{-j}(A)$.
Thus, $V_k=\XX_k\cap f^{-j}(A)$, for some open set $\XX_k\subset\XX$, $k\in\NN$.
Writing $\cx:=\bigcup_k\XX_k$, we have that $\bigcap_kV_k=\bigcup_k(\XX_k\cap f^{-j}(A))=\cx\cap f^{-j}(A)$.
Using that $f$ is non-singular, we get that $f^{-j}(A)\subset f^{-j}(\pi(U))\sim f^{-j}(A)$ and so, as $\cx$ is residual in the open set $f^{-j}(\pi(U))$, we can conclude that $\cx\cap f^{-j}(A)$ is dense in $f^{-j}(A)$.
Indeed, writing $f^{-j}(\pi(U))=f^{-j}(A)\cup H$, where $H$ is a meager set,  and since $H\cap\cx$ is a meager set, we have that $f^{-j}(A)\cap\cx$ is a residual set of the open set $f^{-j}(\pi(U))$.
In particular, $f^{-j}(A)\cap\cx$ is dense in $f^{-j}(\pi(U))$ and so,  $f^{-j}(A)\cap\cx$ is dense in $f^{-j}(A)$.
Since $\cap_kV_k$ being dense in $f^{-j}(A)$ implies that $f^{-j}(A)$ is a Baire subspace of $\XX$, we finished the proof the claim.
\end{proof}

As an arbitrary union of Baire subspaces is a Baire subspace $($\footnote{ Let $A=\bigcup_{\ell}A_{\ell}$, where $L$ is a set of indices and, for each $\ell\in L$, $A_{\ell}$ is a Baire subspace of $\XX$.
Let $V_1, V_2, V_3,\cdots$ be a countable collection of open and dense (in the induced topology) of $A$. We can write $V_k=A\cap \XX_k$ for some open and dense set $\XX_k\subset\XX$.
Thus, $\bigcap_{k}V_k=\cx\cap A$, where $\cx=\bigcap_k\XX_k$ is a residual subset of $\XX$.
As a consequence, $\bigcap_kV_k=\cx\cap\bigcup_{\ell}A_{\ell}=\bigcup_{\ell}\cx\cap A_{\ell}$.
Since $A_{\ell}$ is a Baire subspace of $\XX$, $\cx$ is residual (in particular, dense) in $A_{\ell}$ and so, $\bigcap_kV_k$ is dense in $\bigcup_{\ell}A_{\ell}$, proving that $\bigcup_{\ell}A_{\ell}$ is a Baire subspace of $\XX$.}$)$, it follows from Claim~\ref{Calimufdfghj} above that $\widetilde{U}$ is a Baire subspace of $\XX$.
\end{proof}

Corollary~\ref{Corollaryjuytf9tfg76} below is the version for Baire ergodic components of Proposition~\ref{Propositionkjuytf9tfg76}. 

\begin{Corollary}
\label{Corollaryjuytf9tfg76}
If $U$ is a measurable almost invariant fat set the the following statements are equivalent.
\begin{enumerate}
\item $U$ is Baire ergodic component for $f$.
\item Given an almost invariant Baire potential $\varphi:\XX\to\YY$, there exists $y\in\YY$ such that $\varphi(x)=y$ for a residual set of points $x\in U$.
\item Given a measurable  function $\varphi:X\to\RR$   with $\limsup_{n}\frac{1}{n}\sum_{j=0}^{n-1}\varphi\circ f^{j}(x)\in\RR$ for every $x\in X$, there exists 
$r\in\RR$ such that $\limsup_{n}\frac{1}{n}\sum_{j=0}^{n-1}\varphi\circ f^{j}(x)=r$ for a residual set of points $x\in U$.
\item Given $A\in\mathfrak{A}$, there exists  $ \theta\in[0,1]$ such that $\tau_x(A)=\theta$ for a residual set of points $x\in U$.
\end{enumerate}
\end{Corollary}
\begin{proof}
By Lemma~\ref{LemmAAAAfi75r8}, $\widetilde{U}=\bigcup_{n\ge0}f^{-n}\left(\bigcap_{j\ge0}f^{-j}(\pi(U))\right)$ is an invariant fat measurable set and also a Baire subspace of $\XX$.
Therefore, we can apply Proposition~\ref{Propositionkjuytf9tfg76} to $f|_{\widetilde{U}}$ and the proof follows from the fact that $\widetilde{U}\sim U$.
\end{proof}

\subsection{Topological attractors for Baire ergodic components}
Since, in our context, the omega-limit sets are compact sets, a natural tool to analyze their behavior is the Hausdorff distance.
  
The {\bf\em distance of $x\in\XX$ and $\emptyset\ne U\in2^{\XX}$} is given by $d(x,U)=\inf\{d(x,y)\,;\,y\in U\}$.
Defining the {\bf\em open ball of radius $r>0$ and center on $\emptyset\ne U\in2^{\XX}$} as $$B_{r}(U)=\bigcup_{x\in U}B_r(x)=\{x\in\XX\,;\,d(x,U)<r\},$$ the {\bf\em Hausdorff distance} of two nonempty sets  $U$ and $V\subset\XX$ is given by
$$d_H(U,V)=\inf\{r>0\,;\,B_r(U)\supset V\text{ and }B_r(V)\supset U\}.$$

Let $\KK(\XX)$ be the set of all nonempty compact subsets of $\XX$.
Since $\XX$ is a compact metric space, it is well known that $(\KK(\XX),d_H)$ is also a compact metric space. 

\begin{Lemma}\label{LemmalOMEGAmeasu}
The map $\psi:\widetilde{\XX}\to\KK(\XX)$ given by $\psi(x)=\omega_f(x)$ is a measurable map, where $\widetilde{\XX}=\bigcap_{n\ge0}f^{-n}(\XX)$.
Moreover, $\psi$ is invariant Baire potential on the Baire space $\widetilde{\XX}$.
\end{Lemma}
\begin{proof} As $f$ is continuous, $f^{-j}(\XX)$ is an open and dense set, we get that $\XX$ is a Baire subspace of $\XX$ $($\footnote{ One can also use Lemma~\ref{LemmAAAAfi75r8}, since $\bigcap_{n\ge0}f^{-n}(\XX)=\bigcup_{n\ge0}(\bigcap_{j\ge0}f^{-j}(\XX))$.}$)$.
Moreover, by the definition of $\omega_f(x)$ (see \eqref{Equationovifcf}), we have that $\omega_f(x)=\omega_f(f(x))$ for every $x\in\XX$.
In particular, $\psi(x)=\psi(f(x))\in\KK(\XX)$ for every $x\in\widetilde{\XX}$.
Thus, since $\KK(\XX)$ is a complete separable metric space, we have that $\psi$ is an invariant Baire potential on $\widetilde{\XX}$.
As a consequence, we need only to show that $\psi$ is measurable.

Let $\psi_{n,m}:\widetilde{\XX}\to\KK(\XX)$ be given by $\psi_{n,m}(x)=\{f^{n}(x),\cdots,f^{n+m}(x)\}$.
As $f$ is continuous, it is easy to see that $\psi_{n,m}$ is a sequence of continuous maps and 
it is easy to show that $\lim_m\psi_{n,m}(x)=\overline{\co_f^+(f^n(x))}$.
Hence, the map $\psi_n:\widetilde{\XX}\to\KK(\XX)$ given by $\psi_n(x)=\overline{\co_f^+(f^n(x))}$ is a measurable map, since it is the pointwise limit of measurable maps.
As $\omega_f(x)=\bigcap_{n}\overline{\co_f^+(f^n(x))}$, one can see that $\lim_n\psi_n(x)=\omega_f(x)=\psi(x)$ for every $x\in\widetilde{\XX}$, proving that $\psi$ is a measurable map. 
\end{proof}

\begin{Proposition}[The topological attractor of a Baire ergodic component]\label{Propositiontop-ergodicAttractors} 
If $U\subset \XX$ is a Baire ergodic component of $f$, then there exists a unique topological attractor $A\subset\overline{U}$ attracting a residual subset of $U$.
Moreover, $\omega_f(x)=A$ for a residual set of points $x\in U$.
\end{Proposition}
\begin{proof} Let $\widetilde{\XX}$ and $\psi$ be as in Lemma~\ref{LemmalOMEGAmeasu}
and set $\widetilde{U}:=\bigcup_{n\ge0}f^{-n}\left(\bigcap_{j\ge0}f^{-j}(\pi(U))\right)$.
By Lemma~\ref{LemmAAAAfi75r8}, $\widetilde{U}$ is an invariant measurable fat set and a Baire subspace of $\XX$.
As $\widetilde{U}\subset\widetilde{\XX}$ and $\psi$ is an invariant Baire potential on $\widetilde{\XX}$, we can apply  Proposition~\ref{Propositionkjuytf9tfg76} to $f|_{\widetilde{U}}$ and $\psi|_{\widetilde{U}}$ and conclude that there exists some $A\in\KK(\XX)$ such that $\omega_f(x)=A$ for a residual set of points $x\in\widetilde{U}$.
Since $U\sim\widetilde{U}$, we conclude the proof.
\end{proof}

Proposition~\ref{Propositiontop-ergodicAttractors} above and Lemma~\ref{Lemma98uyhji} imply, for non-singular maps, the known fact that a {\bf\em transitive continuous map is also transitive for generic points}.
Indeed, suppose that $g:\XX\circlearrowleft$ is a non-singular transitive continuous map. 
Thus, by Lemma~\ref{Lemma98uyhji}, $g$ is Baire-ergodic and so, by Proposition~\ref{Propositiontop-ergodicAttractors}, there exist a compact set $A\subset\XX$ and a residual set $R\subset\XX$ such that $\omega_g(x)=A$ for all $x\in R$.
Let $p\in\XX$ be such that $\omega_g(p)=\XX$.
If $A\ne\XX$, let $\delta>0$ be small enough so that $\overline{B_{\delta}(A)}\ne\XX$.
Let $R_n=\{x\in R\,;\,\co_g^+(g^n(x))\subset B_{\delta}(A)\}$, $n\ge0$.
As $\bigcup_{n\ge0}R_n=R\sim\XX$, $R_{\ell}$ is a fat set for some $\ell\ge0$.
Choose an open set $V\sim R_{\ell}$ and $m\in\NN$ so that $g^m(p)\in V$.
As  $g^j(g^m(p))\in g^j(V)\subset \overline{B_{\delta}(A)}$ for every $j\ge\ell$, we get the $\XX=\omega_g(p)=\omega_g(g^m(p))\subset\overline{B_{\delta}(A)}\ne\XX$, a contradiction.

\begin{Lemma}[Ball criterium  for a finite Baire ergodic decomposition]\label{LemmaCriteionForErgodicityII}
If there exists $\delta>0$ such that every invariant measurable set is either meager or it is residual in some open ball of radius $\delta$, then $\XX$ can be decomposed (up to a meager set) into at most a finite number  of Baire ergodic components.
\end{Lemma}
\begin{proof} Let $\mu$ be a Borel probability measure on $\XX$ with $\supp\mu=\XX$.
For instance, we may consider a countable and dense subset  $\{x_n\,;\,n\in\NN\}$ of $\XX$ and take $\mu=\sum_{n\in\NN}\frac{1}{2^n}\delta_{x_n}$.
Note that $\mathfrak{m}:\mathfrak{I}(f)\to[0,1]$, given by $\mathfrak{m}(U)=\mu(\pi(U))$, is Baire $f$-function, where $\pi$ is the Baire projection (Definition~\ref{DefBaireProj}).

By compactness, there exists $\ell\in\NN$ such that $\inf\{\mu(B_{\delta}(p))\,;\,p\in\XX\}\ge1/\ell$.
Thus, if $U\in\mathfrak{I}(f)$ is a fat set then, by hypothesis, $\pi(U)$ contains an open ball $B$ of radius $\delta$ and so, $\mathfrak{m}(U)=\mu(\pi(U))\ge\mu(B)\ge1/\ell$.
On the other hand, if $U\in\mathfrak{I}(f)$ is meager then $\pi(U)=\emptyset$ and so, $\mathfrak{m}(U)=\mu(\emptyset)=0$.
Hence, the proof follows from Proposition~\ref{PropositionBaireProjectionCriterium}.
\end{proof}

\begin{Lemma}\label{Lemmaiuyg652}
 	If $\Lambda$ is a meager compact set then given any $\delta>0$ there is $\varepsilon>0$ such that $B_{\varepsilon}(\Lambda)=\bigcup_{x\in\Lambda}B_{\varepsilon}(x)$ does not contain any ball of radius $\delta$.
That is, $$\lim_{\varepsilon\to0}\sup\{r>0\,;\,B_{r}(p)\subset B_{\varepsilon}(\Lambda)\text{ and }p\in\XX\}=0.$$
 \end{Lemma}
 \begin{proof}
 	Otherwise, as $\Lambda$ is compact, there exist $\delta>0$ and a convergent sequence $p_n\in\Lambda$ such that $B_{\delta}(p_n)\subset B_{1/n}(\Lambda)$ $\forall\,n\ge1$.
This implies that $d(x,\Lambda)=0$ for every $x\in B_{\delta}(p)$, where $p=\lim_n p_n$. By compactness, we get that $B_{\delta}(p)\subset\Lambda$, contradicting the hypotheses of $\Lambda$ being meager.
 \end{proof}

Define the {\bf\em large omega limit} of a point $x\in\XX$ as
$$\Omega_f(x) =\bigcap_{r>0}\bigcap_{n\ge0}\bigg(\overline{\bigcup_{m\ge n}{f^*}^m(B_{r}(x))}\bigg).$$
Note that $\Omega_f(x)$ is a well defined nonempty compact set for every $x\in\XX$, even if $x\notin\XX_0$.

\begin{Lemma}\label{LemmaFGHJKUYTR}
There exists a residual set $\cR\subset\XX$ such that $\Omega_f(f(x))=\Omega_f(x)$ $\forall x\in\cR$.
\end{Lemma}
\begin{proof}
Since $\XX_0$ is an open and dense subset of $\XX$, applying Lemma~\ref{Lemmakufyuvooyb} at Appendix, there exists residual set $\cR\subset\XX_0$  such that if $x\in\cR$ then  $f(x)\in\interior(f^*(V))$ for every open set $V\subset\XX$ containing $x$.

If $x\in\cR$ then, given $\delta>0$, we can choose $\delta_1>0$ so that $B_{\delta_1}(f(x))\subset f(B_{\delta}(x))$.
Hence, we get that $\bigcap_{n\ge0}\overline{\bigcup_{m\ge n}(f^*)^m(B_{\delta}(x))}\supset\bigcap_{n\ge0}\overline{\bigcup_{m\ge n}(f^*)^m(B_{\delta_1}(f(x)))}\supset\Omega_f(f(x)).$
That is,
$\bigcap_{n\ge0}\overline{\bigcup_{m\ge n}(f^*)^m(B_{\delta}(x))}\supset\Omega_f(f(x))\text{ for every }\delta>0$.
and so, $\Omega_f(x)\supset\Omega_f(f(x))$.

On the other hand, by the continuity of $f$, taking $\delta>0$, one can choose $\delta_1>0$  so that $f(B_{\delta_1}(x))\subset B_{\delta}(f(x))$.
Hence,
$\bigcap_{n\ge0}\overline{\bigcup_{m\ge n}(f^*)^m(B_{\delta}(f(x)))}$ $\supset$ $\bigcap_{n\ge0}\overline{\bigcup_{m\ge n}(f^*)^m(B_{\delta_1}(x))}$ $\supset$ $\Omega_f(x),$
proving that 
$\bigcap_{n\ge0}\big(\overline{\bigcup_{m\ge n}(f^*)^m(B_{\delta}(f(x)))}\big)$ $\supset$ $\Omega_f(x)\text{ for every }\delta>0$
and, as a consequence,  $\Omega_f(x)\subset\Omega_f(f(x))$.
\end{proof}

The {\bf\em nonwandering  set} of $f$, denoted by $\Omega(f)$, is the set of points $x\in\XX$ such that $V\cap\bigcup_{n\ge1}{f^*}^n(V)\ne\emptyset$ for every open neighborhood $V$ of $x$. It is easy to see that
$$\Omega(f)=\{x\in\XX\,;\,x\in\Omega_f(x)\},$$
that is $\Omega(f)$ is the set of all ``$\Omega$-recurrent'' points of $\XX$ (recall that a point is called recurrent (or ``$\omega$-recurrent'') if $x\in\omega_f(x)$).

A forward invariant set $V$ is called {\bf\em strongly transitive} if $\bigcup_{n\ge0}{f^*}^n(A)=V$ for every nonempty open set (in the induced topology) $A\subset V$. One can check that $V$ is strongly transitive if and only if $\alpha_f(x)\supset V$ for every $x\in V$.

\begin{Theorem}
\label{TheoremFatErgodicAttractors} If there exists $\delta>0$ such that $\overline{\bigcup_{n\ge0}{f^*}^n(U)}$ contains some open ball of radius $\delta$, for every nonempty open set $U\subset\XX$, then $\XX$ can be decomposed (up to a meager set) into a finite number of Baire ergodic components $U_1,\cdots, U_\ell\subset\XX$,  each $U_j$ is an open set and  the attractors $A_j$ associated to  $U_j$ (given by Proposition~\ref{Propositiontop-ergodicAttractors}) satisfy the following properties.
\begin{enumerate}
\item\label{Item(12b)} Each $A_j$ contains some open ball $B_j$ of radius $\delta$ and $A_j=\overline{\interior(A_j)}$.
\item\label{Item(3b)}
\begin{enumerate}
\item Each $A_j$ is transitive and $\omega_f(x)=A_j$ for a residual set of points $x\in\beta_f(A_j)$.
\item If $\varphi:\XX\to\RR$ is a Borel measurable bounded function then for each $A_j$ there exists $a_j\in\RR$ such that $\limsup\frac{1}{n}\sum_{j=0}^{n-1}\varphi\circ f^j(x)=a_j$ for a residual set of points $x\in\beta_f(A_j)$.
\item If $U$ is a Borel subset of $\XX$ then for each $A_j$ there exists $u_j\in[0,1]$ such that $$\limsup_{n\to+\infty}\frac{1}{n}\#\{0\le j<n\,;\,f^j(x)\in U\}=u_j$$
for a residual set of points $x\in\beta_f(A_j)$.

\end{enumerate}
\item\label{Item+}$\interior(\beta_f(A_j))\sim\beta_f(A_j)\sim U_j$ for every $1\le j\le\ell$. In particular, \begin{center}$\beta_f(A_1)\cup\cdots\cup\beta_f(A_{\ell})$ contains an open and dense subset of $\XX$.\end{center}
\item\label{Item(4b)} $\Omega_f(x)\supset A_j$ if $x\in\overline{U_j}$ and $\Omega_f(x)= A_j$ if $x\in U_j$, $\forall\,j$.
In particular $$\Omega(f)\subset\bigg(\bigcup_{j=1}^{\ell}A_j\bigg)\cup\bigg(\XX\setminus\bigcup_{j=1}^{\ell}U_j\bigg),$$ where $\XX\setminus\bigcup_{j=1}^{\ell}U_j$ is a compact set with empty interior.
\end{enumerate}
Furthermore, if $\bigcup_{n\ge0}{f^*}^n(U)$ contains some open ball of radius $\delta$, for every nonempty open set $U\subset\XX$, then the following statements are true.
\begin{enumerate}\setcounter{enumi}{4}
\item\label{Item(5b)}
For each $A_j$ there is a forward invariant set $\ca_j\subset A_j$ containing an open and dense subset of $A_j$ such that $f$ is strongly transitive in $\ca_j$.
Indeed, $\alpha_f(x)\supset \overline{U_j}\supset A_j\supset\ca_j$ for every $x\in\ca_j$.
\item\label{Item(6b)} Either $\omega_f(x)=A_j$ for every $x\in\ca_j$ with $\omega_f(x)\ne\emptyset$ or $A_j$ has sensitive dependence on initial conditions.
\end{enumerate}
\end{Theorem}
\begin{proof}
Given a  forward invariant measurable fat set $U$ and an open set $V$ such that $U\sim V$, it follows from Proposition~\ref{PropositionRESUMO} at Appendix that 
\begin{equation}\label{Equatiobjhfyr66}
\text{$U\cap V$ is a residual subset of $\interior({f^*}^n(V))$ for every $n\in\NN$.}
\end{equation}
In particular, $U$ is residual in every nonempty open subset of $V$.

By hypothesis, $\overline{\bigcup_{n\ge0}{f^*}^n(V)}$ contains some open ball $B$ with radius $\delta$ and so, $B\cap \bigcup_{n\ge0}{f^*}^n(V)$ is a dense subset of $B$.
By Proposition~\ref{PropositionRESUMO}, 
 $B$ $\cap$ $\bigcup_{n\ge0}\interior({f^*}^n(V))$ is an open and dense subset of $B$.
Hence, it follows from \eqref{Equatiobjhfyr66} that $U$ is residual in $B$.
That is,
\begin{equation}\label{Equatioescrita}
\text{\em every forward invariant fat measurable set 
is residual in some open ball of radius }\delta.
\end{equation}
Thus, the hypothesis of Theorem~\ref{TheoremFatErgodicAttractors} implies the hypothesis of Lemma~\ref{LemmaCriteionForErgodicityII}, since every invariant set is a forward invariant one.
Hence, $\XX$ can be decomposed (up to a meager set) into at most $\ell\ge1$ Baire ergodic components $W_1,\cdots,W_{\ell}$.

By Proposition~\ref{Propositiontop-ergodicAttractors}, each Baire ergodic component $W_j$ has a unique topological attractor $A_j$ such that $\omega_f(x)=A_j$ for a residual set of points $x\in W_j$.

We claim that $A_j$ is a fat set. Indeed, if $A_j$ is a meager set then it follows from Lemma~\ref{Lemmaiuyg652} above that $\exists\,\varepsilon>0$ such that
\begin{equation}\label{Equaiu679}
	\sup\{r>0\,;\,B_{r}(p)\subset B_{\varepsilon}(A_j)\text{ and }p\in\XX\}<\delta/2.
\end{equation}
Let $n_0\ge1$ be big enough so that $W_j':=\{x\in W_j\ {;}\
f^n(x)\in B_{\varepsilon}(A_j),\,\forall n\ge n_0\}$ is a fat set.
As $f$ is  non-singular and $W_j'$ is a fat set, we have that also ${f^*}^{n_0}(W_j')$ is fat.
As $W_j'$ is forward invariant, ${f^*}^{n_0}(W_j')$ is also a forward invariant set.
Nevertheless,
as ${f^*}^{n_0}(W_j')\subset B_{\varepsilon}(A_j)$, it follows from the inequation  \eqref{Equaiu679} above that $W_j'$ cannot be residual in a ball of radius bigger or equal to $\delta/2$, but this is in contradiction with \eqref{Equatioescrita}.
Therefore $A_j$ is fat set.

As $A_j$ is compact fat set, $\interior(A_j)\ne\emptyset$.
This implies, as $A_j$ is compact forward invariant set, that $A_j\supset\overline{\bigcup_{j\ge0}f^{*n}(\interior(A_j))}$.
Hence, by the theorem's hypothesis, $A_j$ contains some open ball $B_j$ of radius $\delta$.
 
Since $\omega_f(x)=A_j$ for every $x$ in a residual set $R_j\subset W_j$, for each $x\in R_j$ there is some $n_x\ge1$ such that $f^{n_x}(x)\in B_j$.
As $\omega_f(x)=A_j$ residually on $B_j\subset A_j$, we get that there exists $p\in B_j$ such that $\co_f^+(p)$ is a dense subset of $A_j$, in particular, $A_j$ is transitive.
It follows from Proposition~\ref{PropositionRESUMO}, $f^n(p)\in {f^*}^n(B_j)\subset\overline{\interior({f^*}^n(B_j))}\subset A_j$ $\forall\,n\ge0$.
In particular, $d_H(\{f^n(p)\},\interior({f^*}^n(B_j)))=0$ $\forall\,n\ge0$.
This implies that $$d_H\bigg(A_j,\bigcup_{n\ge0}\interior({f^*}^n(B_j))\bigg)=d_H\bigg(\co_f^+(p),\bigcup_{n\ge0}\interior({f^*}^n(B_j))\bigg)=0,$$ which proves  that $\bigcup_{n\ge0}\interior({f^*}^n(B_j))$ is an open and dense subset of $A_j$ and so, $A_j=\overline{\interior(A_j)}$.  
As $W_j\sim U_j:=\bigcup_{n\ge0}f^{-n}(\interior(A_j))$, we have that $U_j$ is a Baire ergodic component with $A_j=\overline{\interior(A_j)}$ being its transitive Topological attractor and $\omega_f(x)=A_j$ residually in $U_j$, proving items {\em (1)} and {\em (2)(a)}.
Items {\em (2)(b)} and {\em (2)(c)} follows from Corollary~\ref{Corollaryjuytf9tfg76}, concluding the proof of items  {\em (\ref{Item(12b)})} and {\em (\ref{Item(3b)})}. 
Therefore, we can consider the open sets $U_j,\cdots,U_{\ell}$, instead of $W_1,\cdots,W_{\ell}$,  as the decomposition (up to a meager set) of $\XX$ into Baire ergodic components.
Furthermore, by definition, if $x\in U_j$ then $f^n(x)\in A_j$ for some $n\ge0$ and so, $\omega_f(x)\subset A_j$, proving that $U_j\subset\beta_f(A_j)$.
Hence, as $\beta_f(A_j)\setminus U_j\subset\XX\setminus\bigcup_{n=1}^{\ell}U_n\sim\emptyset$, we conclude the proof of item {\em(\ref{Item+})}.

Given $p\in\overline{U_j}$ and $\varepsilon>0$, let $p_{\varepsilon}\in B_{\varepsilon}(p)\cap R_j$.
As $\omega_f(p_{\varepsilon})=A_j$, we get that $$\bigcap_{n\ge0}\bigg(\overline{\bigcup_{m\ge n}{f^*}^m(B_{\varepsilon}(p))}\bigg)\supset\omega_f(p_{\varepsilon})=A_j.$$
That is, $\bigcap_{n\ge0}\overline{\bigcup_{m\ge n}{f^*}^m(B_{\varepsilon}(p))}\supset A_j$ for every $\varepsilon>0$, proving that $\Omega_f(x)\supset A_j$ for every $p\in\overline{U_j}$ and concluding the proof of item {\em(\ref{Item(4b)})}.

Now, assume that $\bigcup_{n\ge0}{f^*}^n(A)$ contains some open ball of radius $\delta$, for every nonempty open set $A\subset\XX$.
Define, for $0<r<\delta$, $\Delta_r^n(x)=\bigcup_{m\ge n}{f^*}^m(B_r(x))$.
For $0<\varepsilon<r/2$, note that $\Delta_r^n(x)\setminus B_{\varepsilon}(\partial\Delta_r^n(x))$ is a compact set (\footnote{ Recall that $B_{\varepsilon}(\partial\Delta_r^n(x))=\bigcup_{p\in\partial\Delta_r^n(x)}B_{\varepsilon}(p)$.}) and, as $\Delta_r^n(x)$ contains an open ball of radius $\delta$, $\Delta_r^n(x)\setminus B_{\varepsilon}(\partial\Delta_r^n(x))\ne\emptyset$, indeed, it contains a ball of radius $\delta-\varepsilon$.
Hence, 
$$\Omega_f^r(x):=\lim_{\varepsilon\searrow0}\bigcap_{n\ge0}\big(\Delta_r^n(x)\setminus B_{\varepsilon}(\partial\Delta_r^n(x))\big)\in\KK(\XX)$$
is a well defined nonempty compact set for every $x\in\XX$.
Moreover, for every $x\in\XX$, $\Omega_f^r(x)$ contains an open ball $B_{x,r}$ of radius $\delta$ and
\begin{equation}\label{Equationnnhgoi7}
  B_r(x)\cap\alpha_f(y)\ne\emptyset\;\;\forall y\in B_{x,r}.
\end{equation}
As $\Omega_f^0(x):=\lim_{r\searrow0}\Omega_f^r(x)=\bigcap_{r>0}\Omega_f^r(x)$, it follows from (\ref{Equationnnhgoi7}) above that, for every $x
\in\XX$,
\begin{enumerate}[(a)]
	\item $\Omega_f^0(x)$ contains an open ball $B_x$ of radius $\delta$ and
	\item $x\in\alpha_f(y)$ for every $y\in B_x$.
\end{enumerate}
\begin{claim}
$\psi:\XX\to\KK(\XX)$ given by $\psi(x)=\Omega_f^0(x)$ is a measurable map.
\end{claim}
\begin{proof}[Proof of the claim]
Let $\Delta_r^{n,m}(x)=\bigcup_{j=n}^m{f^*}^j(B_r(x))$ and $\psi_{r,\varepsilon,n,m}:\XX\to\KK(\XX)$ be given by $\psi_{r,n,m}(x)=\Delta_r^{n,m}(x)\setminus B_{\varepsilon}(\partial \Delta_r^{n,m}(x))$. 
As $\psi_{r,\varepsilon,n,m}$ is a continuous map and $\lim_{m}\psi_{r,\varepsilon,n,m}(x)=\Delta_r^n(x)\setminus B_{\varepsilon}(\partial \Delta_r^n(x))$, we get that $\psi_{r,\varepsilon,n}:\XX\to\KK(\XX)$, given by $\psi_{r,\varepsilon,n}(x)=\Delta_r^n(x)\setminus B_{\varepsilon}(\partial \Delta_r^n(x))$, is a measurable map.
Likewise $\psi_{r,\varepsilon}:=\lim_{n}\psi_{r,\varepsilon,n}$, $\psi_{r}:=\lim_{\varepsilon\searrow0}\psi_{r,\varepsilon}$, $\psi=\lim_{r\searrow0}\psi_{r}$ are measurable maps.
\end{proof}
Following the proof of Lemma~\ref{LemmaFGHJKUYTR}, one can show that there exists a residual set $\cR\subset\XX$ such that $\psi\circ f(x)=\psi(x)$ for every $x$ for every $x\in\cR$, i.e., $\psi$ is an almost invariant potential.
Thus, it follows from Corollary~\ref{Corollaryjuytf9tfg76} that, for each $1\le j\le\ell$, there exists a compact set $K_j\in\KK(\XX)$ such that $\psi(x)=K_j$ for a residual set of points $x\in U_j$.
Since $\psi(x)=\Omega_f^0(x)$ and every $\Omega_f^0(x)$ contains an open ball of radius $\delta$, we get that, for each $1\le j\le\ell$, there exists $p_j\in\XX$ such that $K_j\supset B_{\delta}(p_j)$. 
In particular, $\Omega_f^0(x)\supset B_{\delta}(p_j)$ for a residual set of points $x\in U_j$. 
Thus, for every $y\in B_{\delta}(p_j)$, $\alpha_f(y)\ni x$ for a residual set of $x\in U_j$.
By compactness, $\alpha_f(y)\supset\overline{U_j}\supset A_j$ for every $y\in B_{\delta}(p_j)$.
As $\interior(A_j)\ne\emptyset$ and $A_j$ is forward invariant, we have that $A_j\supset B_{\delta}(p_j)$.
Since $\alpha_f(f^j(x))\supset\alpha_f(x)$ always, we get that $\alpha_f(x)\supset\overline{U_j}\supset A_j$ for every $x\in\ca_j:=\bigcup_{n\ge0}{f^*}^n(B_{\delta}(p_j))$, proving that $f$ is strongly transitive in the forward invariant set $\ca_j\subset A_j$.
Furthermore, it follows from the transitivity of $A_j$ and Proposition~\ref{PropositionRESUMO} at Appendix that $\ca_j$ contains an open in dense subset of $A_j$, proving item {\em(\ref{Item(5b)})}. 

Suppose that there exists $p\in\ca_j$ such that $\emptyset\ne\Lambda:=\omega_f(p)\ne A_j$.
By compactness, $\Lambda$ is not a dense subset of $A_j=\overline{\ca_j}$ and so,  $\ca_j\setminus\Lambda\ne\emptyset$.
Choose $q\in\ca_j\setminus\Lambda$ and set $r=d_H(\{q\},\Lambda)>0$. 
Let $n_0\ge0$ be such that $\co_f^+(f^{n_0}(p))\subset B_{r/2}(\Lambda)$.
Given $x\in A_j$ and $\varepsilon>0$ let $n_1\ge0$ and 
$n_2\ge n_1+n_0$ be such that $f^{-n_1}(p)\cap B_{\varepsilon}(x)\ne\emptyset\ne f^{-n_2}(q)\cap B_{\varepsilon}(x)$.
As $f^n(p)\in B_{r/2}(\Lambda)$ for every $n\ge n_0$, we get that ${f^*}^{n_2}(B_{\varepsilon}(x))\cap B_{r/2}(\Lambda)\ne\emptyset$ and $q\in {f^*}^{n_2}(B_{\varepsilon}(x))$, proving that, $\sup_{n\ge1}\diam({f^*}^{n}(B_{\varepsilon}(x)))\ge r/2$ for every $x\in A_j$ and $\varepsilon>0$.
As this implies the sensitive dependence on initial conditions (item {\em(\ref{Item(6b)})}), we conclude the proof of the theorem.
\end{proof}

\subsection{$u$-Baire ergodic components and its topological attractors}
Similarly to Baire ergodic components, we can define the $u$-Baire ergodic components.

\begin{Definition}[$u$-Baire ergodic components]\label{Defioihbohb}
An almost invariant measurable set $U\subset\XX$ is called a {\bf \em $u$-Baire ergodic component of $f$} if $U$ is a fat set, $U=W_f^s(U)$ and $V\sim U$ or $V\sim\emptyset$ for every measurable set $V\subset U$ such that $f^{-1}(V)\sim V=W_f^s(V)$.
\end{Definition}

\begin{Lemma}\label{Lemmagdighfyt5r8}
If $U$ is a fat measurable set such that $f^{-1}(U)\sim U=W_f^s(U)$ then $\widetilde{U}=\bigcup_{j\ge0}f^{-j}\big(\bigcap_{n\ge0}f^{-n}(\pi(U))\big)\subset\bigcap_{n\ge0}f^{-n}(\XX)$ is a fat measurable set, $\widetilde{U}\sim U$,  $f^{-1}\big(\widetilde{U}\big)=\widetilde{U}=W_f^s\big(\widetilde{U}\big)$ and $\widetilde{U}$ is a Baire subspace of $\XX$.
\end{Lemma}
\begin{proof}It follows from Lemma~\ref{LemmAAAAfi75r8} that $\widetilde{U}$ is an invariant measurable set, $\widetilde{U}\sim U$ and $\widetilde{U}$ is a Baire subspace of $\XX$.
Thus, we need only to show that $W_f^s(\widetilde{U})=\widetilde{U}$.

Note that $U_0=\bigcap_{n\ge0}f^{-n}(U)$ is the set of all points $x\in U$ such that $f^n(x)\in U$ $\forall\,n\ge0$, this implies that $f(U_0)\subset U_0$.
Moreover, $W_f^s(U_0)=U_0$.
Indeed, $x\in U_0$ $\implies$ $f^n(x)\in U$ $\implies$ $W_f^s(f^n(x))\subset W_f^s(U)=U$ $\forall n\ge0$ and, as $f^n(W_f^s(x))\subset W_f^s(f^n(x))$, we get that $f^n(W_f^s(x))\subset U$ $\forall n\ge0$.
That is, $W_f^s(x)\subset\bigcap_{n\ge0}f^{-n}(U)=U_0$ and so,
$$U_0\subset W_f^s(U_0)=\bigcup_{x\in U_o}W_f^s(x)\subset U_0.$$

Thus, since $\widetilde{U}=\bigcup_{n\ge0}f^{-n}(U_0)$ and $f^{-n}(W_f^s(p))=W_f^s(f^{-n}(p))=\bigcup_{y\in f^{-n}(p)}W_f^s(y)$, we get that $$W_f^s\big(\widetilde{U}\big)=W_f^s\left(\bigcup_{n\ge0}f^{-n}(U_0)\right)=\bigcup_{n\ge0}W_f^s\left(f^{-n}(U_0)\right)=\bigcup_{n\ge0}f^{-n}\left(W_f^s(U_0)\right)=\bigcup_{n\ge0}f^{-n}(U_0)=\widetilde{U}.$$
\end{proof}

The proof of Corollary~\ref{Corollaryfyt5r8jku4} below is similar to the proof of Lemma~\ref{Lemmakuytrdjku4}.

\begin{Corollary}\label{Corollaryfyt5r8jku4}
Let $U$ be an almost invariant measurable set, 
$U$ is an Baire ergodic component of $f$ if and only if $V\sim U$ or $V\sim\emptyset$ for every   measurable set $V\subset U$ such that $f^{-1}(V)=V=W_f^s(V)$.
\end{Corollary}
\begin{proof}  
Assume that $L\sim U$ or $V\sim\emptyset$ for every   measurable set $V\subset U$ such that $f^{-1}(V)=V=W_f^s(V)$.
We need to show that if $V\sim U$ or $V\sim\emptyset$ for every   measurable set $V\subset U$ such that $f^{-1}(V)\sim V=W_f^s(V)$.
Thus, let $V$ be a measurable set such that $f^{-1}(V)\sim V=W_f^s(V)$.
Since $f$ is non-singular, if $V\sim\emptyset$ then $f^{-1}(V)\sim\emptyset\sim V$.
Hence, we can assume that $V$ is a fat set.
Thus, it follows form Lemma~\ref{Lemmagdighfyt5r8} above that $\widetilde{V}:=\bigcup_{n\ge0}f^{-n}\left(\bigcap_{j\ge0}f^{-j}(V)\right)$ is a measurable invariant set with $f^{-1}\big(\widetilde{V}\big)=\widetilde{V}=W_f^s\big(\widetilde{V}\big)$ and $\widetilde{V}\sim V$.
Therefore, it follows from our assumption that  $\widetilde{V}\sim U$.
As a consequence,  $V\sim U$, proving that $V\sim U$ or $V\sim\emptyset$ for every   measurable set $V\subset U$ with $f^{-1}(V)\sim V=W_f^s(V)$.

Assuming that $U$ is a $u$-Baire ergodic component, since every invariant set is almost invariant, we get that $V\sim U$ or $V\sim\emptyset$ for every  invariant measurable set $V\subset U$ such that $V=W_f^s(V)$.
\end{proof}

Let $\mathfrak{I}^u(f)\subset\mathfrak{A}$ be the sub $\sigma$-algebra of all  measurable sets $f^{-1}(U)=U=W_f^s(U)$.
A {\bf\em $u$-Baire $f$-function} is a map $\mathfrak{m}:\mathfrak{I}^u(f)\to[0,+\infty)$ such that $\mathfrak{m}(\XX)>0$ and 
$$A\cap B\sim\emptyset\implies \mathfrak{m}(A)+\mathfrak{m}(B)\le \mathfrak{m}(A\cup B).$$

\begin{Proposition}[Criterium  for a finite $u$-Baire ergodic decomposition]\label{PropositionCriteionFor-u-Ergodicity}
If there exist a $u$-Baire $f$-function $\mathfrak{m}$ and $\ell\in\NN$ such that  either $\mathfrak{m}(U)=0$ or $\mathfrak{m}(U)\ge \mathfrak{m}(\XX)/\ell$ for every $U\in\mathfrak{I}^u(f)$, then $\XX$ can be decomposed into at most $\ell$ $u$-Baire ergodic components.
\end{Proposition}
\begin{proof}
Given $M\subset\XX$ a fat measurable set such that $f^{-1}(M)=M=W_f^s(M)$, define $\cf(M)$ as the collection of all fat invariant measurable set $U$ contained in $M$ and such that $f^{-1}(U)=U=W_f^s(U)$.
Now the proof follows exactly as the proof of Proposition~\ref{PropositionBaireProjectionCriterium}.
That is, take $M_1=\XX$ and note that $\mathcal{F}(M_1)$ is
non-empty, because $M_1\in\mathcal{F}(M_1)$.
We say that  $A\le A'$ if $A'\setminus A$ is meager.
Using the same argument of Claim~\ref{Claimkhgivi} in the proof of Proposition~\ref{PropositionBaireProjectionCriterium}, we can show that every totally ordered subset $\Gamma\subset\mathcal{F}(M_1)$ is finite (in particular, it has an upper bound), using Zorn's Lemma, there exists a maximal element $U_1\in\mathcal{F}(M_1)$ and, by Corollary~\ref{Corollaryfyt5r8jku4}, $U_1$ is necessarily a $u$-Baire ergodic component.

Let $M_2=\XX\setminus U_1$ and note that it satisfies $f^{-1}(M_2)= M_2= W_f^s(M_2)$. Either $M_2$ is meager or we can use the argument above to $M_2$ and obtain a new $u$-Baire ergodic component $U_2$ inside $\XX\setminus U_1$.
Inductively, as in the proof of Proposition~\ref{PropositionBaireProjectionCriterium}, we  construct a collection of $u$-Baire ergodic components $U_1,\cdots,U_i$ while $\XX\setminus(U_1\cup\cdots\cup U_i)$ is a fat set.
Nevertheless, as $U_j\cap U_k =\emptyset$ when $j\ne k$ and $\mathfrak{m}(U_j)\ge1/\ell$ $\forall j$,  we have that $\frac{k}{\ell}\mathfrak{m}(\XX)\le \mathfrak{m}(U_1)+\cdots+\mathfrak{m}(U_i)\le \mathfrak{m}(\XX)$ and so, this process has to stop at some $k\le\ell$.
As a consequence $\XX\sim U_1\cup\cdots\cup U_{k}$.
\end{proof}

\begin{Lemma}\label{Lemmakjbidtx}
If an almost forward invariant open set $U\subset\XX$ is  transitive, i.e., $B\cap \bigcup_{n\ge0}f^{*n}(A)\ne\emptyset$ for every nonempty open sets $A,B\subset U$, then $\bigcup_{n\ge0}f^{-n}(U)$ is a Baire ergodic component.
\end{Lemma}
\begin{proof}
Writing $V=\bigcup_{n\ge0}f^{-n}(U)$, we have that $f^{-1}(V)=\bigcup_{n\ge1}f^{-n}(U)\subset V$.
That is, $f^{-1}(V)\setminus V=\emptyset$.
On the other hand, as $f$ is non-singular and $f(U)\setminus U\sim\emptyset$, we get that $$V\setminus f^{-1}(V)=\left( U\cup\bigcup_{n\ge1}f^{-n}(U)\right)\setminus \bigcup_{n\ge1}f^{-n}(U)= U\setminus \bigcup_{n\ge1}f^{-n}(U)\subset$$
$$\subset U\setminus f^{-1}(U)\subset f^{-1}(f(U))\setminus  f^{-1}(U)=
f^{-1}(f(U)\setminus U))\sim\emptyset,$$
proving that $f^{-1}(V)\triangle V\sim\emptyset$, i.e., $V$ is an almost invariant set.

Let $L\subset V$ be a given almost invariant measurable fat set.
Thus, $T:=\pi(L)\cap V\sim L$ is an almost invariant open subset of $V$.
By the definition of $V$, there exists $a\ge0$ such that $f^a(T)\cap U$ is a fat set. 
As $f$ continuous and non-singular, it follows from Proposition~\ref{PropositionRESUMO} at Appendix that $f^{*a}(T)\sim\interior f^{*a}(T)\ne\emptyset$.
Thus $W:=\interior(f^{*a}(T))\cap U\ne\emptyset$.
By hypothesis, $A\cap\bigcup_{n\ge0}f^{*n}(W)\ne\emptyset$ for every nonempty open set $A\subset U$.
Since, by Proposition~\ref{PropositionRESUMO},  $$\interior(f^{*n}(W))\sim f^{*n}(W)\subset\overline{\interior(f^{*n}(W))},$$ we can conclude that $M:=U\cap \bigcup_{n\ge0}\interior(f^{*n}(W))$ is a dense subset of $U$.
As a consequence, $\bigcup_{j\ge0}f^{-j}(M)$ is an open and dense subset of $V$.
That is, $V\sim \bigcup_{j\ge0}f^{-j}(M)$.
Now, since $f$ is non-singular, and $T\sim f^{-1}(T)$, we get that $W\setminus T\sim\emptyset$ and so, $\interior(f^{*n}(W))\setminus T\sim\emptyset$ $\forall n\ge0$. 
Thus, $M\setminus T\subset \left(\bigcup_{n\ge0}\interior(f^{*n}(W))\right)\setminus T\sim\emptyset$ and so, 
$$\left(\bigcup_{j\ge0}f^{-j}(M)\right)\setminus T\sim \left(\bigcup_{j\ge0}f^{-j}(M\setminus T)\right)\sim\emptyset,$$
proving that $V\sim T$, as $T\subset V$.
Since $L\sim T\sim V$ for any given  almost invariant fat set $L\subset V$, we conclude that $V$ is a Baire ergodic component for $f$.
\end{proof}

One can use Lemma~\ref{Lemmakjbidtx} above to provide more examples of non transitive Baire ergodic maps. A trivial example of such maps is the following.
Given a continuous non-singular transitive map $h:\XX\circlearrowleft$, consider the  continuous and non-singular map $g:\XX\times\{1,2\}\circlearrowleft$, defined by $g(x,j)=(x,1)$.
Note that $g$ is not transitive and, by Lemma~\ref{Lemmakjbidtx}, it is Baire ergodic.

For a nontrivial example, consider a quadratic map $f:[0,1]\circlearrowleft$, $f(x)=4tx(1-x)$, with a parameter $0<t<1$ such that $f$ has a cycle of intervals. That is, there exists closed interval $I_1,\cdots,I_{\ell}\subset[0,1]$ such that $f|_{J}$ is transitive, where $J=I_1\cup\cdots\cup I_{\ell}$.
It is well known that, given a small $\varepsilon>0$, we can choose $t\in(0,1)$ so that $\leb(J)<\varepsilon$ for some circle of intervals $J$.
On the other hand, we always have that $\leb\left(\bigcup_{n\ge0}f^{-n}(J)\right)=1$.
In particular, $V:=\bigcup_{n\ge0}f^{-n}(\interior(J))$ is an open and dense subset of $[0,1]$.
Since $f(J)=J$, we have $f(\interior(J))\subset f(J)=J\sim\interior J$, i.e., $\interior J$ is an almost forward invariant open set.
Hence, it follows from  Lemma~\ref{Lemmakjbidtx} that $V$ is a Baire ergodic component of $f$ and so, since $V\sim[0,1]$, $f$ is Baire ergodic (and it is not transitive).  

\begin{Proposition}[The topological attractor of a $u$-Baire ergodic component]\label{Proposition-u-ergAttrac} 
If $U\subset \XX$ is a $u$-Baire ergodic component of $f$, then there exists a unique topological attractor $A\subset\overline{U}$ attracting a residual subset of $U$.
Indeed, $\omega_f(x)=A$ for a residual set of points $x\in U$.
Furthermore, if $U$ is not a Baire ergodic component of $f$ then $A$ is a meager set.
\end{Proposition}
\begin{proof} The proof is similar to the proof of Proposition~\ref{Propositiontop-ergodicAttractors}. 
Indeed, let $\widetilde{\XX}$ and $\psi$ be as in Lemma~\ref{LemmalOMEGAmeasu}
and set $\widetilde{U}:=\bigcup_{n\ge0}f^{-n}\left(\bigcap_{j\ge0}f^{-j}(\pi(U))\right)$.
It follows from Lemma~\ref{Lemmagdighfyt5r8} that $\widetilde{U}\subset \widetilde{\XX}$ is a Baire subspace of $\XX$, a measurable set, $\widetilde{U}\sim U$ and $f^{-1}(\widetilde{U})=\widetilde{U}=W_f^s(\widetilde{U})$.

By Lemma~\ref{LemmalOMEGAmeasu}, $\psi$ is an invariant Baire potential on $\widetilde{\XX}$. Moreover, since $\omega_f(x)=\omega_f(y)$ for every $y\in W_f^s(x)$,  we get that $\psi|_{\widetilde{U}}$ is a $u$-Baire potential for $f|_{\widetilde{U}}$.
Thus, we can apply Proposition~\ref{Propositionkjuiiiifg7655} to $\psi|_{\widetilde{U}}$ and conclude that there exists $A\in\KK(\XX)$ such that $\omega_f(x)=A$ for a residual set of points $x\in\widetilde{U}$.
Hence, $\omega_f(x)=A$ for a residual set of points $x\in U$, since $U\sim\widetilde{U}$.

If $A$ is not meager then, since $A$ is compact, $\interior A\ne\emptyset$ and, as $\omega_f(x)=A$ for a residual set of points $x\in U$, we get that
\begin{equation}\label{Equatuoibv}
  \co_f^+(x)\cap\interior A\ne\emptyset\text{ for a residual set of points }x\in U.
\end{equation}
In particular, $\co_f^+(x)\cap\,\interior A\ne\emptyset$ for a residual set of points $x\in \interior A$ and so, $\overline{\co_f^+(x)}=A$ for a residual set of $x\in A$.
This implies that, given nonempty open sets $B_0,B_1\subset\interior A$ then $B_0\cap\bigcup_{n\ge0}f^{*n}(B_1)\ne\emptyset$, proving that $\interior A$ is transitive.
Furthermore, it follows from \eqref{Equatuoibv} that $U\sim\bigcup_{n\ge0}f^{-n}(\interior A)$.
Since, by Lemma~\ref{Lemmakjbidtx}, $\bigcup_{n\ge0}f^{-n}(\interior A)$ is an Baire ergodic component of $f$, we conclude that if $A$ is a fat set then $U$ is a Baire ergodic component of $f$.
That is, if $U$ is not a Baire ergodic component then $A$ is a meager set.
\end{proof}

\subsection{Statistical attractors for Baire and $u$-Baire ergodic components}\label{SectionSTATAT}

Milnor's definition of attractors deals only with the topological aspects of the asymptotical behavior of the orbits of a fat set of points, saying little about the statistical properties of those points.
To analyze the region that is frequently visited by a large set of points, a variation of Milnor's definition called {\em statistical attractor} was introduced by Ilyashenko (see, for instance, \cite{Ily}).

As defined in Section~\ref{SecStatOfMainsR}, the (upper) {\bf\em   visiting frequency} of $x\in\XX$ to $V\subset\XX$ is given by
\begin{equation}\label{Equationhgfdfb345}
  \tau_x(V)=\tau_{x,f}(V)=\limsup_{n\to\infty}\frac{1}{n}\#\{0\le j<n\,;\,f^{j}(x)\in V\}.
\end{equation}
and the {\bf\em statistical $\omega$-limit set} of $x\in\XX$ as
$$\omega_f^{\star}(x)=\{y\,;\,\tau_x(B_{\varepsilon}(y))>0\text{ for all }\varepsilon>0\}.$$
According to Ilyashenko (see page 148 of \cite{AAIS}), the {\bf\em statistical basin of attraction} of a compact set $A\subset\XX$ is defined as  $$\beta_f^{\star}(A)=\{x\,;\,\emptyset\ne\omega_f^{\star}(x)\subset A\}.$$

If $M$ is a compact Riemannian manifold and $h:M\circlearrowleft$ is a continuous map, a compact set $A\subset M$ is called a {\bf\em Ilyashenko's statistical attractor} for $h$ when  $\leb(\beta_h^\star(A))>0$ and there is no compact set $A' \subsetneqq A$ such that $\leb(\beta_h^\star(A)\setminus\beta_h^\star(A'))=0$.
Combining Ilyashenko's definition of a statistical attractor with Milnor's definition of a topological attractor, we define the topological statistical attractor as follows $($\footnote{ In \cite{Ca}, the Ilyashenko's statistical attractor is discussed, and an interesting variation of such a metrical  attractor is presented.}$)$.

\begin{Definition}[Topological statistical attractor]\label{DefTopStsAtt}
A compact set $A\subset\XX$ is called a {\em topological statistical attractor} for the map $f:\XX_0\to\XX$ when $\beta_f^{\star}(A)$ and $\beta_f^{\star}(A)\setminus\beta_f^{\star}(A')$ are fat sets for every compact set $A' \subsetneqq A$.
\end{Definition}

A natural approach to prove the existence of a (topological) statistical attractor for a Baire or $u$-Baire ergodic component  is to follow the proof of Proposition~\ref{Propositiontop-ergodicAttractors}, that is, showing that $\omega_f^{\star}:\XX\ni x\mapsto\omega_f^{\star}(x)\in\KK(\XX)$ is a Baire potential and, as $\omega_f^{\star}$ is $f$-invariant,  applying Proposition~\ref{Propositionkjuytf9tfg76} to conclude that $\omega_j^{\star}(x)$ is almost constant. 
Our proof here will be slightly different; therefore, we present the {\em statistical spectrum} of a point, which will have other applications throughout the paper.

\subsubsection{The statistical spectrum}\label{Subsubsecststspec}

Let $\cm^1(\XX)$ the set of all Borel probability measures on $\XX$.
Let $\cd=\{\varphi_1,\varphi_2,\varphi_3,\cdots\}$  be a countable dense subset of $C(\XX,[0,1])$ and 
\begin{equation}\label{eqdistprob}
  \dd(\nu,\mu)=\sum_{n=1}^{+\infty}\frac{1}{2^n}\bigg|\int\varphi_n d\mu-\int\varphi_n d\nu\bigg|.
\end{equation}

It is well know that $\dd$ is a metric on $\cm^{1}(\XX)$ compatible with the weak topology and $(\cm^1(\XX),\dd)$ is a compact metric space.
Let $\KK(\cm^1(\XX))$ be the set of all nonempty compact subsets of $\cm^1(\XX)$ and consider the Hausdorff metric $\dd_H$ on it.
Note that $(\KK(\cm^{1}(\XX)),\dd_H)$ is also a complete metric space.

Let $\widetilde{\XX}=\bigcap_{j\ge0}f^{-j}(\XX)=f^{-1}(\widetilde{\XX})$ and define the map $\Agemo_f:\widetilde{\XX}\to\KK(\cm^{1}(\XX))$, where  $\Agemo_f(x)$ is the set of all accumulation points of  the {\em empirical measures} generated by $x$, $\left\{\frac{1}{n}\sum_{j=0}^{n-1}\delta_{f^j(x)}\right\}_{n\in\NN}$,  in the weak$^{\star}$ topology, i.e., \begin{equation}\label{EquationAgemo}
  \Agemo_f(x)=\bigg\{\mu\in\cm^1(\XX)\,;\,\dd\bigg(\frac{1}{n_k}\sum_{j=0}^{n_k-1}\delta_{f^j(x)},\mu\bigg)\to0\text{ for some sequence }n_k\nearrow+\infty\bigg\}
\end{equation}
The set $\Agemo_f(x)$ is the {\bf\em statistical spectrum} of $x$ by $f$.

\begin{Lemma}\label{LemmaAGEMOmensu}
$\Agemo_f$ is a measurable map. 
\end{Lemma}
\begin{proof}
Given $x\in\XX$, let $\mu_n(x)=\frac{1}{n}\sum_{j=0}^{n-1}\delta_{f^j(x)}$, $K_{\ell,t}(x)=\bigcup_{j=\ell}^{\ell+t}\{\mu_j\}\in\KK(\cm^{1}(\XX))$, $K_{\ell}(x)=\overline{\bigcup_{j\ge\ell}\{\mu_j\}}\in\KK(\cm^{1}(\XX))$. As $\XX\ni x\mapsto K_{\ell,t}(x)\in\KK(\cm^{1}(\XX))$ is a continuous map, hence measurable, and $K_{\ell}(x)=\lim_{t\to\infty}K_{\ell,t}(x)$, we get that $\XX\ni x\mapsto K_{\ell}(x)\in\KK(\cm^{1}(\XX))$ is measurable.
Moreover, as $\Agemo_f(x)=\lim_{\ell\to\infty}K_{\ell}(x)=\bigcap_{\ell\ge1}K_{\ell}(x)\in\KK(\cm^{1}(\XX))$, we conclude that $\Agemo_f$ is measurable.
\end{proof}

\begin{Lemma}\label{LemmaEquationjhvyurd8unm} If $X$ is a Borel subset of $\XX$ and $g:X\circlearrowleft$ is a measurable map then 
$\omega_g^{\star}(x)=\overline{\bigcup_{\mu\in\Agemo_g(x)}\supp\mu}$\; for every $x\in X$.
\end{Lemma}
\begin{proof}Since $g(X)\subset X\subset\XX$ and $\XX$ is compact, we have that $\omega_g^{\star}(x)$ is a nonempty compact subset of $\XX$ and $\Agemo_g(x)$ is a nonempty compact subset of $\cm^1(\XX)$ for every $x\in X$ (nevertheless, we may have $\omega_g^{\star}(x)\not\subset X$ and $\Agemo_g(x)\not\subset\cm^1(X)$).
As $\supp\mu\subset\omega_g^{\star}(x)$ for every $\mu\in\Agemo_g(x)$, we get that $\omega_g^{\star}(x)\supset\bigcup_{\mu\in\Agemo_g(x)}\supp\mu$.
Moreover, $\omega_g^{\star}(x)\supset\overline{\bigcup_{\mu\in\Agemo_g(x)}\supp\mu}$, since $\omega_g^{\star}(x)$ is a compact. 
Conversely, if $p\in\omega_g^{\star}(x)$ and $\varepsilon>0$ then $\mu_n(B_{\varepsilon}(p))\ge\tau(B_{\varepsilon}(p))/2$ for infinitely many $n\in\NN$. Thus, there exists $\mu_{\varepsilon}\in\Agemo_g(x)$ such that $\mu_{\varepsilon}(B_{\varepsilon}(p))$. In particular, $\supp\mu_{\varepsilon}\cap B_{\varepsilon}(p)\ne\emptyset$ for every $\varepsilon>0$, proving that $p\in\overline{\bigcup_{\mu\in\Agemo_g(x)}\supp\mu}$.
Hence, $\omega_g^{\star}(x)\subset\overline{\bigcup_{\mu\in\Agemo_g(x)}\supp\mu}$.
\end{proof}

\begin{Corollary}
\label{CorLemmaEquationjhvyurd8unm}
$\omega_f^{\star}(x)=\overline{\bigcup_{\mu\in\Agemo_f(x)}\supp\mu}$\; for every $x\in\XX$.
\end{Corollary}
\begin{proof}
If $x\notin\widetilde{\XX}$ then $\omega_f(x)=\omega_f^{\star}(x)=\emptyset$ as well as $\Agemo_f(x)=\emptyset$.
Thus, we can assume that $x\in\widetilde{\XX}$.
In this case, taking $g=f|_{\widetilde{\XX}}$, we have that $\omega_f^{\star}(x)=\omega_g^{\star}(x)$ and $\Agemo_f(x)=\Agemo_g(x)$.
Thus, it follows from Lemma~\ref{LemmaEquationjhvyurd8unm} applied to $g$ that $\omega_f^{\star}(x)=\overline{\bigcup_{\mu\in\Agemo_f(x)}\supp\mu}$.
\end{proof}

\begin{Proposition}[The topological statistical attractor of a $u$-Baire ergodic component]\label{PropositionStatisticalAttractorsU-Baire}
If $U\subset \XX$ is a $u$-Baire ergodic component of $f$, then there exists a unique topological statistical  attractor $\ca\subset\overline{U}$ attracting (statistically) a residual subset of $U$.
Moreover, $\omega_f^{\star}(x)=\ca$ for a residual set of points $x\in U$ and $\ca\subset A$, where $A$ is the topological attractor of $U$ (given by Proposition~\ref{Proposition-u-ergAttrac}).
\end{Proposition}

\begin{proof} 
Note that $\Agemo_f(f(x))=\Agemo_f(x)=\Agemo_f(y)$ for every $x\in\widetilde{\XX}:=\bigcap_{n\ge0}f^{-n}(\XX)$ and $y\in W_f^s(x)$.
Thus, since $\Agemo_f:\widetilde{\XX}\to\KK(\cm^1(\XX))$ is measurable map (Lemma~\ref{LemmaAGEMOmensu}) and $\KK(\cm^1(\XX))$ is a completed separable metric space, we get that $\Agemo_f$ is a $u$-Baire potential for $f|_{\widetilde{\XX}}$.

Since $U$ is a $u$-Baire ergodic component, $U$ is a fat set and $f^{-1}(U)\sim U=W_f^s(U)$.
Thus, it follows from Lemma~\ref{Lemmagdighfyt5r8} that 
$\widetilde{U}=\bigcup_{j\ge0}f^{-j}\big(\bigcap_{n\ge0}f^{-n}(\pi(U))\big)\subset\widetilde{\XX}$ is a fat measurable set, $\widetilde{U}\sim U$,  $f^{-1}\big(\widetilde{U}\big)=\widetilde{U}=W_f^s\big(\widetilde{U}\big)$ and $\widetilde{U}$ is a Baire subspace of $\XX$.

Hence, it follows from Proposition~\ref{Propositionkjuiiiifg7655}, applied to $f|_{\widetilde{U}}$ and $\Agemo_f|_{\widetilde{U}}$, that there exists $\cu\in\KK(\cm^1(\XX))$ such that $$\Agemo_f(x)=\cu$$
for a residual set of points $x\in \widetilde{U}$.
As, by Corollary~\ref{CorLemmaEquationjhvyurd8unm}, 
$\omega_f^{\star}(x)=\overline{\bigcup_{\mu\in\Agemo_f(x)}\supp\mu}$, we get that $$\omega_f^{\star}(x)=\ca:=\overline{\bigcup_{\mu\in\cu}\supp\mu}$$
for a residual set of points $x\in \widetilde{U}\sim U$.
\end{proof}

\section{Applications of the ergodic formalism and proofs of mains theorems}\label{SectionAplications}

We begin this section using Proposition~\ref{Propositionkjuytf9tfg76} to prove a generalization of Theorem~\ref{mainTheoTrans}.

\begin{Theorem}\label{Theoremkviduro}
Let $X$ be a Baire space. If a non-singular map $f:X\circlearrowleft$ is  continuous and transitive then the following statements are true.
\begin{enumerate}
\item Given a Borel measurable function $\varphi:X\to\RR$ with $\limsup_{n}\frac{1}{n}\sum_{j=0}^{n-1}\varphi\circ f^{j}(x)\in\RR$ for every $x\in X$, there exists $r\in\RR$ such that $\limsup_{n}\frac{1}{n}\sum_{j=0}^{n-1}\varphi\circ f^{j}(x)=r$ for a residual set of points $x\in X$.
\item Given a Borel set $A\subset X$, there exists $\theta\in[0,1]$ such that $\limsup_{n}\frac{1}{n}\#\{0\le j<n\,;\,f^{j}(x)\in A\}=\theta$ for a residual set of points $x\in X$.
\end{enumerate}
\end{Theorem}
\begin{proof}
Since $X$ is a Baire space and $f$ is continuous, non-singular and transitive, it follows from Lemma~\ref{Lemma98uyhji} that $f$ is Baire ergodic.
Thus, both items above follows from items {\em (3)} and {\em (4)} of Proposition~\ref{Propositionkjuytf9tfg76}.
\end{proof}

\begin{proof}[\bf Proof of Theorem~\ref{mainTheoTrans}]
Since a compact metric space is a Baire space, Theorem~\ref{mainTheoTrans} is a particular case of Theorem~\ref{Theoremkviduro} above.
\end{proof}

\begin{proof}[\bf Proof of Theorem~\ref{Theoremjhvuvpb}]
Let $\mathfrak{I}(f)$ be the sub $\sigma$-algebra of all $f$ invariant Borel sets and
$\mathfrak{m}:\mathfrak{I}(f)\to[0,1]$ given by $$\mathfrak{m}(U)=\mu\left(\bigcup_{n\ge0}\interior(f^n(\pi(U)))\right),$$ where $\pi$ is the Baire projection (Definition~\ref{DefBaireProj}).
Given $U,V\in\mathfrak{I}(f)$, it follow from  Proposition~\ref{PropositionRESUMO} at Appendix that $U$ is residual in the open set $\interior(f^n(\pi(U)))$ and $V$ is residual in the open set $\interior(f^{\ell}(\pi(V)))$ for every $n,\ell\ge0$.
So, if $\interior(f^n(\pi(U)))\cap \interior(f^{\ell}(\pi(V)))\ne\emptyset$ for some $n,\ell\ge0$, we get that $U\cap V\not\sim\emptyset$.
That is,
\begin{equation}\label{Equagtiiu543}
  U\cap V\sim\emptyset\implies \interior(f^n(\pi(U)))\cap \interior(f^{\ell}(\pi(V)))=\emptyset,\; \forall n,\ell\ge0.
\end{equation}

Since, $\pi(U)\cup\pi(V)=\pi(U\cup V)$ always, it follows from \eqref{Equagtiiu543} that $\mathfrak{m}(U)+\mathfrak{m}(V)=\mathfrak{m}(U\cup V)$ $\forall U,V\in\mathfrak{I}(f)$ with $U\cap V\sim\emptyset$, proving that $\mathfrak{m}$ is a Baire $f$-function.

Taking $\gamma:=\inf\left\{\mu\left(\bigcup_{n\ge0}\interior\big(f^n(B_{\varepsilon}(x))\big)\right)\,;\,x\in\XX\text{ and }\varepsilon>0\right\}$ and $\NN\ni\ell\ge1/\gamma$, it follows from the theorem hypothesis that $\mathfrak{m}(U)\ge1/\ell$ for every non meager set $U\in\mathfrak{I}(f)$.
Hence, it follows from Proposition~\ref{PropositionBaireProjectionCriterium} that $\XX$ can be decomposed (up to a meager set) into a collection $U_1,\cdots, U_{\ell}$ of Baire ergodic components.
Using Proposition~\ref{Propositiontop-ergodicAttractors}, each ergodic component $U_j$ has a topological attractor $A_j$ such that $\beta_f(A_j)\sim U_j$ and $\omega_f(x)=A_j$ for a residual set of points $x\in\beta_f(A_j)$, proving item {\em (1)}. 
Items {\em (2)} and {\em (3)} follow from Corollary~\ref{Corollaryjuytf9tfg76} applied to each $U_j$.
\end{proof}

\subsection{Maps with abundance of  historic behavior}\label{SectionMWAHB}
 Let $\XX_0$ be a measurable subset of a compact metric space $\XX$ and $f:\XX_0\to\XX$ a measurable map.
As observed in Section~\ref{SecStatOfMainsR}, a point $x\in\bigcap_{n\ge0}f^{-n}(\XX)$ has historic behavior when $\frac{1}{n}\sum_{j=0}^{n-1}\varphi\circ f^j(x)$ does not converge for some continuous function $\varphi:\XX\to\RR$. This means that $\frac{1}{n}\sum_{j=0}^{n-1}\delta_{f^j(x)}$ does not converge in the weak$^{\star}$ topology,  or equivalently, that $\#\Agemo_f(x)\ge2$.
Let us  denote the set of $x\in\XX$ with historic behavior by $HB(f)$.

We say that a map $f$ has {\em abundance of historic behavior} if $HB(f)$ is a fat set.
As it was proved by Dower \cite{Do} and Takens \cite{Ta08}, a generic point in a basic set of an Axiom A diffeomorphism has historic behavior. This holds true for a generic point in the basin of attraction of a non-periodic transitive hyperbolic attractor as well.
In essence, the presence of a non-periodic transitive hyperbolic attractor implies the abundance of historic behavior!
There exists an extensive bibliography about historic behavior (for instance, see  \cite{BKNRS,CV,CCSV,CTV,DOT,EKS,FKO,FV,Ga,KS16,KS17,LR,LST,LW,MY,NKS,Ta95,Th,Ti,Ya}), in particular about the topological entropy and Hausdorff dimension of the set of points with historic behavior.
Pesin and Pitskel \cite{PP} showed that, in the full shift $\sigma:\Sigma_2^+\circlearrowleft$, the topological entropy of the set of with historic behavior is equal to the entropy of the whole system, i.e., $h_{top}(\sigma|_{HB(\sigma)})=h_{top}(\sigma)=\log2$.
Barreira and Schmeling \cite{BS} showed that the full Hausdorff dimension of  $HB(\sigma)$ in the shift space $\Sigma_2^+$.

  A well-known example of dynamics having an open set of points with historic behavior is Bowen's Eye (Figure~\ref{BowensEye.png}), attributed to Bowen by Takens \cite{Ta95}.
  A somehow old question was whether it would be possible to regularize the oscillations of the averages along the orbit of  the points by taking higher order averages.
  Nevertheless, Jordan, Naudot and Young \cite{JNY} showed, using a classical result from Hardy
  \cite{Ha}, that if time averages
    $\frac1n\sum_{j=0}^{n-1}\varphi\circ f^j(x)$ of a bounded function  $\varphi:\XX\to\RR$ do not converge, then
  {\em all higher order averages (Césaro or H\"older) do not exist either}.

\begin{figure}
\begin{center}\includegraphics[scale=.15]{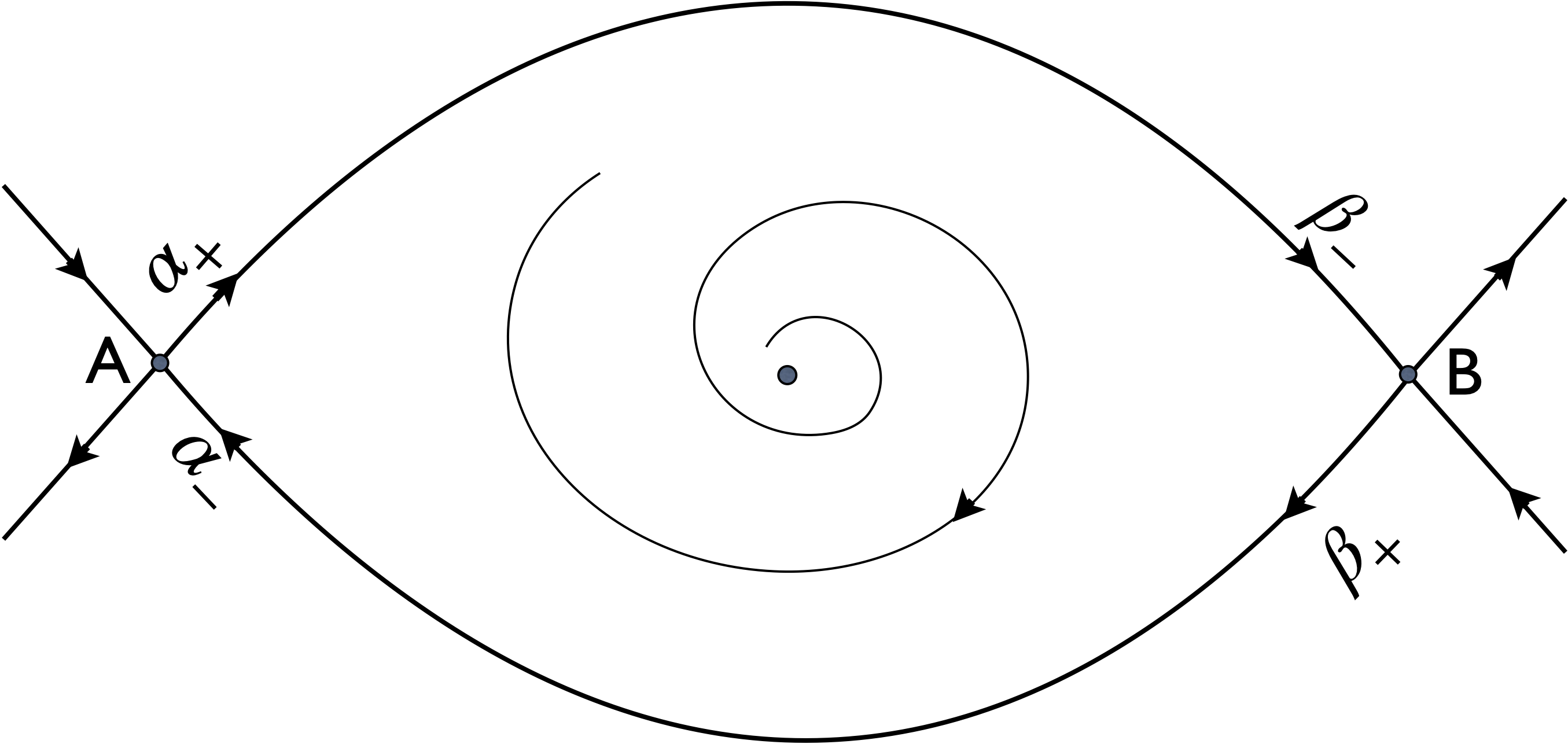}
\caption{A planar flow with divergent time averages attributed to Bowen.}\label{BowensEye.png}
\end{center}
\end{figure}

In \cite{ArP}, the authors used  a (Caratheodory) metric measure constructed from the pre-measured $\tau_x$, the upper  visiting frequency (defined by Equation~ \ref{Equationhgfdfb345}), to associate an invariant measure $\eta_x$ with each point $x$ in the phase space of a given dynamical system.
If $\mu$ is an ergodic invariant probability measure then $\eta_x=\mu$ for $\mu$-almoste every $x$.
In the case of the Bowen's Eye flow, for all wandering points $x$ (an open and dense set with full Lebesgue measure), $\eta_x$ is exactly the expected measure if we were able to regularize the Birkhoff averages.
Indeed, $\eta_x=\eta:=(\frac{|\alpha_-|}{|\alpha_-|+\beta_+})\delta_A+(\frac{|\beta_-|}{|\beta_-|+\alpha_+})\delta_B$ for every wandering point $x$,
where $A$ and $B$ are the saddle singularities of the flow and
 $\alpha_{\pm},\beta_{\pm}$ being  the eigenvalues of $A$ and $B$ (see Figure~\ref{BowensEye.png}) (\footnote{ The condition given by Takens (\cite{Ta95}) to assure the divergence of the time averages is $\big(\frac{|\alpha_-|}{|\alpha_-|+\beta_+}\big)\big(\frac{|\beta_-|}{|\beta_-|+\alpha_+}\big)>1$ and this implies that $2>\big(\frac{|\alpha_-|}{|\alpha_-|+\beta_+}\big)+\big(\frac{|\beta_-|}{|\beta_-|+\alpha_+}\big)>1$, showing that $\eta$ is a finite measure, but not a probability measure.}).
That is, in the ``Caratheodory sense'' one can regularize the Bowen's Eye.
Nevertheless, it was shown in \cite{ArP} that hyperbolicity may imply not only abundance of historic behavior, but also abundance of {\em wild historic behavior}.
A point $x$ has {\bf\em  wild historic behavior} when $\eta_x(U)=\infty$ for every nonempty open set $U$.

\begin{theorem}[\cite{ArP}]\label{TheoremWildHistBeh}
 \label{mthm:genericallywild}
    The set of points with wild historic behavior in 
\begin{enumerate}[(i)]
\item every strongly transitive topological one-sided  Markov chain with a
      denumerable set of symbols $($\footnote{ The original hypothesis of item {\em (i)} at Theorem~\ref{TheoremWildHistBeh} in \cite{ArP} is that the Markov chain is topologically mixing, nevertheless a strongly transitive map with a periodic point can be decomposed into a finite collection of disjoints sets such that the first return map to one of those sets are topologically exact and so, topologically mixing.}$)$;
\item every open continuous transitive and positively expansive map of a compact metric space;
\item each local homeomorphism defined on an open dense subset of a compact space admitting an induced full branch Markov map;
\item the support of a non-atomic expanding invariant probability measure $\mu$ for a $C^{1+}$ local diffeomorphism away from a non-flat critical/singular set on a compact manifold;
\item the basin of attraction $\beta_f(\Lambda)$ of any transitive hyperbolic attractor $\Lambda$, except when $\Lambda$ is an attracting periodic orbit;
 \end{enumerate}
is a topologically generic subset (denumerable intersection of open and dense subsets).	
\end{theorem}

Although the theorem above shows a very complicated and unpredictable behavior for the forward orbit of generic points in most of the well known dynamical systems, we can use Baire and $u$-Baire ergodicity to extract statistical information about systems with abundance of historic behavior or even with abundance of wild historic behavior.
Indeed, the maps of items {\em (i)} to {\em (iv)} above are strongly transitive, the map of item {\em (v)} is strongly $u$-transitive (see the definition in Section~\ref{SectionStngTrans} below) and, as one can see in the next section, we determine the topological statistical attractors and calculate the (upper) Birkhoff averages of any continuous function along  the orbits of generic points with historic behavior for such maps (see Theorem~\ref{TheoremUEorHB} and \ref{TheoremUEorHBinject} and Corollary~\ref{CorollaryAxA} below).

\subsection{Strongly transitive maps}\label{SectionStngTrans}
 Strongly transitive maps (or sets)  appears profusely in dynamics $($\footnote{ All transitive, continuous and piecewise monotone interval maps, expanding maps of a connected compact manifold, transitive circle homeomorphisms, transitive translations of a compact metrizable topological group,  the shift map $\sigma:\Sigma_n^+\circlearrowleft$, $n\ge2$ and Viana's maps are examples of strongly transitive maps.
Moreover, one can use $f|_{\ca_j}$ of item {\em (5)} in Theorem~\ref{TheoremFatErgodicAttractors} to produce many examples of strongly transitive maps. See also \cite{PV} to more examples and properties of strongly transitive maps.}$)$
and Theorem~\ref{TheoremUEorHB} presents a dichotomy for those maps, a strongly transitive map is either uniquely ergodic or has abundance of historic behavior.
Moreover, this theorem shows a strong connection between the statistical behavior of generic orbits and the set of the invariant probability measures.

Given a Baire metric space $\XX$ and an open set $U\subset\XX$, we say that a continuous map  $f:\XX\circlearrowleft$ is a {\bf\em strongly transitive on $U$} if $\bigcup_{n\ge0}f^n(V)\supset U$ for every open set $V\subset U$ (this means that $\overline{\co_f^-(x)}\supset U$ for every $x\in U$). 
In the spirit of $u$-Baire ergodicity, a continuous map $f:\XX\circlearrowleft$ is called {\bf\em strongly $u$-transitive} on an open set $U\subset\XX$ when $\overline{W_f^s(\co_f^-(x))}\supset U$ for every $x\in U$.
Of course that all strongly transitive maps are strongly $u$-transitive, as $\co_f^-(x)\subset\co_f^-(W_f^s(x))=W_f^s(\co_f^-(x))$.

Let us denote {\bf\em the set of all ergodic invariant Borel probability measures of $f$} by $\cm_e^1(f)$. Note that, if  $\mu\in\cm_e^1(f)$ then $\beta_f(\mu)\ne\emptyset$, as $\mu(\beta_f(\mu))=1$.

\begin{Proposition}\label{PropositionMataLeao}
Let $\XX$ be a separable Baire metric  space and $f:\XX\circlearrowleft$ a continuous map.
If we take $\XX_0=\{x\in\XX\,;\,\overline{\co_f^-(W_f^s(x))}=\XX\}$ then $\Agemo_f(x)\supset\{\mu\in\cm^1(f)\,;\,\beta_f(\mu)\cap\XX_0\ne\emptyset\}\supset\{\mu\in\cm_{e}^1(f)\,;\,\mu(\XX_0)>0\}$ for a residual set of points $x\in\XX$.
\end{Proposition}
\begin{proof}
	Let $\dd$ be a metric on $\cm^1(\XX)$ compatible with the weak$^{\star}$ topology.
Given $x\in\XX$ and $\ell\in\NN$, let $\delta_{\ell,x}=\frac{1}{\ell}\sum_{j=0}^{\ell-1}\delta_{f^j(x)}\in\cm^1(\XX)$.
Consider any $\mu\in\cm^1(f)$, with $\beta_f(\mu)\cap\XX_0\ne\emptyset$.
Let $p\in\beta_f(\mu)\cap\XX_0$, that is, $\Agemo_f(p)=\{\mu\}$ and $\overline{\co_f^-(W_f^s(p))}=\XX$.
As $\overline{\co_f^-(W_f^s(p))}=\XX$ and $\Agemo_f(y)=\Agemo_f(p)$ for every $y\in\co_f^-(W_f^s(p))$, we get that $\beta_f(\mu)$ is a dense set in $\XX$. 

Given $r>0$ and $n\in\NN$, let $$V(r,n)=\{x\in\beta_f(\mu)\,;\,\dd(\delta_{m,x},\mu)<r\;\forall\,m\ge n\}.$$

As $V(r,1)\subset V(r,2)\subset V(r,3)\subset\cdots$ and $\bigcup_{n\in\NN}V(r,n)=\beta_f(\mu)$, given $t\in\NN$ there is $n(t)\ge t$ such that $V(r,n(t))$ is a $(1/t)$-dense set in $\beta_f(\mu)$ and so, sinse $\beta_f(\mu)$ is dense in $\XX$, $V(r,n(t))$ is a $(1/t)$-dense set in $\XX$ (i.e., $B_{1/t}(V(r,n(t)))=\bigcup_{x\in V(r,n(t))}B_{1/t}(x)=\XX$).
As $\XX$ is separable, $V(r,n(t))$ admits a countable $(1/t)$-dense subset.
That is, there is countable set $V'(r,n(t))\subset V(r,n(t))$ such that $B_{1/t}(V'(r,n(t)))=\XX$.
It follows from the continuity of $f$  that there exists $\varepsilon(r,n(t),y)>0$ such that $$\dd(\delta_{n,x},\mu)<r\text{ for every }x\in B_{\varepsilon(r,n(t),y)}(y)\text{ and }y\in V'(r,n(t)).$$
Note that the set $W_r(m)=\bigcup_{t\ge m}\bigcup_{y\in V'(r,n(t))}B_{\varepsilon(r,n(t),y)}(y)$ is an open and $(1/m)$-dense set for every $m\in\NN$.
Moreover, if $x\in W_r(m)$ then  $\dd(\delta_{x,n},\mu)<r$  for some $n\ge m$.
Defining $W_r=\bigcap_{m\in\NN}W_r(m)$, we get that $W_r$ is residual and for each $x\in W_r$ there is $\ell_j\to+\infty$ such that $\dd(\delta_{\ell_j,x},\mu)<r$.
Finally, we have that $W(\mu)=\bigcap_{n\in\NN}W_{1/n}$ is also a residual set and $\mu\in\Agemo_f(x)$ for every $x\in W(\mu)$. 
Taking a countable and dense set $\{\mu_1,\mu_2,\mu_3,\cdots\}\subset\{\mu\in\cm^1(f)\,;\,\beta_f(\mu)\cap\XX_0\ne\emptyset \}$, we get that $W=\bigcap_{n\in\NN}W(\mu_n)$ is a residual set and, by the compactness of $\Agemo_f(x)$, also that $\Agemo_f(x)=\overline{\{\mu\in\cm^1(f)\,;\,\beta_f(\mu)\cap\XX_0\ne\emptyset \}}$ for every $x\in W$, which completes the proof. 
\end{proof}

In \cite{DGS}, Denker, Grillenberger and Sigmund called the points satisfying $\Agemo_f(x)=\cm^1(f)$ the {\em points of maximal oscillation}.
We observe that the connection between historical behavior, maximal oscillation, and {\em specification} has been studied by many authors.
Moreover, in many cases, maps with specification are strongly transitive (or strongly $u$-transitive).  In those cases, one can use Corollary~\ref{CoroMataLeao}, instead of the specification property, to show the maximal oscillation for generic points.

\begin{Corollary}\label{CoroMataLeao}
If $\XX$ is a separable Baire metric space and $f:\XX\circlearrowleft$ is a continuous strongly transitive map then the set of all points with maximal oscillation is a residual subset of $\XX$.
\end{Corollary}

\begin{proof}[\bf Proof of Theorem~\ref{TheoremFatErgodicAttractorsMAIN}]
We note that Theorem~\ref{TheoremFatErgodicAttractorsMAIN} satisfies all the hypothesis of Theorem~\ref{TheoremFatErgodicAttractors} and so, all the items of Theorem~\ref{TheoremFatErgodicAttractorsMAIN} follow directly from Theorem~\ref{TheoremFatErgodicAttractors}, with the exception of item {\em(vii)}. 

Item {\em(vii)} is a consequence of Proposition~\ref{PropositionMataLeao}.
Indeed, suppose that exist $p,q\in\ca_j\cap\per(f)$ such that $\co_f^+(p)\cap\co_f^+(q)=\emptyset$.
Let $$\mu=\frac{1}{\#\co_f^+(p)}\sum_{x\in\co_f^+(p)}\delta_{x}\;\text{ and }\;\nu=\frac{1}{\#\co_f^+(q)}\sum_{x\in\co_f^+(q)}\delta_{x}.$$
Taking $g=f|_{A_j}$, it follows from item {\em (5)} of Theorem~\ref{TheoremFatErgodicAttractorsMAIN} that  $\alpha_f(x)\supset\overline{\ca_j}=A_j$ for every $x\in\ca_j$.
This implies that $\co_f^-(x)$ is dense in $A_j$ when $x\in\ca_j$.
As $A_j$ is a forward invariant set, we get that $\co_g^-(x)$ is dense in $A_j$ for every $x\in\ca_j$.
Finally, as we always have that $\co_g^-(x)\subset \co_g^-(W_g^s(x))$, we get that $(A_j)_0:=\{y\in A_j\,;\,\overline{\co_g^-(W_g^s(x))}=A_j\}\supset\ca_j$ and so, by Proposition~\ref{PropositionMataLeao}, $\Agemo_f(x)=\Agemo_g(x)\supset\{\mu,\nu\}$ for a residual set of points $x\in A_j$.
As $\beta_f(A_j)\sim\bigcup_{n\ge0}f^{-n}(A_j)$, it follows that  $\Agemo_f(x)\supset\{\mu,\nu\}$ for a residual set of points $x\in\beta_f(A_j)$, proving that, generically, the points of $\beta_f(A_j)$ have historical behavior.
\end{proof}

\begin{proof}[\bf Proof of Theorem~\ref{TheoremUEorHB}]
Since $\XX$ is compact,  $f$ continuous and the whole $\XX$ is strongly transitive, it follows from  Proposition~\ref{PropositionMataLeao} that
\begin{equation}\label{eqlkjd567}
  \Agemo_f(x)=\cm^1(f)\ne\emptyset\text{ for a residual set of points $x\in\XX$.}
\end{equation}
Thus,  by \eqref{eqlkjd567},  if $f$ is not uniquely ergodic  then  $\#\Agemo_f(x)>1$ for a residual set of points $x\in\XX$ and so, a generic point $x\in\XX$ has historic behavior, showing item {\em (\ref{itemhgn})}.
Moreover, the poof of item {\em (\ref{itemngyj46uu})} follows  from \eqref{eqlkjd567} and Lemma~\ref{LemmaEquationjhvyurd8unm}, applied to $X=\XX$ and $g=f$.

Given $\varphi\in C^0(\XX,\RR)$ and $x\in\XX$, it follows from the convergence in the weak* topology that
$$\limsup_{n\to\infty}\frac{1}{n}\sum_{j=0}^{n-1}\varphi\circ f^j(x)=\limsup_{n\to\infty}\int\varphi\;d\bigg(\frac{1}{n}\sum_{j=0}^{n-1}\delta_{f^j(x)}\bigg)=\sup\bigg\{\int\varphi d\mu\,;\,\mu\in\Agemo_f(x)\bigg\}.$$
Moreover, by the compactness of $\cm^1(f)$, we have that 
$$
  \sup\bigg\{\int\varphi d\mu\,;\,\mu\in\cm^1(f)\bigg\}=\max\bigg\{\int\varphi d\mu\,;\,\mu\in\cm^1(f)\bigg\}.
$$
Hence, if $\varphi\in C^0(\XX,\RR)$ and $x$ is a generic point in $\XX$, we can use \eqref{eqlkjd567} to conclude that 
$$\limsup\frac{1}{n}\sum_{j=0}^{n-1}\varphi\circ f^j(x)=\sup\bigg\{\int\varphi d\mu\,;\,\mu\in\Agemo_f(x)\bigg\}=$$
$$=\sup\bigg\{\int\varphi d\mu\,;\,\mu\in\cm^1(f)\bigg\}=\max\bigg\{\int\varphi d\mu\,;\,\mu\in\cm^1(f)\bigg\},$$
proving item {\em (\ref{itemdrt})}.

To prove item {\em (\ref{itemvgjmu})}, recall that, if $\mu_n$ is a sequence of probability measures converging to $\mu$ in the weak* topology, then $\mu(\overline{U})\ge\lim_n\mu_n(\overline{U})\ge \lim_n\mu_n(U)\ge\mu(U)$ for every open set $U\subset\XX$.
Since this implies that
$$
  \max\{\mu(\overline{U})\,;\,\mu\in\Agemo_f(x)\}\ge\tau_x(\overline{U})\ge\tau_x(U)\ge\max\{\mu(U)\,;\,\mu\in\Agemo_f(x)\}$$
for every $x\in\XX$ and every open set $U\subset\XX$,
it follows from item {\em (\ref{itemdrt})} that
\begin{equation}\label{eqrhvtjnh}
  \max\{\mu(\overline{U})\,;\,\mu\in\cm^1(f)\}\ge\tau_x(U)\ge\max\{\mu(U)\,;\,\mu\in\cm^1(f)\}
\end{equation}
for a residual set of points $x\in\XX$ and every open set $U\subset\XX$.

Suppose that $f$ is non-singular. Given a Borel set $V$, let $U$ be an open set such that $V\sim U$.
By item {\em (\ref{itemvgjmu})} there exists a residual set $R\subset\XX$ such that \eqref{eqrhvtjnh} holds for every $x\in R$.
Noting that $M:=U\triangle V$ is a meager set, $f$ is non-singular  and $\tau_x(M)>0$ $\implies$ $x\in\bigcap_{n\ge0}f^{-n}(M)$, we get that $\tau_x(U)=\tau_x(V)$ for every $x\in R':=R\setminus\bigcap_{n\ge0}f^{-n}(M)\sim R$, where $R'$ is residual in $\XX$.
And this concludes the proof of item {\em (\ref{itemvgg668mu})}.
\end{proof}

An interesting example of a strongly transitive map is the ``Furstenberg's Example''.

\begin{Corollary}\label{CorollaryFurstenberg}
	If $f$ is the Furstenberg minimal analytic diffeomorphism of the torus $\TT^2$ having the Lebesgue as a non-ergodic invariant measure \cite{Fu} then, generically, the points of $\TT^2$ have historic behavior.
Furthermore, if $\varphi:\TT^2\to\RR$ is a continuous function then 
$$
  \limsup_{n\to+\infty}\frac{1}{n}\sum_{j=0}^{n-1}\varphi\circ f^j(x)=\max\bigg\{\int\varphi d\mu\,;\,\mu\in\cm^1(f)\bigg\}$$
for a residual set of points $x\in\TT^2$.
\end{Corollary}
\begin{proof}
	As $f^{-1}$ is also a minimal homeomorphism, we get that $\alpha_f(x)=\omega_{f^{-1}}(x)=\TT^2$ for every $x\in\TT^2$, proving that $f$ is a strongly transitive map.
As $f$ is not uniquely ergodic \cite{Fu}, it follows from Theorem~\ref{TheoremUEorHB} that a generic point $x\in\TT^2$ has historic behavior.
\end{proof}

\begin{proof}[\bf Proof of Theorem~\ref{TheoremUEorHBinject}]
The proof follows the same argument of proof of Theorem~\ref{TheoremUEorHB}.
\end{proof}

Knowing that periodic points are a dense subset of any transitive hyperbolic diffeomorphism, Corollary~\ref{CorollaryAxA} below follow straightforward from Theorem~\ref{TheoremUEorHBinject}. 

\begin{Corollary}\label{CorollaryAxA}
If $f:M\circlearrowleft$ is a transitive $C^1$ Anosov diffeomorphism then a generic point $x\in M$ has historic behavior, $\omega_f^{\star}(x)=\omega_f(x)=M$ and
$$
  \limsup_{n\to+\infty}\frac{1}{n}\sum_{j=0}^{n-1}\varphi\circ f^j(x)=\max\left\{\int\varphi d\mu\,;\,\mu\in\cm^1(f)\right\},$$ whenever $\varphi:M\to\RR$ is a continuous function.
\end{Corollary}

\subsection{Topologically growing maps}

Let $\XX$ be a compact metric space and $\XX_0$ an open and dense subset of $\XX$.
A non-singular continuous map $f:\XX_0\to\XX$ is called {\bf\em $\delta$-growing}, $\delta>0$, if for each nonempty open set $V\subset\XX$ there is $n\ge0$, $q\in\XX$ and a connected component $U\subset V$ of $f^{-n}(B_{\delta}(q))$ such that $f^n(U)=B_{\delta}(q)$.
A {\bf\em topologically growing map} is a $\delta$-growing map for some $\delta>0$.

An open set $V_{n,\delta}(p)$ is called a  {\bf\em pre-ball of order $n\in\NN$,  radius  $\delta>0$ for $p\in\XX$} if there is $q\in\XX$ such that
\begin{enumerate}
	\item $V_{n,\delta}(p)$ is the connected component of $f^{-n}(B_{\delta}(q))$ containing $p$ and
\item $f^n(p)\in B_{\delta/2}(q) \subset f^n(V_{n,\delta}(p))= B_{\delta}(q)$.
\end{enumerate}

We say that $n\in\NN$ is a {\bf\em $\delta$-growing time to $p\in\XX$} when there exists a pre-ball $V_{n,\delta}(p)$ for $p$. 
Let us denote by $\cg(\delta,p)\subset\NN$ the {\bf\em set of all $\delta$-growing time to $p$}.

If $n\ge2$ is $\delta$-growing time to $p$ then $n-1$ is a $\delta$-growing time to $f(p)$. 	
Indeed, if $V_{n,\delta}(p)$ is a pre-ball of order $n$ and radius $\delta$ for $p$, with $f^n(V_{n,\delta}(p))=B_{\delta}(q)$, then $V_{n-1,\delta}(f(p)):=f(V_{n,\delta}(p))$ is a pre-ball of order $n-1$ and radius $\delta$ for $f(p)$ with $f^{n-1}(V_{n-1,\delta}(f(p)))=B_{\delta}(q)$.
That is, $\cg(\delta,f(p))\supset\cg(\delta,p)-1:=\{n-1\,;\,n\in\cg(\delta,p)\}$ for every $p\in\XX$. 
As, for $r>0$, $$\mathfrak{G}_{r}(n,\delta):=\{p\in\XX\,;\,n\in\cg(\delta,p)\text{ with }\diam(V_{n,\delta}(p))<r\}$$ is an open set, if $f$ is a $\delta$-growing map, then  $$\mathfrak{G}(\delta):=\bigcap_{\ell\in\NN}\bigcap_{n\in\NN}\bigcup_{n\le m\in\NN}\mathfrak{G}_{1/\ell}(m,\delta),$$ the set of all {\bf\em points with infinity many $\delta$-growing times for arbitrarily small pre-balls}, is a residual set.

Given $x\in\mathfrak{G}(\delta)$, the {\bf\em omega-limit in $\delta$-growing time for $x$}, denoted by $\omega_{\delta,f}(x)$, is the set of all $y\in\XX$ such that $y=\lim_j f^{n_j}(x)$ with $n_j\in\cg(\delta,x)$ and $\diam(V_{n_j,\delta}(x))\to0$.
It is easy to see that $\omega_{\delta,f}(x)$ is compact, but not necessarily forward invariant, and that $\omega_{\delta,f}(x)=\omega_{\delta,f}(f(x))$ for every $x\in\mathfrak{G}(\delta)$.

Let us assume that $f$ is a $\delta$-growing map.
Hence, it  follows from Theorem~\ref{TheoremFatErgodicAttractors} that $\XX$ can be decomposed into a finite collection of Baire ergodic components $U_1,\cdots,U_{\ell}$, with $U_j$ being open sets.
Let $A_j$ be the topological attractor for $U_j$.
Also by Theorem~\ref{TheoremFatErgodicAttractors}, $A_j$ contains a ball $B_j$ of radius $\delta$.
In particular, $\interior(A_j)\ne\emptyset$ and, as $A_j$ is transitive, we get that $A_j=\overline{\interior(A_j)}$.

\begin{Lemma}\label{LemmaLAMBDAgrowing}
$\exists\Lambda_j\subset B_{\delta/2}(\Lambda_j)\subset A_j$ s.t. $\omega_{\delta,f}(x)=\Lambda_j$ residually on $U_j$.
\end{Lemma}
\begin{proof}
Let $\cg(r,\delta,x)$ be the set of all $\delta$-growing times to $x$ such that $\diam(V_{n,\delta}(x))<r$, $\mathfrak{G}_r(\delta)=\bigcap_{n\in\NN}\bigcup_{n\le m\in\NN}\mathfrak{G}_{r}(m,\delta)$ and 
 $\psi_{n,r}:\mathfrak{G}_r(\delta) \to\KK(\XX)$ be given by $\psi_{n,m,r}(x)=\{f^j(x)\,;\,j\in\{n,\cdots,n+m\}\cap\cg(r,\delta,x)\}$. 
As $\psi_{n,m,r}$ is continuous, the map $\psi_{n,r}:\mathfrak{G}_r(\delta) \to\KK(\XX)$ given by $\psi_{n,r}(x)=\lim_{m}\psi_{n,m,r}(x)=\overline{\{f^j(x)\,;\,j\ge n\,;\,j\in\cg(r,\delta,x)\}}$ is a measurable map, as it is the pointwise limit of continuous (and so, measurable) maps.
Similarly, $\psi_r$ is also measurable, where $\psi_r(x)=\lim_n\psi_{n,r}(x)=\bigcap_n\overline{\{f^j(x)\,;\,j\ge n\,;\,j\in\cg(r,\delta,x)\}}$.
Finally, as $\omega_{\delta,f}(x)=\lim_{\NN\ni\ell\to\infty}\psi_{1/\ell}(x)$, we get that $\mathfrak{G}(\delta)\ni x\mapsto\omega_{\delta,f}(x)\in\KK(\XX)$ is a measurable map and so, an invariant Baire potential.
Thus, it follows from Corollary~\ref{Corollaryjuytf9tfg76} that there exists $\Lambda_j\in\KK(\XX)$ such that $\omega_{\delta,f}(x)=\Lambda_j$ for a residual set of points $x\in U_j$.
As $A_j=\overline{\interior(A_j)}$, $\omega_{\delta,f}(x)=\Lambda_j$ for a residual set of points $x\in\interior(A_j)$.
By the definition of a $\delta$-growing time to $x$, if $y\in\omega_{\delta,f}(x)$ and $x\in\interior(A_j)$ we get that $B_{\delta/2}(y)\subset A_j$ and, as a consequence, $B_{\delta/2}(\Lambda_j)\subset A_j$.
\end{proof}

Recall that a Borel probability measure $\mu$ is called $\varphi$-maximizing measure with respect to $f$ if it is $f$-invariant and $$\int\varphi d\mu=\sup\bigg\{\int\varphi d\nu\,;\,\nu\in\cm^1(f)\bigg\}.$$
\begin{Theorem}
\label{TheoremFatErgodicAttractors2}
Let $\XX$ be a compact metric space and $\XX_0$ an open and dense subset of $\XX$.
If $f:\XX_0\to\XX$ is a $\delta$-growing map then $\XX$ can be decomposed (up to a meager set) into a finite number of Baire ergodic components $U_1,\cdots, U_\ell\subset\XX$, each $U_j$ is an open set and  the attractors $A_j$ associated to  $U_j$ satisfy the following properties for each $1\le j\le \ell$.
\begin{enumerate}[(T1)]
\item Each $A_j$ is transitive,  contains an open ball of radius $\delta$ and $A_j=\overline{\interior(A_j)}$.
\item $\overline{\Omega(f)\setminus\bigcup_{j=0}^{\ell}A_j}$ is a compact set with empty interior.
\item For each $A_j$ there is a forward invariant set $\ca_j\subset A_j$ containing an open and dense subset of $A_j$ such that $f$ is strongly transitive in $\ca_j$.
\item  $\omega_f^{\star}(x)=\omega_f(x)=A_j$ for a residual set of points $x\in U_j\sim\beta_f(A_j)$.
\item $h_{top}(f|_{\ca_j})>0$.
\item $f|_{\ca_j}$ has an uncountable set of ergodic invariant probability measures.
\item If $x$ is a generic point of $U_j$ and $\varphi\in C(\XX,\RR)$ then  
\begin{equation}\label{Equationkljknlnl}
    \limsup_{n\to+\infty}\frac{1}{n}\sum_{j=0}^{n-1}\varphi\circ f^j(x)\ge\sup\bigg\{\int\varphi d\mu\,;\,\mu\in\cm^1(f|_{\ca_j})\bigg\}.
\end{equation}
 In particular, if there is a $\varphi$-maximizing measure $\mu_j$ with respect to $f|_{A_j}$ such that $\mu_j(\ca_j)=1$ then (\ref{Equationkljknlnl}) becomes  an equality.
\end{enumerate}
Furthermore, 
\begin{enumerate}[(T1)]
\setcounter{enumi}{7}
\item $f$ has sensitive dependence on initial conditions.
\item Generically, the points of  $\XX$ have historic behavior.
\item If $\XX$ is a compact manifold (possibly with boundary) then $\overline{\per(f)}\supset\bigcup_{j=1}^{\ell}A_j$.
\end{enumerate}
\end{Theorem}
\begin{proof} The decomposition into a finite number of Baire ergodic components $U_j$, the topological attractor $A_j$, items {\em (T1), (T2), (T3)} and $\omega_f(x)=A_j$ for a residual set of points $x\in U_j$ follow directly from Theorem~\ref{TheoremFatErgodicAttractors}.

Moreover, it follows from Theorem~\ref{TheoremFatErgodicAttractors} and Proposition~\ref{PropositionMataLeao} that $\Agemo_f(x)\supset\cm^1(f|_{\ca_j})$ for a residual set of points $x\in \ca_j$, where $\ca_j=\{x\in A_j\,;\,\alpha_f(x)\supset A_j\}$ contains an open and dense subset of $A_j$.
Let $\Lambda_j\subset B_{\delta/2}(\Lambda_j)\subset A_j$ be the compact set given by Lemma~\ref{LemmaLAMBDAgrowing} such that $\omega_{\delta,f}(x)=\Lambda_j$ for a residual set of points $x\in U_j$ and consider a point $p\in B_{\delta/8}(\Lambda_j)$.
\begin{Claim}[Local horseshoes]\label{ClaimLocalHorseshoes}
Given $0<\varepsilon<\delta/4$ there exist open sets $S_0$ and $S_1$, with $\overline{S_j}\subset B_{\varepsilon}(p)$, $\overline{S_0}\cap\overline{S_1}=\emptyset$ and
integers $n_0,n_1\in\NN$ such that $\overline{S_j}$ is a connected component of $f^{-n_j}(\overline{B_{\varepsilon}(p)})$ and $f^{n_j}(\overline{S_j})=\overline{B_{\varepsilon}(p)}$.
\end{Claim}
\begin{proof}[Proof of the claim] Let $q\in B_{\delta/4}(p)\cap\Lambda_j$.
As $\mathfrak{G}(\delta)$ contains a residual set, let $p_0,p_1\in B_{\varepsilon}(p)\cap \mathfrak{G}(\delta)$ be so that $q\in \omega_{\delta,f}(p_j)$ for $j=0,1$.
Let $$r=\min\{d(p_0,p_1)/3,d(p_0,\partial B_{\varepsilon}(p))/3,d(p_1,\partial B_{\varepsilon}(p))/3\}$$
and $n_j\in\cg(r,\delta,p_j)$, $j=0,1$, so that $f^{n_j}(p_j)$ is close enough to $q$ so that $f^{n_j}(p_j)\in B_{\delta/2}(p)$.
Hence, there are pre-balls $V_{n_j,\delta}(p_j)$, $j=0,1$, with diameters smaller than $r$ such that $f^{n_j}(V_{n_j,\delta}(p_j))\supset\overline{B_{\varepsilon}(p)}$.
Let $S_j$ be the connected component of $(f^{n_j}|_{V_{n_j,\delta}(p_j)})^{-1}(B_{\varepsilon}(p))$ containing $p_j$.
Thus, $f^{n_j}(\overline{S_j})=\overline{B_{\varepsilon}(p)}$, $\overline{S_j}\subset B_r(p_j)\subset B_{\varepsilon}(p)$ and  $\overline{S_0}\cap \overline{S_1}\subset B_{r}(p_0)\cap B_r(p_1)=\emptyset$, proving the claim.
\end{proof}

Let $F_{\varepsilon}:\overline{S_0}\cup\overline{S_1}\to \overline{B_{\varepsilon}(p)}$ be the induced map given by  $F_{\varepsilon}(x)=f^{R(x)}(x)$ with $R(x)=n_j$ for $x\in \overline{S}_j$, where   $S_j$ and $n_j$ are as in Claim~\ref{ClaimLocalHorseshoes} above.
Taking $\Gamma_{\varepsilon}=\bigcap_{n\ge0}F_{\varepsilon}^{-n}(\overline{B_{\varepsilon}(p)})$,
one can use the itinerary map (i.e., $I:\Gamma_{\varepsilon}\to\Sigma_2^+$ given by $I(x)(n)=\mathbb{1}_{U_1}\circ F_{\varepsilon}^n(x)$) to obtain a semiconjugation between
$F_{\varepsilon}|_{\Gamma_{\varepsilon}}$ and the shift $\sigma:\Sigma_2^+\circlearrowleft$.
As this implies that $h_{top}(F_{\varepsilon})\ge\log2$, we get that $h_{top}(f|_{A_j})\ge\frac{1}{2}(n_0+n_1)h_{top}(F_{\varepsilon})>0$, proving item {\em (T5)}.
Likewise, it follows from the semiconjugation that $\cm_e^1(F_{\varepsilon})$ is an uncountable set and, as $\int R d\mu\le\max\{n_0,n_1\}$ for every $\mu\in\cm_e^1(F_{\varepsilon})$, we get that $\cm_{e}^1(f|_{\ca_j})$ is also an uncountable set, proving item {\em (T6)}.
The items {\em (T7)} and {\em (T9)} follow from item {\em (T3)}, {\em (T6)} and Proposition~\ref{PropositionMataLeao}.

The presence of horseshoes inside each $\ca_j$ (claim~\ref{ClaimLocalHorseshoes}) implies that there exists $x\in\ca_j$ such that $\emptyset\ne\omega_f(x)\ne A_j$ and so, by item {\em (6)} of Theorem~\ref{TheoremFatErgodicAttractors}, $f|_{A_j}$ has sensitive dependence on initial condition, $\forall1\le j\le\ell$. Since $U_j\sim\bigcup_{n\ge0}f^{-1}(A_j)$ $\forall j$ and $\XX\sim U_1\cup\cdots\cup U_{\ell}$, we get that $f$ has sensitive dependence on initial condition, showing  item {\em (T8)}.

Note that Claim~\ref{ClaimLocalHorseshoes} implies that for every $p\in B_{\delta/4}(\Lambda_j)$ and every $0<\varepsilon<\delta/2$ there is a $f$ invariant ergodic probability measure $\mu$ such that $\supp\mu\cap B_{\varepsilon}(p)\ne\emptyset$.
That is, $\overline{\bigcup_{\mu\in\cm_e^1(f|_{\ca_j})}\supp\mu}\supset B_{\delta/4}(\Lambda_j)$ and, by transitivity and compactness, this implies that $$A_j\supset\overline{\bigcup_{\mu\in\cm_e^1(f|_{\ca_j})}\supp\mu}\supset\overline{\ca_j}=A_j.$$
Hence, as $\omega_f^{\star}(x)=\overline{\bigcup_{\mu\in\Agemo_f(x)}\supp\mu}\supset \overline{\bigcup_{\mu\in\cm_e^1(f|_{\ca_j})}\supp\mu}$ for a residual set of points $x\in U_j$ (Corollary~\ref{CorLemmaEquationjhvyurd8unm}), we get that $\omega_f^{\star}(x)=A_j$ for a residual set of points $x\in U_j$, completing the proof of item {\em (T4)}.

Finally, if $\XX$ is a compact manifold (possibly with boundary), we can use Brouwer fixed-point theorem to prove that $F_{\varepsilon}$ has a fixed point on $\overline{S_1}$ (and also in $\overline{S_2}$). Thus $\per(f)\cap B_{\varepsilon}(p)$ for every $\varepsilon>0$ and every $p\in B_{\delta/4}(\Lambda_j)$.
Hence, using that $f|_{A_j}$ is transitive, we get that $\overline{\per(f)}\supset A_j$, proving item {\em (T10)} and completing the proof of the theorem.
\end{proof}

\begin{proof}[\bf Proof of Theorem~\ref{TheoremFatErgodicAttractors2MAIN}]
With the exception of item {\em (6)}, all items of Theorem~\ref{TheoremFatErgodicAttractors2MAIN} follows directly from Theorem~\ref{TheoremFatErgodicAttractors2}.
To check item {\em (6)}, observe that  each Baire ergodic component $U_j$ is almost equal to the base of attraction $\beta_f(A_j)$, i.e., $U_j\sim\beta_f(A_j)$. 
If $\varphi\in C(\XX,\RR)$, it follows from the Baire ergodicity that exists $a_+\in\RR$ such that 
$$\limsup_{n\to+\infty}\frac{1}{n}\sum_{j=0}^{n-1}\varphi\circ f^j(x)=a_+$$
for a residual set of points $x\in U_j$.
By the same reasoning, there exists $b_+\in\RR$ such that 
$$-\liminf_{n\to+\infty}\frac{1}{n}\sum_{j=0}^{n-1}\varphi\circ f^j(x)=\limsup_{n\to+\infty}\frac{1}{n}\sum_{j=0}^{n-1}-\varphi\circ f^j(x)=b_+$$
for a residual set of points $x\in U_j$.
By item {\em (T7)} of Theorem~\ref{TheoremFatErgodicAttractors2}, if $x$ is a generic point of $\beta_f(A_j)$ then, 
\begin{equation}
   a_+=\limsup_{n\to+\infty}\frac{1}{n}\sum_{j=0}^{n-1}\varphi\circ f^j(x)\ge\sup\bigg\{\int\varphi d\mu\,;\,\mu\in\cm^1(f|_{\ca_j})\bigg\}
\end{equation}
and
$$
    b_+=\limsup_{n\to+\infty}\frac{1}{n}\sum_{j=0}^{n-1}-\varphi\circ f^j(x)\ge\sup\bigg\{\int-\varphi d\mu\,;\,\mu\in\cm^1(f|_{\ca_j})\bigg\}.
$$
Writing $a_-=-b_+$, we get that  
$$\liminf_{n\to+\infty}\frac{1}{n}\sum_{j=0}^{n-1}\varphi\circ f^j(x)=a_-=-\limsup_{n\to+\infty}\frac{1}{n}\sum_{j=0}^{n-1}-\varphi\circ f^j(x)\le$$
$$\le -\sup\bigg\{\int-\varphi d\mu\,;\,\mu\in\cm^1(f|_{\ca_j})\bigg\}=\inf\bigg\{\int\varphi d\mu\,;\,\mu\in\cm^1(f|_{\ca_j})\bigg\}$$
for a residual set of points $x\in\beta_f(A_j)$.
\end{proof}

\subsection{Regular attractors and physical measures}\label{Connections}

In this section let $f:\XX\circlearrowleft$ be a continuous map defined on a compact metric space $\XX$.
Let $m$ be a reference Borel measure with full support, i.e., $\supp m=\XX$ $($\footnote{ Typically, one can assume that $m$ is the Lebesgue measure when $\XX$ is a Riemannian manifold.}$)$.
Recall that, given a compact $A$, the {\bf\em basin of attraction of $A$} is defined as
$$\beta_f(A)=\{x\in\XX\,;\,\omega_f(x)\subset A\}.$$

In the same way as the definition of topological attractors (see Section~\ref{SecStatOfMainsR}) or topological statistical attractors (Definition~\ref{DefTopStsAtt}), we define {\bf\em metrical attractor} (with respect to the reference measure $m$) as a compact set $A\subset\XX$ such that $m(\beta_f(A))>0$ and $m(\beta_f(A)\setminus\beta_f(A'))>0$ for every compact set $A'\subset A$.
\begin{Definition}[Regular attractors]
A metrical attractor $A$ is called {\bf\em regular} when  $\omega_f(x)=A$ for $m$ almost every $x\in\beta_f(A)$.
Likewise, a topological attractor $A$ is {\bf\em regular} when $\omega_f(x)=A$ for a residual set of points $x\in\beta_f(A)$.
\end{Definition}

Note that if $m$ is the Lebesgue measure then all regular metrical attractor is a metrical attractor in Milnor sense \cite{Mi}.
Furthermore, most of the metrical attractors in the literature are regular attractors:
\begin{enumerate}
\item the attractors of $C^2$ non-flat interval maps (including the wild attractors),
\item hyperbolic attractors for $C^2$ diffeomorphisms,
\item non uniformly expanding attractors for $C^{1+}$ maps with non degenerated critical region (including Viana's maps \cite{Vi}),
\item non-uniformly hyperbolic attractors for $C^{1+}$ maps,
\item Lorenz, Henon and Kan attractors \cite{Lo,He,Ka}.
\end{enumerate}
In particular, most the attractors supporting an SRB or, more in general, a physical measure are regular metrical attractors.
As mentioned in Section~\ref{SecStatOfMainsR}, the {\bf\em basin of attraction} of a measure $\mu\in\cm^1(\XX)$, denoted by $\beta_f(\mu)$, is the set of all $x\in\XX$ such that $\frac{1}{n}\sum_{j=0}^{n-1}\delta_{f^j(x)}$ converges to $\mu$ in the weak$^{\star}$ topology.

\begin{Definition}[Physical measures]
A probability measure $\mu\in\cm^1(\XX)$ is called a {\bf\em physical measure}, with respect to the reference measure $m$, if $m(\beta_f(\mu))>0$.
\end{Definition}

Consider the partial order $\le$ on $\KK(\XX)$ given by the inclusion, i.e., $A\le B$ when $A\subset B$.
A map $\varphi:\XX\to\KK(\XX)$ is {\bf\em upper semicontinuous} at a point $p\in\XX$ if $\limsup_{n\to \infty}\varphi(x_n)\le\varphi(p)$ for every sequence $x_n\to p$.
Similarly, $\varphi$ is {\bf\em lower semicontinuous} at $p\in\XX$ if $\liminf_{n\to\infty}\varphi(x_n)\ge\varphi(p)$ for every sequence $x_n\to p$.
It is easy to check that $\varphi$ is upper semicontinuous at $p$ if and only if for every  $\varepsilon>0$ there exist $\delta>0$ such that $\varphi(x)\subset B_{\varepsilon}(\varphi(p)):=\bigcup_{x\in \varphi(p)}B_{\varepsilon}(x)\subset\XX$ $\forall x\in B_{\delta}(p)$, where $B_r(p)$ denotes the open ball on $\XX$ (not on $\KK(\XX)$) of radius $r>0$ and center $p\in\XX$.
As the same, $\varphi$ is lower semicontinuous at $p$ if and only if for every  $\varepsilon>0$ there exist $\delta>0$ such that $\varphi(p)\subset B_{\varepsilon}(\varphi(x))$ for every $x\in B_{\delta}(p)$.

Consider the maps $\omega_f:\XX\to\KK(\XX)$, $\omega_f^{\star}:\XX\to\KK(\XX)$ and $\Agemo_f:\XX\to\KK(\cm^1(\XX))$, where $\omega_f(x)$ is the omega-limit  of $x$ (see the beginning of Section~\ref{SectionAttractors}), $\omega_f^{\star}(x)$ is the statistical omega-limit of $x$ (see the beginning of Section~\ref{SectionSTATAT}) and $\Agemo_f(x)$ is the statistical spectrum of $x$ (see Section~\ref{Subsubsecststspec}).
To analyze the points of $\XX$ where $\omega_f,\omega_f^{\star}$ and $\Agemo_f$ are semicontinuous, we need Fort's Theorem below.

\begin{theorem}[M. K. Fort, \cite{Fo}]
For any Baire topological space $X$ and compact topological space $Y$, the set of continuity points of a semicontinuous map from $X$ to $\KK(Y)$ is a Baire generic subset of $X$.
\end{theorem}

Let $d_{H}$ the Hausdorff distance on $\KK(\XX)$ with respect to the distance $d$ on $\XX$.
Let $\dd$ be a distance on $\cm^1(\XX)$ compatible with the weak* topology.
For instance, we may consider the distance given by \eqref{eqdistprob} at Section~\ref{SectionSTATAT}.
Defining $\overline{\dd}(\mu,\nu)=\dd(\mu,\nu)+d_{H}(\supp\mu,\supp\nu)$, we have that $\overline{\dd}$ is a distance on $\cm^1(\XX)$.
Let $\overline{\dd}_H$ be the Hausdorff distance on $\KK(\cm^1(\XX))$ associated to $\overline{\dd}$.

\begin{Proposition}\label{Propkjgfd671}
There exists a residual set $\cR\subset\XX$ such that $\omega_f,\omega_f^{\star}:(\XX,d)\to(\KK(\XX),d_H)$ and $\Agemo_f:(\XX,d)\to(\KK(\cm^1(\XX),\overline{\dd}_H)$ are upper semicontinuous maps at every point of $\cR$ $($\footnote{ As  the induced topology generated by $\overline{\dd}$ is stronger than the weak* topology (induced by $\dd$), the map $\Agemo_f:(\XX,d)\to(\KK^1(\cm^1(\XX)),\dd_H)$ is also upper semicontinuous at all points of $\cR$, where $\dd_H$ is the Hausdorff distance on $\KK(\cm^1(\XX))$ associated to $\dd$ and used at Section~\ref{SectionSTATAT}.}$)$.
\end{Proposition}
\begin{proof}
Let us consider the maps $\varphi:(\XX,d)\to(\KK(\XX),d_H)$ and $\psi:(\XX,d)\to(\KK(\cm^1(\XX)),\overline{\dd}_H)$ given by $\varphi(x)=\overline{\co_f^+(x)}=\overline{\{f^{n-1}(x)\,;\,n\in\NN\}}$ and $\psi(x)=\overline{\{\frac{1}{n}\sum_{j=0}^{n-1}\delta_{f^j(x)}\,;\,n\in\NN\}}$.
\begin{Claim}\label{Claimljbhiyrd76}
$\varphi$ and $\psi$ are lower semicontinuous maps. 
\end{Claim}
\begin{proof}[Proof of the claim]
Let $p\in\XX$. Since $B_{\varepsilon}(\co_f^+(p))\supset\overline{\co_f^+(p)}$ for every $\varepsilon>0$, it follows from the compactness of $\overline{\co_f^+(p)}$ that there is $\ell_{\varepsilon}\in\NN$ such that $\bigcup_{j=0}^{\ell_{\varepsilon}}B_{\varepsilon/2}(f^j(p))\supset\overline{\co_f^+(p)}$.
On the other hand, as $\XX\ni x\mapsto \{x,\cdots,f^{\ell_{\varepsilon}}(x)\}\in\KK(\XX)$ is a continuous map, one can see that there exists $\delta>0$ such that $B_{\varepsilon}(f^j(x))\supset B_{\varepsilon/2}(f^j(p))$ for every $0\le j\le\ell_{\varepsilon}$ and $x\in B_{\delta}(p)$.
Thus, $B_{\varepsilon}(\varphi(x))=B_{\varepsilon}(\overline{\co_f^+(x)})\supset \bigcup_{j=0}^{\ell_{\varepsilon}}B_{\varepsilon}(f^j(x))\supset\bigcup_{j=0}^{\ell_{\varepsilon}}B_{\varepsilon}(f^j(p))\supset\overline{\co_f^+(p)}=\varphi(p)$ for every $x\in B_{\delta}(p)$, proving the lower semi continuity of $\varphi$.
A similar argument show the lower semicontinuity of $\psi$.
Indeed, write $\mu_n(x)=\frac{1}{n}\sum_{j=0}^{n-1}\delta_{f^j(x)}$ and let, for given $\varepsilon>0$, $\ell_{\varepsilon}\in\NN$ be such that $\bigcup_{j=0}^{\ell_{\varepsilon}}B_{\varepsilon/2}(\mu_n(p))\supset\overline{\{\mu_n(p)\,;\,n\in\NN\}}$.
Taking $\delta>0$ small enough, it follows from the continuity of $(\XX,d)\ni x\mapsto \{\mu_1(x),\cdots,\mu_n(x)\}\in(\KK(\cm^1(\XX)),\overline{\dd}_H)$ that $\bigcup_{j=0}^{\ell_{\varepsilon}}B_{\varepsilon}(\mu_j(x))\supset \bigcup_{j=0}^{\ell_{\varepsilon}}B_{\varepsilon/2}(\mu_j(p))$ $\forall x\in B_{\delta}(p)$.
As for $\varphi$, this implies that $B_{\varepsilon}(\psi(x))\supset \psi(p)$ proving the lower semi continuity of $\psi$.
\end{proof}

It follows from Claim~\ref{Claimljbhiyrd76} and Fort's theorem above that there exists  residual set $\cR_{\varphi}$ and $\cR_{\psi}\subset\XX$ such that $\varphi$ is continuous at every point of $\cR_{\varphi}$, as well as, $\psi$ is continuous at the points of $\cR_{\psi}$.

Let $p\in\cR:=\cR_{\varphi}\cap\cR_{\psi}$.
Given $\varepsilon>0$ let $U_{\varepsilon}=\co_f^+(p)\setminus B_{\varepsilon}(\omega_f(p))$, recalling that $B_{\varepsilon}(\omega_f(x))=\bigcup_{x\in\omega_f(p)}B_{\varepsilon}(p)\subset\XX$.
Note that $U_{\varepsilon}$ is a finite set and choose an open set $V\subset\XX$ containing  $U_{\varepsilon}$ and such that $V\cap B_{\varepsilon}(\omega_f(p))=\emptyset$. 
As $V\cup B_{\varepsilon}(\omega_f(p))$ contains $\overline{\co_f^+(p)}$ and $\lim_{x\to p}\overline{\co_f^+(x)}=\overline{\co_f^+(p)}$, there exists $\delta>0$ such that $\overline{\co_f^+(x)}\subset V\cup B_{\varepsilon}(\omega_f(p))$ and  $\overline{\co_f^+(x)}\cap V$ is a finite set for every $x\in B_{\delta}(p)$.
This implies that $\omega_f(x)\subset B_{\varepsilon}(\omega_f(p))$ for every $x\in B_{\delta}(p)$, proving that $\omega_f$ is upper semicontinuous at $p$.
A similar argument shows that $\Agemo_f$ is upper semicontinuous at $p$.

Finally, the upper semicontinuity of $\omega_f^{\star}$ at a point $p\in\cR$ follows from the  upper semicontinuity  of $\Agemo_f(x)$.
Indeed, by the upper semicontinuity, given $\varepsilon>0$, there exists $\delta>0$ such that $B_{\varepsilon}(\Agemo_f(p))\supset \Agemo_f(x)$ for every $x\in B_{\delta}(p)$.
Hence, if $x\in B_{\delta}(p)$ and  $\mu\in\Agemo_f(x)$, there exists $\nu\in\Agemo_f(p)$ such that $\overline{\dd}(\mu,\nu)<\varepsilon/2$.
That is, $\dd(\mu,\nu)+d_H(\supp\mu,\supp\nu)<\varepsilon/2$.
Thus, by Lemma~\ref{LemmaEquationjhvyurd8unm} (applied to $X=\XX$ and $g=f$), $\supp\mu\subset B_{\varepsilon/2}(\supp\nu)\subset B_{\varepsilon/2}(\overline{\bigcup_{\eta\in\Agemo_f(p)}\supp\eta})=B_{\varepsilon/2}(\omega_f^{\star}(p))$ for every $\mu\in\Agemo_f(x)$ and $x\in B_{\delta}(p)$.
As a consequence,  $\omega_f^{\star}(x)=\overline{\bigcup_{\mu\in\Agemo_f(x)}\supp\mu}\subset\overline{B_{\varepsilon/2}(\omega_f^{\star}(p))}\subset B_{\varepsilon}(\omega_f^{\star}(p))$ for every $x\in B_{\delta}(p)$, proving the upper semicontinuity of $\omega_f^{\star}$ at every $p\in\cR$.
\end{proof}

\begin{Theorem}\label{TheoremErgMetSRB}
Suppose that $f$ is non-singular and $U$ is a $u$-Baire ergodic component of $f$.
Let $\Lambda$ and $\Lambda^{\star}$ be, respectively, the topological attractor and the topological statistical attractor of $U$.
\begin{enumerate}
\item If $A$ is a metrical attractor and $\overline{\{x\in\XX\,;\,\omega_f(x)=A\}}\cap U$ is a fat set then $A\subset\Lambda$.
\item If $\mu\in\cm^1(\XX)$ is a  physical measure and $\overline{\beta_f(\mu)}\cap U$ is a fat set then $\supp\mu\subset\Lambda^{\star}\subset\Lambda$.
\end{enumerate}
\end{Theorem}
\begin{proof}
Let $\cR$ be the residual set given by Proposition~\ref{Propkjgfd671}.
Let $A$ be a metrical attractor such that $\overline{\beta_f^+(A)}\cap U$ is a fat set, where $\beta_f^+(A)=\{x\in\XX\,;\,\omega_f(x)=A\}$.
In this case, there exists an nonempty open set $V$ such that $\beta_f^+(A)$ and $U$ are respectively dense and residual in $V$.
As $\beta_f^+(A)$ is dense in $V$, given $p\in V\cap U\cap\cR$ there exists a sequence $x_n\in\beta_f^+(A)$ such that $\lim_nx_n=p$.
Hence, it follows from the upper semicontinuity of $\omega_f$ at $p$ that  $A=\lim_{n\to\infty}\omega_f(x_n)\subset\omega_f(p)$.
That is, $A\subset\omega_f(p)$ for every $p\in V\cap U\cap\cR$.
On the other hand, by Proposition~\ref{Proposition-u-ergAttrac},  $\omega_f(x)=\Lambda$ for a residual set of points $x\in U$.
This implies that $A\subset\omega_f(x)=\Lambda$ for a residual set of points $x\in V$ and so,  $A\subset\Lambda$.

Now, let $\mu\in\cm^1(\XX)$ be a physical measure such that $\overline{\beta_f(\mu)}\cap U$ is a fat set.
In this case, let $V:=A\cap B\ne\emptyset$, where  $A=\interior(\overline{\beta_f(\mu)})$ and $B$ is any open set such that $B\sim U$. 
Thus, given $p\in V\cap U\cap\cR$ there exists a sequence $x_n\in\beta_f(\mu)$ such that $\lim_nx_n=p$.
Note that $\omega_f^{\star}(x)=\supp\mu$ for every $x\in\beta_f(\mu)$
Thus, by the upper semi continuity of $\omega_f^{\star}$ we get that $\supp\mu=\lim_n\omega_f^{\star}(x_n)\subset\omega_f^{\star}(p)$.
That is, $\supp\mu\subset\omega_f^{\star}(p)$ for every $p\in V\cap U\cap\cR$.
As, by Proposition~\ref{PropositionStatisticalAttractorsU-Baire}, $\omega_f^{\star}(x)=\Lambda^{\star}$ for a residual set of points $x\in U$ and as $\Lambda^{\star}\subset\Lambda$, we get that $\supp\mu\subset\Lambda^{\star}\subset\Lambda$.
\end{proof}

Now we can prove the last theorem of Section~\ref{SecStatOfMainsR}. It can de seen in Section~\ref{SectionCOntF} at Appendix the necessary information about continuous foliations. 

\begin{proof}[\bf Proof of Theorem~\ref{mainThojhgf}]

The statement and the proof of Claim~\ref{Calimliuytr4} below are also true for any continuous foliation.

\begin{Claim}\label{Calimliuytr4}
$W_f^s(U)$ is an open set for every open set $U\subset M$.
\end{Claim}
\begin{proof}[Proof of the claim]
Given  $p\in W_f^s(U)$, let $u\in U$ and $\varepsilon>0$ be such that $p\in W_f^s(u)$ and $B_{\varepsilon}(u)\subset U$.
By the continuity of $W_f^s$, there exists $\delta>0$ such that $W_f^s(q)\cap B_{\varepsilon}(u)\ne\emptyset$ for every $q\in B_{\delta}(p)$.
Thus,  $q\in W_f^s(W_f^s(q)\cap B_{\delta}(u))\subset W_f^s(U)$ $\forall\, q\in B_{\varepsilon}(p)$, proving that $W_f^s(U)$ is an open set.
\end{proof}

As $W_f^s$ is a continuous foliation, one can use the holonomy between local transverse sections to prove the Claim~\ref{Claimjfiopo4} below, see Lemma~\ref{LemmaCotFoliBairePart} at Appendix.

\begin{Claim}\label{Claimjfiopo4}
If $R$ is residual in an open set $U$ then $W_f^s(R)$ is residual in $W_f^s(U)$.
\end{Claim}

Recall the definition of $\mathfrak{I}^u(f)$ in Section~\ref{SectionBaireErgodicity} just below Definition~\ref{Defioihbohb}
and consider $\mathfrak{m}:\mathfrak{I}^u(f)\to[0,+\infty)$ given by $\mathfrak{m}(Y)=\leb(W_f^s(\pi(Y)))$ (see Definition~\ref{DefBaireProj}), as $\pi(Y)$ is an open set, it follows from the Claim~\ref{Calimliuytr4} that $W_f^s(\pi(Y))$ also open, in particular, measurable.
So, $\leb(W_f^s(\pi(Y)))$ is well defined.

It follows from Claim~\ref{Claimjfiopo4} above that, $U\cap V\sim\emptyset$ $\implies$ $W_f^s(\pi(U))\cap W_f^s(\pi(V))=\emptyset$ for every $U$ and $V\in\mathfrak{I}^u(f)$ and so, $\mathfrak{m}$ is a $u$-Baire $f$-function.
Furthermore, it follows from $f$ being a homeomorphism and from the claims above, that  $f^{-1}(\pi(U))=\pi(U)=W_f^s(U)$ for every $U\in\mathfrak{I}^u(f)$ and so, by the hypothesis of the theorem, $\mathfrak{m}(U)=\leb(W_f^s(\pi(U)))=\leb(W_f^s(\bigcup_{n\ge0}f^n(\pi(U))))\ge\leb(M)/\ell$ for every fat set $U\in\mathfrak{I}^u(f)$, where $\ell=\min\{n\in\NN\,;\,n\ge\leb(M)/\varepsilon\}$.

Thus, it follows from Proposition~\ref{PropositionCriteionFor-u-Ergodicity} that there exist Baire ergodic components $U_1,$ $\cdots,$ $U_k$, with $1\le k\le\ell\le\leb(M)/\varepsilon$, such that $M\sim U_1\cup\cdots\cup U_k$.
The proof of items (1) and (2) follows straightforward from Proposition~\ref{Proposition-u-ergAttrac} applied to each $U_j$.
Finally, if $\beta_f(\mu)$ is dense in an open set $V\ne\emptyset$ then $\overline{\beta_f(\mu)}\cap U_j$ is a fat set for some $1\le j\le k$.
Hence, by Theorem~\ref{TheoremErgMetSRB}, $\supp\mu\subset A_j$, concluding the proof of Theorem~\ref{mainThojhgf}.\end{proof}

\subsection{Interval maps}\label{sectionIntMap}
In \cite{Pi21} a more complete set of applications of the ergodic formalism in the study of interval maps is presented, here we give just one example (Theorem~\ref{TheoInterMap}) of such applications, since it is used in the proof of Theorem ~\ref{TheoremSkewOne}.

A $C^2$ interval map $f:[0,1]\circlearrowleft$ is called {\bf\em non-degenerated} if $f$ is non-flat and $\per(f)$ is a meager set.
Recall that $f$ is {\em non-flat} if for each $c\in\cc_f:=(f')^{-1}(0)$  there exist $\varepsilon>0$, $\alpha\ge1$ and a $C^2$  diffeomorphisms $\phi:(c-\varepsilon,c+\varepsilon)\to \text{Im}(\phi)$ such that $\phi(c)=0$ and
$f(x)=f(c)+\big(\phi(x)\big)^\alpha$ for every $x\in(c-\varepsilon,c+\varepsilon)$. 

\begin{Theorem}\label{TheoInterMap}
If a non-degenerated $C^2$ interval map does not admit periodic attractors, then $[0,1]$ can be decomposed (up to a meager set) into a finite collection $U_1,\cdots,U_{\ell}$ of Baire ergodic components ($1\le\ell\le\#\cc_f$), where each  $U_j$ is an open set having a topological attractor $A_j$ such that $\omega_f(x)=A_j$ for a residual set of points $x\in U_j$.
Moreover, each $\overline{U_j}\subset\alpha_f(c_j)$ for some $c_j\in\cc_f$.
\end{Theorem}
\begin{proof}
It has been proved by de Melo and van Strien that a $C^2$ non-flat map interval map does not admit wandering intervals, see Theorem~A in chapter IV of \cite{MvS} (a previous proof for $C^3$ maps appeared in \cite{MS89}).
As $f$ does not have periodic attractors and $\per(f)$ is a meager set, it follows from the {\em Homterval Lemma} (see Lemma~3.1 in \cite{MvS}) that $\interior(\bigcup_{n\ge0}f^n(U))\cap\cc_f\ne\emptyset$ for every open set $U\subset[0,1]$.
This implies that $[0,1]=\bigcup_{c\in\cc_f}\alpha_f(c)$. 
It is easy to see that, if $\alpha_f(c)$ is a fat set, then it is a Baire ergodic component (by Theorem~\ref{TheoremProposi0ytd6881}).
Let $\{c_1,\cdots,c_{\ell}\}\subset\cc_f$ be such that
\begin{enumerate}
\item $\alpha_f(c_j)$ is a fat set for every $1\le j\le \ell$;
\item $\interior(\alpha_f(c_j))\cap\interior(\alpha_f(c_k))=\emptyset$ for $j\ne k$;
\item $\bigcup_{j=1}^{\ell}\interior(\alpha_f(c_j))$ is dense in $[0,1]$.
\end{enumerate}
Thus, taking $U_j:=\interior(\alpha_f(c_j))$ for $1\le j\le \ell$, it follows from Proposition~\ref{Propositiontop-ergodicAttractors} that $\omega_f(x)=A_j$ for a residual set of points $x\in U_j$, where $A_j$ is the topological attractor of $U_j$.
\end{proof}

\subsection{Viana maps}
\label{SectionVmaps} Let us recall the definition of a Viana map.
For that consider the unitary circle $S^1=\RR/\ZZ$, $d\ge 16$, $\alpha>0$, $\sigma:S^1\to S^1 $ given by $\sigma(\theta)=d\,\theta$ mod $\ZZ$ and $g_{\alpha}:S^1\times\RR\to S^1\times\RR$ given by $$g_{\alpha}(\theta,x)=(\sigma(\theta),a_0+\alpha\sin(2\pi\theta)-x^2),$$
where $a_0$ is such that the point $0\in\RR$ is pre-periodic to the quadratic map $q(x):= a_0+x^2$.
In \cite{Vi}, Viana  proved that there exists  $\alpha>0$ small, a closed  interval $I\subset(-2,2)$ and   $C^3$ small neighborhood $\cn$ of $g_\alpha$ such that if $g \in\cn$ then
\begin{enumerate}
\item $g(S^1\times I)\subset S^1\times I$;
\item $\bigcap_{n\ge0} g^n(S^1\times I)$ is a forward invariant compact set with nonempty interior;
\item Lebesgue almost every point $p\in\bigcup_{n\ge0} g^n(S^1\times I)$ has all its Lyapunov exponents positive (with respect to $g$);
\item the critical set of $\cc_\phi=\{x\,;\,\det Dg(x)=0\}$ is the graph of a $C^2$ function $c_g:S^1\to\RR$ arbitrarily close to the null function. In particular, the critical set of $g$ is non-flat. 
\end{enumerate}
A {\bf\em Viana map} is a map $f:J\circlearrowleft$ given by $f:=g|_{J}$, where $g\in \cn$ and $J=\bigcap_{n\ge0}   g^n(S^1\times I)$.

\begin{Theorem}\label{TheorenmVianaMpashg}
If $f:J\circlearrowleft$ is a Viana map then the following statements are true.
\begin{enumerate}
\item Given a Borel measurable bounded function $\varphi:J\to\RR$, there exist $\gamma\in\RR$ and a residual set $\cR\subset J$ such that $$\limsup_{n\to+\infty}\frac{1}{n}\sum_{j=0}^{n-1}\varphi\circ f^{j}(x)=\gamma,\;\;\forall x\in\cR.$$
Moreover, if $\varphi$ is continuous  then $
  \gamma=\max\left\{\int\varphi d\mu\,;\,\mu\in\cm^1(f)\right\}
$.
 
\item Given a Borel set $V\subset J$, there exist $\theta\in[0,1]$ and a residual set $\cR\subset J$ such that $$\tau_x(V)=\theta,\;\;\forall x\in\cR.$$
Moreover, $\sup\left\{\mu\left(\overline{U}\right);\,\mu\in\cm^1(f)\right\}$ $\ge$ $\theta\ge
\sup\left\{\mu\left(U\right);\,\mu\in\cm^1(f)\right\}$,
where $U$ is any open set such that $V\sim U$.
\end{enumerate}
\end{Theorem}
\begin{proof}Note that $f$ is non-singular continuous map.
Indeed, since $\cc_f=(\det Df)^{-1}(0)$, the critical set of $f$, is a compact set with empty interior, we get that $f$ is a local diffeomorphism on the open and dense set $J\setminus\cc_f$, showing that $f$ is non-singular.

Since Theorem~C of \cite{AV} says that $f$ is a strongly transitive map (in particular, $f$ is transitive), the proof of Theorem~\ref{TheorenmVianaMpashg} follows from Theorem~\ref{mainTheoTrans}~and~\ref{TheoremUEorHB}.
\end{proof}

\subsection{Non-uniformly hyperbolic dynamics}\label{Sectionjgtiyt9675}
Let $M$ be a Riemannian manifold and consider a   non-flat map  $f\in C^1(M,M)$, i.e.,  $\cc:=\{x\in\XX\,;\,\det Df(x)=0\}$ is a compact meager set and  the following conditions hold for some $\beta,B>0$.
\begin{enumerate}
\item[(C.1)]
\quad $\displaystyle{(1/B)\dist(x,\cc)^{\beta }|v|\le|Df(x)v|\le B\,\dist(x,\cc)^{-\beta }}|v|$ for all $v\in T_x {M}$.
\end{enumerate}
For every $x,y\in {M}\setminus\cc$ with
$\dist(x,y)<\dist(x,\cc)/2$ we have
\begin{enumerate}
\item[(C.2)]\quad $\displaystyle{\left|\log\|Df(x)^{-1}\|-
\log\|Df(y)^{-1}\|\:\right|\le
(B/\dist(x,\cc)^{\beta })\dist(x,y)}$.
\end{enumerate}

A set $\Lambda\subset M$ has {\bf\em slow recurrence to the critical/singular region} (or  {\bf\em satisfies the slow approximation condition}) if
for each $\varepsilon>0$ there is a $\delta>0$
such that
\begin{equation}\label{EquationSlowRec}
\limsup_{n\to+\infty}
\frac{1}{n} \sum_{j=0}^{n-1}-\log \mbox{dist}_{\delta}(f^j(x),\cc)
\le\varepsilon
\end{equation}
for every $x\in\Lambda$,
where $\dist_{\delta}(x,\cc)$ denotes the $\delta$-{\bf\em truncated distance} from $x$ to $\cc$ defined as
$\dist_{\delta}(x,\cc)=\dist(x,\cc)$ if $\dist(x,\cc) \leq \delta$
and $\dist_{\delta}(x,\cc) =1$ otherwise.

We say that $\Lambda\subset M$ is a {\bf\em non-uniformly expanding} set, NUE for short,  if $\Lambda$ has slow recurrence to the crittical/singular region and 
\begin{equation}\label{EquationNUE}
   \limsup_{n\to+\infty}\frac{1}{n}\sum_{j=0}^{n-1}\log\|(Df\circ f^j(x))^{-1}\|^{-1}\ge\lambda>0
\end{equation}
for every $x\in\Lambda$.

The main property of a point $x\in M$ satisfying (\ref{EquationSlowRec}) and (\ref{EquationNUE}) is the existence of {\em hyperbolic pre-balls}. Given $0<\sigma<1$ and $\delta>0$, a {\bf\em $(\sigma,\delta)$-hyperbolic pre-ball of center $x$ and order} $n\in\NN$ is an open set $V_n(x)$ containing $x$ such that 
\begin{enumerate}
\item $f^{n}$ maps $\overline{V_n(x)}$ diffeomorphically onto the ball
$\overline{B_{\delta}(f^{n}(x))}$;
\item $dist(f^{n-j}(y),f^{n-j}(z)) \le
\sigma^{j}\dist(f^{n}(y),f^{n}(z))$ $\forall y, z\in V_n(x)$ and $1\le j<n$.
\end{enumerate}

\begin{Lemma}[Lemma~5.2 of \cite{ABV}]\label{LemmaABV}
	If $x$ satisfies (\ref{EquationNUE}) then there exists $\NN_x\subset\NN$, with $\limsup_n\frac{1}{n}\#(\{1,$ $\cdots,$ $n\}$ $\cap$ $\NN_x)>0$, $0<\sigma<1$ and $\delta>0$ ($\sigma$ and $\delta$ depending only on $\lambda$) such that for every $n\in\NN_x$ the $(\sigma,\delta)$-hyperbolic pre-ball $V_n(x)$ is well defined.
\end{Lemma}

Hence, the existence of a dense NUE set implies that $f$ is a topologically growing map and so we have the following corollary of Theorem~\ref{TheoremFatErgodicAttractors2MAIN} (or  Theorem~\ref{TheoremFatErgodicAttractors2}).

\begin{Corollary}
\label{TheoremForToNUE}
 If $f$ has a dense NUE set $\Lambda$ then there exists a finite collection of topological attractors $A_1,\cdots,A_{\ell}$ satisfying  the following properties.
\begin{enumerate}
\item $\beta_f(A_1)\cup\cdots\cup\beta_f(A_{\ell})$ contains an open and dense subset of $M$.
\item $\omega_f^{\star}(x)=\omega_f(x)=A_j$ for a residual set of points $x\in\beta_f(A_j)$ for every $1\le j\le\ell$.
\item \label{itemjhgfd1} $f|_{A_j}$ has an uncountable set of invariant ergodic probability measures $\mu$ with all its Lyapunov exponents being positive.
\item \label{itemjhgfd2} The set of expanding periodic points is dense in $A_j$.
\item The set $\ca_j=\{x\in A_j\,;\,\alpha_f(x)\supset A_j\}$ is a forward invariant set containing an open and dense subset of $A_j$ and $f|_{\ca_j}$ is strongly transitive.
\item  If $x$ is a generic point of $U_j$ and $\varphi\in C(M,\RR)$, then  
$$
  \limsup_{n\to+\infty}\frac{1}{n}\sum_{j=0}^{n-1}\varphi\circ f^j(x)\ge\sup\bigg\{\int\varphi d\mu\,;\,\mu\in\cm^1(f|_{\ca_j})\bigg\}.
$$
\end{enumerate}
Furthermore, generically, the points of $M$ have historic behavior. 
\end{Corollary}
\begin{proof}
All items, with the exception of items {\em (\ref{itemjhgfd1})} and {\em (\ref{itemjhgfd2})}, are a direct consequence of Theorem~\ref{TheoremFatErgodicAttractors2MAIN}.
Hence, we want to comment only the two exceptions.
As $M$ is a compact manifold, we get from Theorem~\ref{TheoremFatErgodicAttractors2MAIN} (or Theorem~\ref{TheoremFatErgodicAttractors2}) that $\per(f)$ are dense in $A_j$.
In the proof of Theorem~\ref{TheoremFatErgodicAttractors2}, we use  Brouwer fixed-point theorem to produce a dense set of periodic points.
Here, as we have contraction in the hyperbolic pre-ball, we can use the Banach fixed-point theorem and obtain a dense set of expanding periodic points.
Moreover, the ``horseshoe'' that appears on Claim~\ref{ClaimLocalHorseshoes} is only a topological one, here the same argument produces a uniformly expanding horseshoe $F_j:\Lambda_j\circlearrowleft$ conjugated to the shift $\sigma:\Sigma_2^+\circlearrowleft$, where $F_j$ is a $f$-induced map and $\Lambda_j\subset A_j$.
Thus, all the uncountable $f$-invariant probability measures produced by $F_j$ are expanding probability measures (i.e., with all their Lyapunov exponents positive) and their support are contained in $A_j$ (so, they are $f|_{A_j}$-invariant probability measures).
\end{proof}

Notice that the hypothesis of the existence of a NUE dense set appear with frequency in the literature.
Indeed, it is common to assume the strong hypothesis of maps having a NUE set with full Lebesgue measure.
For more information of such maps see, for instance, \cite{ABV,Pi11, Pi06, PV, LPV}. 
We can also mention the fact that all  cycle of interval for $C^{1+}$ interval maps with non-flat critical region has a dense NUE subset.

Using $u$-Baire ergodicity, one can obtain for {\em Partially Hyperbolic Systems} a result similar to Corollary~\ref{TheoremForToNUE} above. 
Given a $C^1$ diffeomorphism $f:M\circlearrowleft$, we say that $f$ is  {\bf\em partially hyperbolic} (with a strong stable direction)  if there exist a $Df$-invariant splitting $TM=\EE^{c}\oplus\EE^{s}$, $C>0$, $\lambda>1$ and $\sigma\in(0,1)$ such that the following two conditions holds:
\begin{enumerate}
\item $\|Df|_{\EE^{s}}(x)\|\|Df^{-1}|_{\EE^{c}}(x)\|\le\sigma$ for every $x\in M$ and
\item $\|Df^n|_{\EE^{s}}(x)\|\le C\lambda^{-n}$ for every $x\in  M$.
\end{enumerate}

An invariant set $\Lambda\subset M$ is a {\bf\em non-uniformly hyperbolic set (NUH)} for the partially hyperbolic diffeomorphism  $f$ if the third condition below holds:
\begin{enumerate}
	\item[(3)] $\limsup\frac{1}{n}\sum_{j=0}^{n-1}\log\|(Df|_{\EE^{c}}\circ f^j(x))^{-1}\|^{-1}\ge\lambda>0$ for every $x\in \Lambda$.
\end{enumerate}

\begin{Theorem}\label{TheoremForToPH}
 If a partially hyperbolic $C^1$ diffeomorphism $f:M\circlearrowleft$ defined on a compact Riemannian manifold has dense NUH set $\Lambda$ then there exists a finite collection of topological attractors $A_1,\cdots,A_k$ satisfying the following properties.
\begin{enumerate}
\item $\beta_f(A_1)\cup\cdots\cup\beta_f(A_k)$ contains an open and dense subset of $M$.
\item $\omega_f(x)=A_j$ for a residual set of points $x\in\beta_f(A_j)$ for every $1\le j\le k$.
\item If $\mu$ is a SRB/physical measure for $f$ then $\supp\mu\subset A_j$ for some $1\le j\le k$ or $\beta_f(\mu)$ is a nowhere dense set.
\end{enumerate}
\end{Theorem}
\begin{proof}[Sketch of the proof] As $\limsup\frac{1}{n}\sum_{j=0}^{n-1}\log\|(Df|_{\EE^{c}}\circ f^j(x))^{-1}\|^{-1}\ge\lambda>0$ for $x\in\Lambda$, we get a property similar of the hyperbolic pre-balls.
That is, following \cite{ABV} (see Lemma~2.7 and 2.10 at \cite{ABV}, see also \cite{AP}), there are  $0<\sigma<1$ and $\delta>0$ and $\NN_x\subset\NN$ ($\#\NN_x=+\infty$) such that, for any $n\in\NN_x$ there is a ``hyperbolic pre-disc'' $V_n^{c}(x)$ containing $x$, tangent to $\EE^{c}$, and satisfying 
\begin{enumerate}
\item $f^{n}$ maps $\overline{V_n^{c}(x)}$ diffeomorphically onto the center-unstable disc
$\overline{B_{\delta}^{c}(f^{n}(x))}$  (\footnote{ $B_{\delta}^{c}(p)=\bigcup_{\gamma\in\Gamma_{c}(p,\delta)}\gamma([0,1])$, where  $\Gamma_{c}(p,\delta)$ is the set of all $C^1$ curves $\gamma:[0,1]\to M$ such that $\gamma(0)=p$, $\gamma'(t)\in\EE^{c}(\gamma(t))$ for every $t\in[0,1]$ and $\int_0^1|\gamma'(t)|dt=\delta$.});
\item $dist_{c}(f^{n-j}(y),f^{n-j}(z)) \le
\sigma^{j}\dist_{c}(f^{n}(y),f^{n}(z))$ $\forall y, z\in V_n^{c}(x)$ and $1\le j<n$ (\footnote{ $\dist_{c}(p,q)=\min\{\int_0^1|\gamma'(t)|dt\,;\,\gamma\in C^1([0,1],M),\gamma'(t)\in\EE^{c}(\gamma(t)),\gamma(0)=p,\gamma(1)=q\}$.}).
\end{enumerate}
Moreover, as the angle for all $x\in M$ between the $\EE^s(x)$ and $\EE^c(x)$ is bounded from zero, there is $r>0$ not depending on $x$ such that $W_f^s(B_{\delta}^c(p))\supset B_r(p)$ for every $p\in M$. In particular, $W_f^s(f^n(V_n^c(x)))\supset B_r(f^n(x))$ always.
Therefore, it follows from the denseness of $\Lambda$ that, for every open set $U\subset M$, taking $p\in\Lambda\cap U$ and $n\in\NN_p$ big enough, we get $W_f^s(f^n(U))\supset W_f^s(V_n^c(p))\supset B_r(f^n(p))$  (see Figure~\ref{ParHy.png}).
\begin{figure}
\begin{center}\includegraphics[scale=.25]{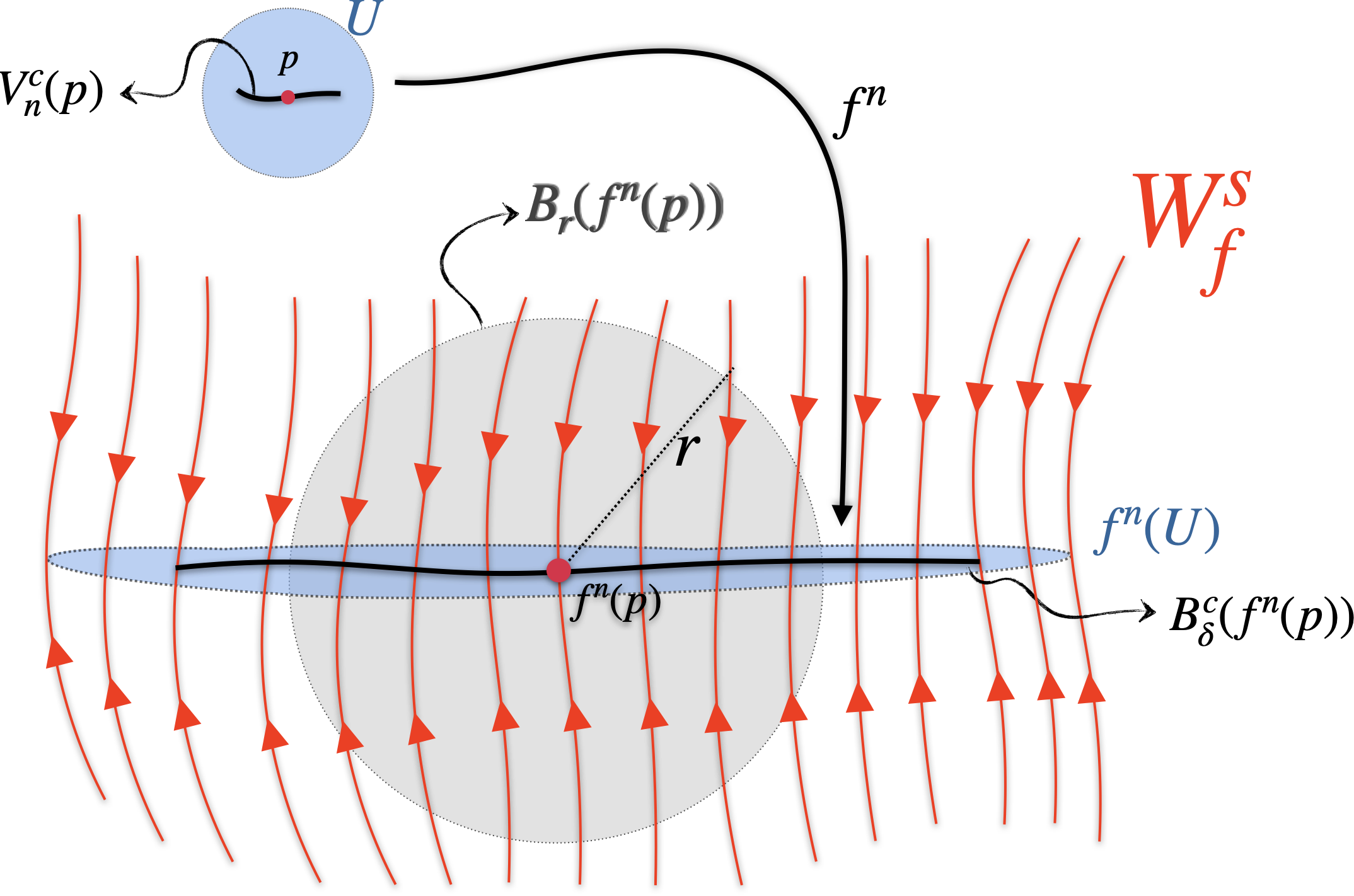}
\caption{The figure shows the ball $B_r(f^n(p))$ of radius $r>0$ contained in $W_f^s(V_n^c(p))\subset W_f^s(f^n(U))$.}\label{ParHy.png}
\end{center}
\end{figure}
 Hence, to conclude the the proof, it is enough to note that $f, W_f^s$ and $\ell$ satisfy the hypothesis of Theorem~\ref{mainThojhgf}, where $\ell=\min\big\{n\ge \frac{\leb(M)}{\min\{\leb(B_{\delta}(x))\,;\,x\in M\}}\,;\,n\in\NN\big\}$.\end{proof}

\subsection{Skew products}
In Theorem~\ref{TheoremSkewOne} below we give a simple condition to prove the existence and finiteness of topological attractors for skew products with one-dimensional fiber.

\begin{Theorem}[Skew products with one-dimensional fiber]\label{TheoremSkewOne}
Let $M$ be a compact Riemannian manifold and $F:M\times[0,1]\circlearrowleft$ be a $C^2$ map with $(\det DF)^{-1}(0)$ being meager and given by $F(x,t):=(g(x),f(x,t))$, where 
\begin{enumerate}
\item $g:M\circlearrowleft$ is a strongly transitive local homeomorphism and 
\item $\#(f_x')^{-1}(0)<\infty$  for every $x\in M$, with $f_x:[0,1]\circlearrowleft$ given by $f_x(t)=f(x,t)$.
\end{enumerate}
Suppose that there exists $p\in\per(g)$, with period $n\in\NN$, such that $f_p{^n}$ is non-degenerated interval map (see Section~\ref{sectionIntMap}).
If $f_p{^n}$ does not have periodic attractors, then there exists a finite collection of topological attractors $A_1,\cdots,A_{\ell}\subset M\times[0,1]$,  such that
$$\beta_F(A_1)\cup\cdots\cup\beta_F(A_{\ell})\sim M\times[0,1].$$
Moreover $\omega_F(x)=A_j$ for a residual set of points $x\in\beta_F(A_j)$, 
 $\ell$ $\le$ $\#\{c\in[0,1]\,;\,(f_p^n)'(c)=0\}$ and $F$ has historic behavior generically on $M\times[0,1]$ when  $g$ has historic behavior generically on $M$.
\end{Theorem}
\begin{proof} Note that $F$ is a non-singular map, since $\cc=(\det DF)^{-1}(0)$ is meager and  $F|_{(M\times[0,1])\setminus\cc}$ is a local diffeomorphism.

As $T:=f_p{^n}$ is a non-degenerated $C^2$ interval map, it follows from Theorem~\ref{TheoInterMap} that there exists $\{c_1,\cdots,c_{\ell}\}\subset\cc_{_T}:=(T')^{-1}(0)$ such that $U_j^0:=\interior(\alpha_{_T}(c_j))$ is a Baire ergodic component to $T$ and $\bigcup_{j=0}^{\ell}\interior(\alpha_{_T}(c_j))$ is dense in $[0,1]$.

As $\alpha_g(p)={M}$, given an open set $V\subset{M}\times[0,1]$,  there is $n\ge0$ such that $F^n(V)\cap (\{p\}\times[0,1])\ne\emptyset$.
Since $g$ is a local homeomorphism and $f_x$ piecewise monotone interval map, we get that $F^n(V)\supset B_{\varepsilon}(p)\times(a,b)$ for some $\varepsilon>0$ and $0<a<b<1$, where $B_{\varepsilon}(p)\subset{M}$ is the open ball on ${M}$ of radius $\varepsilon$ and center $p$.
Thus, it follows from Theorem~\ref{TheoInterMap} that $(a,b)\cap U_j^0\ne\emptyset$ for some $0\le j\le \ell$ and so, $T^k((a,b))\ni c_j$ for some $k\ge0$.
This implies that there exists $m\ge0$ such that $F^m(V)$ contains an open neighborhood of $(p,c_j)$. Therefore, one can find $\{c_1',\cdots,c_s'\}\subset\{c_1,\cdots,c_{\ell}\}$ such that 
\begin{itemize}
\item $\alpha_F((p,c_j'))$ is a fat set  for every $1\le j\le s$;
\item $\alpha_F((p,c_j'))\cap\alpha_F((p,c_k'))\sim\emptyset$ for $j\ne k$;
\item $\bigcup_{j=1}^{s}\alpha_F((p,c_j'))\sim{M}\times [0,1]$.
\end{itemize}

Taking $U_j=\alpha_F((p,c_j'))$, we have that $F(U_j)\subset U_j$ and $\bigcup_{n\ge0}F^n(V)$ contains an open neighborhood of $(p,c_j)$ for every nonempty open set $V\subset{M}\times[0,1]$ such that $V\cap U_j\ne\emptyset$.
This implies that $F|_{U_j}$ is asymptotically transitive.
Since, $F$ is a non-singular continuous map, it follows from Theorem~\ref{TheoremProposi0ytd6881} that $F|_{U_j}$ is Baire ergodic.

As $F(U_j)\subset U_j$ $\forall j$, $\bigcup_{j=1}^{s}\alpha_F((p,c_j'))\sim{M}\times [0,1]$ and $\alpha_F((p,c_j'))\cap\alpha_F((p,c_k'))\sim\emptyset$ for $j\ne k$, we get that $U_j\sim F^{-1}(U_j)$.
Therefore, $U_1,\cdots,U_s$ are Baire ergodic decomposition for $F$.
As $U_1\cup\cdots\cup U_s\sim{M}\times[0,1]$, the topological attractors $A_j$ for $U_j$ (given by Proposition~\ref{Propositiontop-ergodicAttractors} applied to $F|_{U_j}$), are such that $\omega_F(x)=A_j$ for a residual set of points $x\in U_j$.
Hence, $\beta_f(A_j)\supset U_j$ contains an open set (by the continuity of $F$), $\beta_f(A_1)\cup\cdots\cup\beta_f(A_s)\sim{M}\times[0,1]$ and $\omega_F(x)=A$ for a residual set of points $x\in\beta_f(A_j)$.

To conclude the proof is enough to use  the fact that the existence of a residual subset $R$ of ${M}$ with all $x\in R$ having historical behavior for $g$ implies that all points of the residual set  $R\times[0,1]$ have historic behavior for $F$.
\end{proof}

A concrete application of Theorem~\ref{TheoremSkewOne} can be the following.
Take an initial logistic map  $f_{t_0}(x)=4t_0x(1-x)$, $f_{t_0}:[0,1]\to[0,1]$, where $t_0\in(0,1]$ is such that $f_{t_0}$ does not have a periodic attractor.
For instance, $f_{t_0}$ may be infinitely renormalizable (Exemple~\ref{Unimodal}), a Misiurewicz map or a map with an absolutely continuous invariant measure.
Consider $g:\TT^2\circlearrowleft$ a $C^1$ the uniformly expanding map on the torus $\TT^2$ induced by the linear map $L(x,y)=(3x+y,x+2y)$, $\varphi:\TT^2\to(0,1)$ a continuous map such that $\varphi([(0,0)])=t_0$ and $F:\TT^2\times[0,1]\circlearrowleft$ given by $F(p,x)=(g(p),4\varphi(p)x(1-x))$.
In this case, $F$ has a single topological attractor $A$, $\beta_F(A)\sim\TT^2\times[0,1]$, $\omega_F(p,x)=A$ and $(p,x)\in\TT^2\times[0,1]$ has historic behavior for a residual set of points $(p,x)\in\TT^2\times[0,1]$.

In Theorem~\ref{TheoremSkewMult} below, we present  a version of  Theorem~\ref{TheoremSkewOne} for skew products with multidimensional fiber.

\begin{Theorem}[Skew products with multidimensional fiber]\label{TheoremSkewMult}
Let $\XX$ and $\YY$ be compact metric spaces and $F:\XX\times\YY\circlearrowleft$ a continuous map such that $\#F^{-1}(p)<+\infty$ $\forall p\in\XX\times\YY$.
Suppose that $F(x,y)=(g(x),f(x,y))$, where $g:\XX\circlearrowleft$ be is strongly transitive local homeomorphism.
Suppose that $f_x:\YY\circlearrowleft$ a non-singular map for every $x\in\XX$, where $f_x(y)=f(x,y)$.

If there is $p\in\per(g)$  such that  ${f_p}^n$ is a growing map, where $n$ is the period of $p$, then there exists a finite collection of topological attractors $A_1,\cdots,A_{\ell}$ such that
\begin{enumerate}
	\item $\beta_F(A_1)\cup\cdots\cup\beta_F(A_{\ell})$ contains a residual  subset of $\XX\times\YY$ and,
\item for every $1\le j\le\ell$, $\omega_F(x)=A_j$ for a residual set of points $x\in\beta_F(A_j)$.
\end{enumerate}
Moreover, $F$ has historic behavior generically on $\XX\times\YY$.
\end{Theorem}
\begin{proof}[Sketch of the proof]
Using that $g$ is a local homeomorphism and $f_x$ is non-singular for every $x$, we get that $F$ is non-singular.

Let $h={f_p}^n$.
As we are assuming that $h$ is a growing map, it follows from Theorem~\ref{TheoremFatErgodicAttractors2MAIN} that there is a finite collection of topological attractors $A_{1,h},\cdots,A_{\ell,h}$ such that $\beta_h(A_{1,h})\cup\cdots\cup\beta_h(A_{\ell,h})$ contains an open and dense subset of $\YY$.
Moreover, each $A_{j,h}$ contains a $h$ forward invariant set $\ca_{j,h}$ such that $\alpha_h(p)\supset\overline{U_{j,h}}$, where $U_{j,h}$ is the Baire ergodic component associated to $A_{j,h}$.
As $\YY=\overline{U_{1,h}}\cup\cdots\cup\overline{U_{\ell,h}}$, choosing any collection of points $c_1,\cdots,c_{\ell}$, with $c_j\in\ca_{j,h}$, we get that $V_j:=\overline{\co_F^-(p,c_j)}$ is an $F$ invariant set with $\interior(V_j)\supset\{p\}\times\ca_{j,h}$.
Moreover, it is not difficult to see that $F|_{V_j}$ is asymptotically transitive and so, by Theorem~\ref{TheoremProposi0ytd6881}, $F|_{V_j}$ is Baire ergodic, proving that $V_j$ is a Baire ergodic components of $F$.
Thus, items (1) and (2) follows from Proposition~\ref{Propositiontop-ergodicAttractors}.

Note that, by Theorem~\ref{TheoremFatErgodicAttractors2MAIN}, $\#\Agemo_h(x)\ge2$ for a residual set of $x\in\beta_h(A_{j,h})$, since a generic point of $\beta_h(A_{j,h})$ has historical behavior.
Thus, since $\ca_{j,h}\setminus\beta_h(A_{j,h})\sim\emptyset$, we can choose the points $c_j$ so that $\Agemo_h(c_j)\ge2$ for every $1\le j\le \ell$.
Using this, we can follow the argument on the proof of Proposition~\ref{PropositionMataLeao} to show that $\Agemo_F(x)\supset\Agemo_F((p,c_j))$ for a residual set of points $x\in\beta_F(A_j)$, where $A_j$ is given by item (1).

As $\#\Agemo_h(c_j)\ge2$ and $\Agemo_h(c_j)\subset\Agemo_{F^n}((p,c_j))$, we also have $\#\Agemo_{F}((p,c_j))\ge2$ and so, $\Agemo_F(x)\ge2$ residual set of points $x\in\beta_F(A_j)$, showing that a generic point $x\in\beta_F(A_j)$ has historical behavior $\forall j$. 
Thus, generically, the points of $\XX\times\YY$ have historical behavior, since $\beta_F(A_1)\cup\cdots\cup \beta_F(A_\ell)\sim\XX\times\YY$.  
\end{proof}

\section{Appendix}

\subsection{Continuous non-singular maps}

\begin{Lemma}\label{Lemma1passo}
Let $\YY$ and $\XX$ be Baire metric spaces and $h:\YY\to\XX$ be a non-singular, continuous and surjective map.
If $A$ is a residual subset of $\YY$ then $h(A)$ is a residual subset of $\XX$.
\end{Lemma}
\begin{proof}
If $h(A)$ is not residual in $\XX$ then there exists a nonempty open set $V\subset\XX$ such that $h(A)\cap V$ is a meager set.

Taking $\XX'=V$, $\YY'=h^{-1}(\XX')$, $A'=A\cap\YY'$ and $g=h|_{\YY'}$, we get that $\YY'$ is an nonempty open subset of $\YY$ (by the continuity of $h$), $A'$ is a residual subset of $\YY'$ and $g:\YY'\to\XX'$ is a non-singular, continuous and surjective map.
But this is impossible, since $g$ is non-singular, $h(A)\cap V$ is a meager set and $A'=g^{-1}(g(A'))=g^{-1}(h(A)\cap V)$ is a residual set.
\end{proof}

\begin{Lemma}\label{LemmaDoInterior}
Let $\YY$ and $\XX$ be Baire metric spaces and $h:\YY\to\XX$ a continuous map.
If $h$ is non-singular then $h(A)\subset\overline{\interior(h(A))}$ for every open set $A\subset\YY$.
\end{Lemma}
\begin{proof}
Let $a\in A$, where $A\subset\YY$ is an open set, and consider the following claim.
\begin{claim}
$h\big(\,\overline{B_{\delta}(a)}\,\big)\subset\overline{\interior\big(h\big(\,\overline{B_{\delta}(a)}\,\big)\big)}$ for every $\delta>0$.
\end{claim}
\begin{proof}[Proof of the claim]
Otherwise, for some $\delta>0$, there exist $p\in \overline{B_{\delta}(a)}$ and $\varepsilon>0$ such that $$B_{\varepsilon}(h(p))\cap h(\overline{B_{\delta}(a)})\subset\partial\big(h(\overline{B_{\delta}(a)})\big).$$
This implies that  $B_{\varepsilon}(h(p))\cap h(\overline{B_{\delta}(a)})$ is a meager set and, as $h$ is non-singular, we get also that $h^{-1}\big(B_{\varepsilon}(h(p))\cap h(\overline{B_{\delta}(a)})\big)$ is a meager set.
By continuity, $h^{-1}(B_{\varepsilon}(h(p)))$ is an open set containing the point $p\in\overline{B_{\delta}(a)}$ and so, $h^{-1}(B_{\varepsilon}(h(p)))\cap B_{\delta}(a)$ is a nonempty open set.
As $h^{-1}(B_{\varepsilon}(h(p)))\cap B_{\delta}(a)$ is 
contained in the meager set $h^{-1}\big(B_{\varepsilon}(h(p))\cap h(\overline{B_{\delta}(a)})\big)$, we get a   contradiction.
\end{proof}

Taking $r>0$ such that $B_r(a)\subset A$, it follows form the claim above that $$h(a)\in h(B_r(a))=h\bigg(\bigcup_{0<\delta<r}\overline{B_{\delta}(a)}\bigg)=\bigcup_{0<\delta<r}h(\overline{B_{\delta}(a)})\subset$$
$$\subset\bigcup_{0<\delta<r}\overline{\interior\big(h\big(\,\overline{B_{\delta}(a)}\,\big)\big)}\subset\overline{\interior\big(h\big(B_r(a)\big)\big)}\subset\overline{\interior(h(A))}.$$
\end{proof}

\begin{Corollary}
\label{CorLemResnaoSingsing}
Let $\YY$ and $\XX$ be Baire metric spaces and $h:\YY\to\XX$ be a non-singular continuous map.
If $A$ is a residual subset of $\YY$ then $h(A)$ is a residual subset of $h(\YY)$.
\end{Corollary}
\begin{proof}
It follows from Lemma~\ref{LemmaDoInterior} that $h(\YY)\sim\interior(h(\YY))\ne\emptyset$.
Take $\XX'=\interior(h(\YY))$, $\YY'=h^{-1}(\XX')$, $A'=A\cap\YY'$ and $h'=h|_{\YY'}$.
Since $\XX'\subset h(\YY)$, we get that $h'$ is a non-singular, continuous and surjective map.
Therefore, it follows from Lemma~\ref{Lemma1passo} that $h'(A')$ is a residual subset of $\XX'$, since $A'$ is a residual subset of $\YY'$. 
Hence, $h(\YY)\sim\interior(h(\YY))=\XX'\sim h'(A')\subset h(A)\subset h(\YY)$, showing that $h(A)$ contains a residual subset of $h(\YY)$ and so, it is a residual subset of $h(\YY)$.
\end{proof}

Proposition~\ref{PropositionRESUMO} below summarizes the previous results of the Appendix for the maps that appear in Section~\ref{SectionAttractors}.
\begin{Proposition}
\label{PropositionRESUMO}
Let $\XX$ be compact metric space, $\XX_0$ an open and dense subset of $\XX$ and $f:\XX_0\to\XX$ a continuous non-singular map.
\begin{enumerate}
\item If $A$ is an open subset of $\XX$ then $f^*(A)\subset\overline{\interior(f^*(A))}$.
\item If $U\subset\XX$ is a nonempty open set and $A\subset U$ is a Borel set that is residual in $U$, then $f^*(A)$ is a residual subset of $f^*(U)$.
\end{enumerate}
\end{Proposition}
\begin{proof}
First, let $\YY=\XX_0$ and $h=f$.
If $A$ is an open subset of $\XX$, then $A\cap\XX_0$ is an open set of $\YY$ and, by Lemma~\ref{LemmaDoInterior}, $f^*(A)=f(A\cap\XX_0)$ $=$ $h(A)$ $\subset$ $\overline{\interior h(A)}$ $\subset$ $\overline{\interior(f(A\cap\XX_0))}$ $=$ $\overline{\interior(f^*(A))}$, proving item (1).

To prove item (2), take $\YY=U\cap\XX_0$, $A'=A\cap\YY$ and $h=f|_{\YY}$.
Since $A'$ is a residual subset of $\YY$ and $h$ is a non-singular continuous map, it follows from Corollary~\ref{CorLemResnaoSingsing} that $f^*(A)=f(A\cap\XX_0)=f(A\cap U\cap\XX_0)=h(A')\sim h(\YY)=f(U\cap\XX_0)=f^*(U)$.
\end{proof}

\begin{Lemma}\label{Lemmakufyuvooyb}
Let $\YY$ be a separable Baire  metric space, $\XX$ a Baire  metric space and $h:\YY\to\XX$ a continuous non-singular map.
There exists a residual set $\YY'\subset\YY$ such that 
if $p\in\YY'$ then $h(p)\in\interior(h(V))$ for every open set $V\subset\YY$ containing $p$.
\end{Lemma}
\begin{proof}
Let $\cc$ be the set of all points $p\in\YY$ such that $h(p)\notin\interior(h(V))$ for some neighborhood $V$ of $p$.
Let $\cq$ be a dense and countable subset of $\YY$ and, for each $n\in\NN$, define
$\cu(n)$ as the set of all points $p\in\YY$ such that $h(p)\notin\interior(h(B_{1/n}(q)))$ for some $q\in\cq$.

Note that if $V$ is an open neighborhood of a point $p\in\YY$ and $h(p)\notin\interior(h(V))$ then $h(p)\notin\interior(h(U))$ for every open set $U\subset V$ containing $p$.
Indeed, if $p\in\interior h(U)$ then $p\in\interior h(U)\subset \interior h(V)$, a contradiction.
Hence, we can conclude that 
$$\cu(1)\subset\cu(2)\subset\cu(3)\subset\cdots\subset \bigcup_{n\in\NN}\cu(n)=\cc.$$

As the boundary of an open set is a meager set and as $h$ is non-singular, we get that $h^{-1}(\partial(\interior(h(B_{1/n}(q))))$ is a meager set for every $n\in\NN$ and $q\in\cq$.
Thus,
$$\widetilde{\cc}:=\bigcup_{q\in\cq,\,n\in\NN}h^{-1}(\partial(\interior(h(B_{1/n}(q))))$$
is a meager set.
By Lemma~\ref{LemmaDoInterior}, if $p\in\cu(n)$ then  $h(p)\in\partial(\interior(h(B_{1/n}(q)))$ for some $q\in\cq$ and so, $\cu(n)\subset\widetilde{\cc}$ for every $n\in\NN$, proving that $\cu(n)$ is a meagre set. 
As this implies that $\YY':=\YY\setminus\cc$ is a residual set, we conclude the proof.
\end{proof}

\subsection{Continuous foliations}\label{SectionCOntF}
In this section, let $M$ be a $C^k$ manifold, $k\ge0$.
A partition $\cf$ of a $M$ is called a ($\ell$-dimensional) continuous foliation when every element of $\cf$ is a path-connected set and there exists a collection $\ca$ of maps $\varphi:(0,1)^{\ell}\times(0,1)^{\dim(M)-\ell}\to M$ satisfying the following conditions.
\begin{enumerate}
\item $\im(\varphi)$ is an open subset of $M$;
\item $\varphi$ is a homeomorphism between $(0,1)^{\ell}\times(0,1)^{\dim(M)-{\ell}}$ and $\im(\varphi)$;
\item if $\psi\in\ca$ and $\im(\varphi)\cap\im(\psi)\ne\emptyset$ then $\psi^{-1}\circ\varphi:\varphi^{-1}(\im(\psi))\to\psi^{-1}(\im\varphi)$ can be written as $$\psi^{-1}\circ\varphi(x,y)=(h_{\varphi,\psi}(x,y),g_{\varphi,\psi}(y))\in(0,1)^{\ell}\times(0,1)^{\dim(M)-{\ell}},$$
where $x\in(0,1)^{\ell}$ and $y\in(0,1)^{\dim(M)-{\ell}}$;
\item $\bigcup_{\psi\in\ca}\im(\psi)=M$.
\end{enumerate}
The collection of maps $\ca$ above is called a {\bf\em $\cf$-atlas}.
For Lemma~\ref{LemmaCotFoliBairePart} below, suppose that $\cf$ is a continuous foliation of $M$ with a $\cf$-atlas $\ca$.

\begin{Lemma}\label{LemmaCotFoliBairePart}
 If $R$ is residual in an open set $U$ then $\cf(R)$ is residual in $\cf(U)$.
\end{Lemma}
\begin{proof}[Sketch of the proof]
Given $x\in M$, denote the element of $\cf$ containing $x$ by $\cf(x)$.
If $V\subset M$, define $\cf(V)=\bigcup_{x\in V}\cf(x)$.
Let $\pi_2:\RR^{\ell}\times\RR^{\dim M-{\ell}}\to \RR^{\dim M-{\ell}}$ be the projection on the second coordinate, i.e., $\pi_2(x,y)=y$.
Given $\varphi\in\ca$ and $x\in\im(\varphi)$, let $\cf_{\varphi}(x)=\varphi(\RR^{\ell}\times\{\pi_2\circ \varphi^{-1}(x)\})$.
One can show that $\cf_{\varphi}(x)$ is the connected component of $\cf(x)\cap\im(\varphi)$ containing $x$.

A map $\gamma:V\to M$ is a {\bf\em $\varphi$-transverse section for $\cf$} when $V$ is an open subset of $(0,1)^{\dim(M)-\ell}$ and $\gamma(v)=\varphi(h_{\gamma}(v),v)$ for some continuous map $h_{\gamma}:V\to(0,1)^{\ell}$. 
It follows from the definition that a $\varphi$-transverse section $\gamma$ is a continuous injective map, $\im(\gamma)\subset\im(\psi)$ and $\#(\im(\gamma)\cap\cf_{\psi}(x))\le1$ for every $x\in\im(\psi)$.
A {\bf\em local transverse section for $\cf$} is a $\psi$-transverse section for some $\psi\in\ca$.

If $\varphi\in\ca$ and $\gamma$ is a $\varphi$-transverse section then  $V_{\varphi}(\gamma):=\bigcup_{x\in\im(\gamma)}\cf_{\varphi}(x)$ is an open neighborhood of $\im(\gamma)$ contained in $\im(\varphi)$.
Define $\pi_{\varphi,\gamma}:V_{\varphi}(\gamma)\to\im(\gamma)$ by $\pi_{\varphi,\gamma}(x)=\im(\gamma)\cap\cf_{\varphi}(x)$ and note that
\begin{center}
$\pi_{\varphi,\gamma}$ is a continuous non-singular map.
\end{center}
Indeed, up to the homeomorphism $\varphi$, $\pi_{\varphi,\gamma}$ is the projection of an open subset of $\RR^{\ell}\times V\subset(0,1)^{\ell}\times\RR^{\dim(M)-\ell}$ onto the graph of $h_{\gamma}$, i.e., $\pi_{\varphi,\gamma}=\varphi\circ\pi_{h_{\gamma}}\circ\varphi^{-1}(x))$, where $\pi_{h_{\gamma}}(u,v)=(h_{\gamma}(v),v)$ is a projection onto the image of the graph of a continuous map, which is non-singular map.

If $\gamma_1$ and $\gamma_2$ are $\varphi$-transverse sections and $\gamma_1(v_2)\in\cf_{\varphi}(\gamma_a(v_1))$ then there exist an open neighborhood  $A$ of $\gamma_1(v_1)$ and an open neighborhood  $B$ of $\gamma_2(v_a)$ such that  the map $$h:A\cap\im(\gamma_1)\to B\cap\im(\gamma_2)$$ 
defined by $$h(x)=\pi_{\varphi,\gamma_2}(x)$$ is a homeomorphism such that $\cf(h(x))=\cf(x)$ for every $x$ in the domain of $h$.
This map is a local holonomy. 
Using standard arguments given by the assumption that every  $\cf(x)$ is path-connected, one can show that the following result.
\begin{Claim}
\label{Claimhvguf}
If $q\in\cf(p)$ and  $\gamma_1,\gamma_2$ are local transverse section for $\cf$ such that $p\in\im(\gamma_1)$ and $q\in\im(\gamma_2)$ then exist an open neighborhood $A$ of $p$, an open neighborhood  $B$ of $q$ and a homeomorphism $h:A\cap\im(\gamma_1)\to B\cap\im(\gamma_2)$ such that $\cf(h(x))=\cf(x)$ $\forall x\in A\cap\im(\gamma_1)$.
\end{Claim}
Finally, over the hypothesis of Claim~\ref{Claimhvguf}, it follows from $\pi_{\varphi_1,\gamma_1}$ and  $\pi_{\varphi_2,\gamma_2}$ being non-singular continuous maps $($\footnote{ We are assuming that $\gamma_j$ is a $\varphi_j$-transverse section, $i=1,2$, for $\varphi_1,\varphi_2\in\ca$.}$)$ that,  
if $U\subset A$ is an open set an $R$ is residual in $U$, then $ V:=(\pi_{\varphi_2,\gamma_2})^{-1}(h(\pi_{\varphi_1,\gamma_1}(U)))$ is an open subset of $M$ containing $q$ and $\cf_{\varphi_2}(h(\pi_{\varphi_1,\gamma_1}(R)))=(\pi_{\varphi_2,\gamma_2})^{-1}(h(\pi_{\varphi_1,\gamma_1}(R)))$ is residual in $V$.
Therefore, since $\cf(R)\supset \cf_{\varphi_2}(h(\pi_{\varphi_1,\gamma_1}(R)))$, we get that $\cf(R)\cap V$ is residual in $\cf(U)\cap V=V$.

Given an open set $U\subset M$ and $R\subset U$ residual in $U$, we can use the argument above to every $p\in U$ and $q\in\cf(p)$, showing that $\cf(R)$ is a residual in $\cf(U)$.
\end{proof}

\end{document}